\documentclass[11pt]{article}
\usepackage{graphicx, color, xcolor}

\usepackage{hyperref}

\usepackage{amsmath,amsfonts,amssymb, amsthm}

\usepackage{lmodern} 

\def\forall{\hbox{for all}~}
\def\L{\mathbf{L}}

\def\ve{\varepsilon}

\def\Z{{\mathbb Z}}

\def\R{{\mathbb R}}

\def\implies{\Longrightarrow}
\def\vp{\varphi}

\def\E{{\cal E}}

\def\TV{\hbox{Tot.Var.}}
\def\vs{\vskip 2em}

\def\v{\vskip 1em}
\def\O{{\cal O}}
\def\M{{\cal M}}
\def\C{{\cal C}}
\def\J{{\cal J}}

\def\ov{\overline}
\def\Tilde{\widetilde}
\def\Hat{\widehat}
\def\meas{\hbox{meas}}

\def\dint{\int\!\!\int}
\def\sign{{\rm sign}}
\def\bega{\begin{array}}
\def\enda{\end{array}}
\def\begi{\begin{itemize}}
\def\endi{\end{itemize}}

\def\ds{\displaystyle}
\def\bel{\begin{equation}\label}
\def\eeq{\end{equation}}
\def\sqr#1#2{\vbox{\hrule height .#2pt
\hbox{\vrule width .#2pt height #1pt \kern #1pt
\vrule width .#2pt}\hrule height .#2pt }}

\newtheorem{theorem}{Theorem}[section]
\newtheorem{corollary}{Corollary}[section]
\newtheorem{lemma}{Lemma}[section]
\newtheorem{example}{Example}[section]
\newtheorem{proposition}{Proposition}[section]
\newtheorem{remark}{Remark}[section]
\newtheorem{definition}{Definition}[section]

\begin{document}

\title{\bf Conservation Laws with Discontinuous Gradient-Dependent Flux: 
the Stable Case}
\vs

\author{Debora Amadori$^*$,  Alberto Bressan$^{**}$
and Wen Shen$^{**}$
\\
\, \\
{\small  *  Dipartimento di Matematica Pura e Applicata, Universit\`a degli Studi dell'Aquila, }
\\  {\small  Via Vetoio,
67010 Coppito, Italy.}\,\\ 
{\small  ** Mathematics Department, Penn State University,
University Park, PA 16802, U.S.A. } \,\\ 
\,\\
{\small E-mails:  debora.amadori@univaq.it, ~axb62@psu.edu,~
wxs27@psu.edu}
\,\\   }
\maketitle  

\begin{abstract}
The paper is concerned with a scalar conservation law with discontinuous gradient-dependent flux.
Namely, the flux is described by two different functions $f(u)$ or $g(u)$, 
when the gradient $u_x$ of the solution is positive or negative, respectively.
We study here the stable case where $f(u)<g(u)$ for all $u\in\R$, with $f,g$ smooth but 
possibly not convex.  A front tracking algorithm is introduced, proving that piecewise constant
approximations converge to the trajectories of a contractive semigroup on $\L^1(\R)$. 
In the spatially periodic case, we prove 
that semigroup trajectories coincide with the unique limits of a suitable class of
vanishing viscosity approximations.
\end{abstract}

\section{Introduction}
\label{sec:1}
\setcounter{equation}{0}

We consider a scalar conservation law where the flux depends on the gradient of the solution 
in a discontinuous way: 
\bel{1}
u_t + \Big[\bar \theta(u_x) f(u) + \bigl(1-\bar \theta(u_x) \bigr)g(u)\Big]_x~=~0,\eeq
where $\bar\theta$ is the step function
\bel{2}\bar \theta(s)~=~\left\{\bega{rl} 1\quad &\hbox{if}\quad s>0,\cr
0\quad &\hbox{if}\quad s<0.\enda\right.\eeq
Here $f$ and $g$ are two smooth flux functions, not necessarily convex.
The flux in \eqref{1} switches between $f$ and  $g$ when $u_x$ changes sign.

The equations \eqref{1}-\eqref{2}) can be used as a model for traffic flow, where  $u$ 
is the density of cars. 
In the region where the  density decreases, i.e.~$u_x(t,x) <0$, 
the cars accelerate, 
while if the density increases, the cars 
decelerate.   As a modeling assumption one may expect that, when
the drivers are in accelerating or decelerating mode (with their foot on the gas or on the brake pedal), their behavior is markedly different. This can lead to two different flux functions: $f$ and $g$, 
for deceleration and acceleration, respectively.
The above equations capture this feature of  traffic flow,
assuming that the switching between the acceleration and deceleration mode happens instantly. 
If $\bar\theta$ is replaced by a smooth function, such transition would be gradual. 

In principle, there can be other models where the flux has
discontinuous dependence on the gradient of the conserved quantity. 
In this paper we thus consider \eqref{1} in a general setting, not necessarily restricted
to traffic flow.

To construct a suitable approximation of \eqref{1}-\eqref{2}, we consider a  smooth, nondecreasing function 
$\theta:\R\mapsto [0,1]$ with 
\bel{tprop1} 
\theta(s)~=~\left\{\bega{cl} 1\quad &\hbox{if}\quad s\geq 1,\cr
0\quad &\hbox{if}\quad s\leq -1,\enda\right.
\qquad\qquad 
\theta(0)\,=\,\frac{1}{ 2}\,,
\eeq
\bel{tprop2}
\left\{ \begin{array}{cl}
\theta'(s)=\theta'(-s)\geq 0\quad &\hbox{for all}~s\in \R, \\[1mm]
\theta'(s)\geq {1/ 2} \quad
&\hbox{for}~~|s|\leq  1/2\,, \\[1mm]
\theta'(s)>0\quad &\hbox{for}~~|s|<1  .
\end{array} \right.
\eeq
For $\ve>0$ small, we then set
\bel{tdef}
\theta_\ve(s)~\doteq~\theta(s/\ve).\eeq
Replacing $\bar\theta$ with $\theta_\ve$ in (\ref{1}) and adding a further viscosity, 
one obtains 
\bel{3} u_t +  \Big[\theta_\ve(u_x) f'(u) + \bigl(1-\theta_\ve(u_x) \bigr)g'(u)\Big] u_x
~=~\theta'_\ve(u_x) \bigl[ g(u) - f(u)\bigr] u_{xx} +\delta\, u_{xx}\,.
\eeq
For any $\ve,\delta>0$, since 
$\theta_\ve$ is increasing, in the case where
$f <g$ the equation \eqref{3}  is parabolic and
hence the Cauchy problem is  well posed.   
We refer to this as the \textbf{stable case}. 

Notice that, when $f < g$, the conservation law (\ref{1}) produces a sink 
at every point of local maximum, and a source at every point of local minimum 
(see Fig.~\ref{f:df4}). In the parabolic approximation (\ref{3}),
due to the large viscosity $\theta'_\ve \approx  \ve^{-1}$ present when $u_x \approx 0$, 
our analysis will show that every local max or min is immediately spread out over an entire interval $\bigl[a(t), b(t)\bigr]$ 
where the solution is nearly constant.

\begin{figure}[htbp]
\begin{center}
\resizebox{.65\textwidth}{!}{
\begin{picture}(0,0)%
\includegraphics{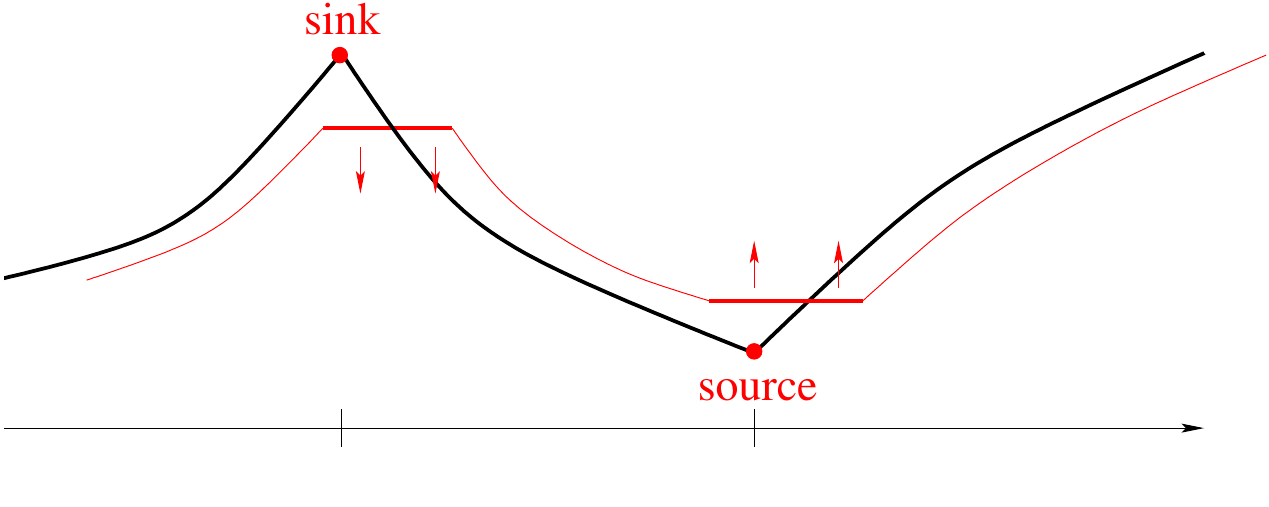}%
\end{picture}%
\setlength{\unitlength}{3947sp}%
\begin{picture}(10140,4043)(-1232,214)
\put(1391,364){\makebox(0,0)[lb]{\smash{\fontsize{20}{24}\usefont{T1}{ptm}{m}{n}{\color[rgb]{0,0,0}$x_1$}%
}}}
\put(4690,364){\makebox(0,0)[lb]{\smash{\fontsize{20}{24}\usefont{T1}{ptm}{m}{n}{\color[rgb]{0,0,0}$x_2$}%
}}}
\put(6001,3164){\makebox(0,0)[lb]{\smash{\fontsize{20}{24}\usefont{T1}{ptm}{m}{n}{\color[rgb]{0,0,0}$\bar u$}%
}}}
\end{picture}%
}
\caption{\small  A solution of (\ref{1})-(\ref{2})  in the stable case where $f< g$.  
Here the initial flux is $f(u)$ for $x<x_1$ and $x>x_2$, and $g(u)$ for $x_1<x<x_2$.
This produces a sink at $x_1$ and a source at $x_2$. }
\label{f:df4}
\end{center}
\end{figure}
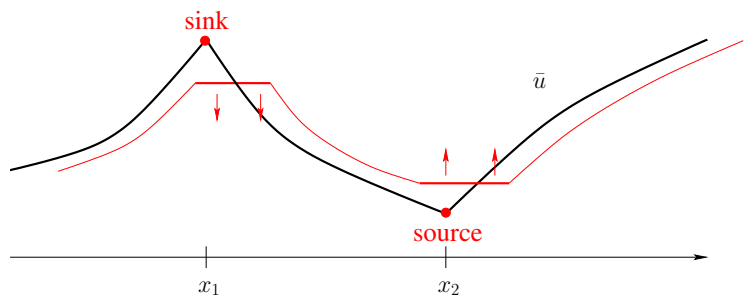

The {\bf unstable case} where $f>g$ is studied in the companion paper \cite{ABSu}.
In this case the equation \eqref{3} becomes backward parabolic, hence ill posed.
The conservation law (\ref{1}) has a source (or a sink) at every point $x_j$ of local maximum 
(or local minimum), respectively. See Fig.~\ref{f:df5}.
In a typical situation, this further increases the total variation of the solution. 
In \cite{ABSu} piecewise monotone admissible solutions of \eqref{1} are constructed. 
For a given initial data $\bar u$, a unique solution can be singled out by requiring that  the number of interfaces 
(i.e. the points where the flux switches between $f$ and $g$) remains as small as possible.

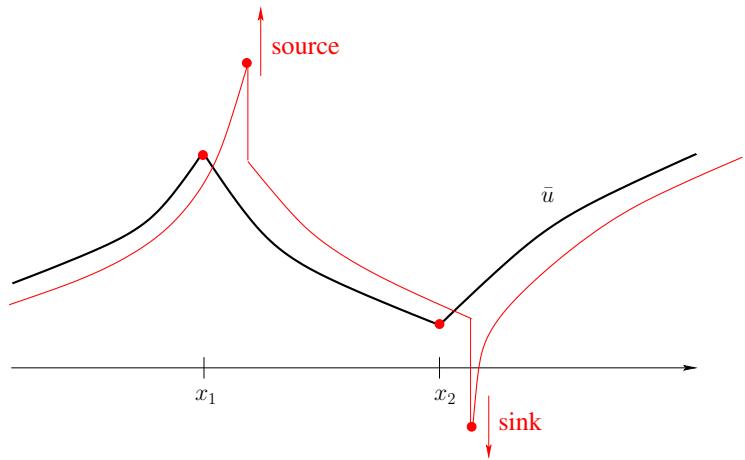
\begin{figure}[htbp]
\begin{center}
\resizebox{.65\textwidth}{!}{
\begin{picture}(0,0)%
\includegraphics{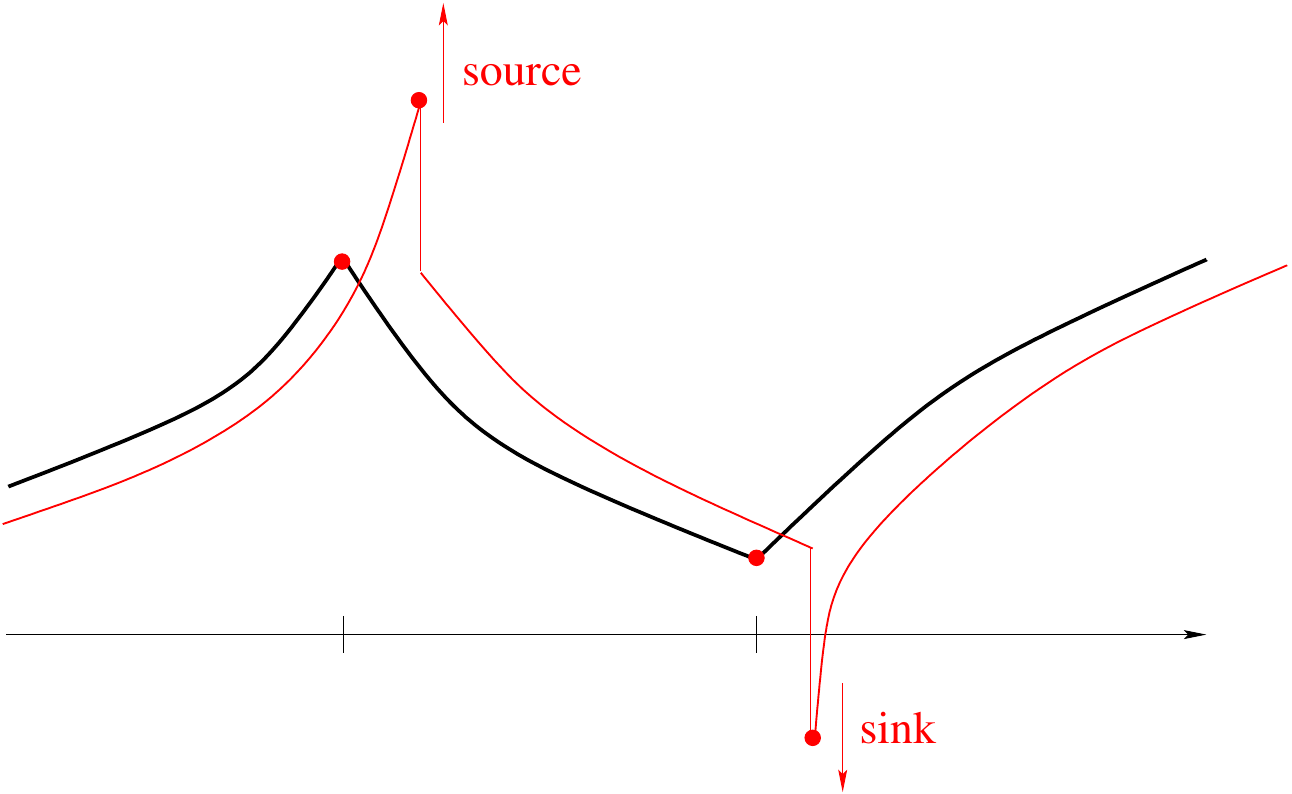}%
\end{picture}%
\setlength{\unitlength}{3947sp}%
\begin{picture}(10319,6339)(-1251,-434)
\put(6226,3164){\makebox(0,0)[lb]{\smash{\fontsize{20}{24}\usefont{T1}{ptm}{m}{n}{\color[rgb]{0,0,0}$\bar u$}%
}}}
\put(1376,389){\makebox(0,0)[lb]{\smash{\fontsize{20}{24}\usefont{T1}{ptm}{m}{n}{\color[rgb]{0,0,0}$x_1$}%
}}}
\put(4726,389){\makebox(0,0)[lb]{\smash{\fontsize{20}{24}\usefont{T1}{ptm}{m}{n}{\color[rgb]{0,0,0}$x_2$}%
}}}
\end{picture}%
}
\caption{\small  A solution of (\ref{1})-(\ref{2})  in the unstable case where $f> g$.  
Here the initial flux is $f(u)$ for $x<x_1$ and $x>x_2$, and $g(u)$ for $x_1<x<x_2$.
This produces a source at $x_1$ and a sink at $x_2$.}
\label{f:df5}
\end{center}
\end{figure}

Throughout the present paper we focus on the stable case $f<g$, and assume
\begi
\item[{\bf (A1)}] {\it The flux functions $f,g$ are $\C^2$ and satisfy, for some $c_0>0$,}
\bel{eq:f-le-g}
g(u) - f(u)~\geq ~c_0 \qquad \forall u\in\R\,.
\eeq
\endi

For any  initial data
\bel{idata} 
u(0,x)~=~\bar u(x) {\color{black} \ \in \L^1_{loc}(\R)}\,,
\eeq
the  parabolic equation \eqref{3} has  a unique solution \cite{Lieb}.
Our goal is to identify the limit of these solutions as $\ve, \delta\to 0+$.  

\begin{definition} \label{def:11} A locally integrable function $u=u(t,x)$ defined on $[0,T]\times\R$ 
is a  {\bf weak solution} to 
the Cauchy problem (\ref{1})-(\ref{2}), (\ref{idata}) if the following holds.
\begi
\item[(i)] The map $t\mapsto u(t,\cdot)$ is continuous with values in $\L^1_{loc}(\R)$, and satisfies (\ref{idata}).
\item[(ii)]  There exists a measurable function $\Theta: [0,T]\times\R\mapsto \{0,1\}$  such that, at every point $(t,x)$
where the partial derivative $u_x(t,x)$ is well defined, we have
\bel{18} \Theta(t,x)~=~\left\{\bega{rl} 1\quad &\hbox{if}\quad u_x(t,x)>0,\\[1mm]
0\quad &\hbox{if}\quad u_x(t,x)<0.\enda\right.\eeq
Moreover, for every test function $\vp\in \C^1_c\bigl(\,]0,T[\,\times\R)$, one has
\bel{weaksol}
 \dint \left\{ u\,\vp_t + \Big[\Theta \,
 f(u) + \bigl(1-\Theta \, 
 \bigr)g(u)\Big] \vp_x \right\}\, dx\,dt~=~0.\eeq
 \endi
 \end{definition}

We will show that, as $\ve,\delta\to 0+$, the solutions of the viscous problem (\ref{3}) and (\ref{idata})
converge to a unique limit, which provides a weak solution to the corresponding Cauchy problem for (\ref{1}).
To simplify the analysis, our main result will be stated in the spatially periodic case, 
where $u(t,\cdot)$ lies in the space
 \begin{align}
 \L^1_{per}(\R)\,\doteq\,\bigg\{ u:\R\mapsto\R\,;~& u(x)=u(x+1)~~\hbox{for all}~x\in\R\,,
 \nonumber \\
~~ & \|u\|_{\L^1_{per}} \doteq \int_0^1 \bigl|u(x)\bigr|\, dx\,<\,+\infty\,\bigg\}.
 \label{sper}
 \end{align}
Since the conservation law (\ref{1}) has finite propagation speed, considering spatially periodic solutions is 
hardly a restriction.  The case where $u(x+p)=u(x)$ for a different period $p$ can be easily reduced to (\ref{sper})
by a linear rescaling of space and time variables.

Next Theorem is the main result of the paper.

\begin{theorem}\label{t:11} Let $f,g$ satisfy {\bf (A1)}.
For any fixed $\ve,\delta>0$, and $\bar u\in \L_{per}^1(\R)$,  
the solution  $u^{\ve,\delta}(t,x)= S^{\ve,\delta}_t \bar u$ of (\ref{3}), (\ref{idata}) is well defined. 
These solutions
yield a contractive semigroup on $\L^1_{per}(\R)$.

Letting $\ve,\delta\to 0$, for every $\bar u\in \L^1_{per}(\R)$ 
 the convergence   $u^{\ve,\delta}(t,x)\to u(t,x)$  holds
in $\L^1_{loc}(\R_+\times\R)$. Denoting by  $S_t\bar u=u(t,\cdot)$ these limit functions,  
one obtains a semigroup of weak solutions to (\ref{1})-(\ref{2}), which is contractive w.r.t.~the  distance on $\L^1_{per}$.
\end{theorem}

The remainder of the paper is organized as follows. In Section~\ref{sec:2} we introduce an algorithm for constructing front tracking 
approximations, where $f,g$ are replaced by piecewise affine
flux functions $f_\nu, g_\nu$. As $\nu\to\infty$, we show that
these approximations converge to a unique limit semigroup $S$ on $\L^1(\R)$, which is contractive
w.r.t.~the $\L^1$ distance.
Further properties of this semigroup are proved in Section~\ref{sec:3}.
In particular, for every trajectory $t\mapsto u(t,\cdot)=S_t \bar u$, we show that the $\L^\infty$ norm and the total variation decay in time. In turn, this allows us to prove that $u$ is 
a weak solution to the Cauchy problem (\ref{1}), (\ref{idata}), according to Definition~\ref{def:11}.

The following three sections are concerned with solutions to the parabolic problem (\ref{3}),
in the spatially periodic case, proving Theorem~\ref{t:11}. 
By standard parabolic theory~\cite{Lieb}, for every $\ve,\delta>0$ the solutions to (\ref{3}) 
yield a contractive semigroup $S^{\ve,\delta}$ on $\L^1_{per}(\R)$.  
As $\ve,\delta\to 0$, to prove that 
the vanishing viscosity limits yield the same semigroup $S$ generated by front tracking approximations,
we follow the approach introduced in \cite{BY}.
The heart of the matter
is described in Proposition~\ref{p:41}.  Given a piecewise constant initial data $\bar u$ 
and $\ve_0>0$,
we consider a front tracking approximate solution $u_\nu$, which is close to the exact solution
$t\mapsto u(t,\cdot)= S_t\bar u$.   We then construct a smooth function $w=w(t,x)$ 
which is at the same
time close to the front tracking approximate solution $u_\nu$ 
and to the viscous solution $S^{\ve,\delta}_t\bar u$.
The construction of $w(t,\cdot) $, described in Section~\ref{sec:6},   is achieved by 
gluing together different types of approximations on different domains.
Namely: (1) intervals where $u_\nu(t,\cdot)$ attains a local maximum or minimum, (2)
neighborhoods of points where $u_\nu$ has a large jump, and (3) the remaining region where $u_\nu$ has only small jumps.  In Section~\ref{sec:7} we estimate by how much
the function $w$ fails to satisfy the parabolic equation (\ref{3}).  Since the semigroup
$S^{\ve,\delta}$ is contractive, by a standard error formula we thus obtain a bound on the distance
$\bigl\|w(t,\cdot)-S^{\ve,\delta}_t\bar u\bigr\|_{\L^1_{per}}$.
Finally, in Section~\ref{sec:8} we observe that, in the non-periodic case, 
solutions $u(t,\cdot)=S_t\bar u\in \L^1(\R)$ obtained as limits of front tracking approximations
can also be recovered as vanishing viscosity limits.  However, in this case as $\ve,\delta\to 0$
the convergence is only achieved in $\L^1_{loc}$.

\begin{remark} {\rm Our analysis shows that the 
limits of front tracking approximations are unique, and
yield exactly the same contractive semigroup on $\L^1_{per}$ as the limits of vanishing viscosity approximations (\ref{3}).  All semigroup trajectories are weak solutions to (\ref{1}) in the sense of 
Definition~\ref{def:11}.  \\
On the other hand, as shown in Example~\ref{ex:31}, there exist
weak solutions of (\ref{1}) which satisfy Definition~\ref{def:11} but cannot be recovered as
vanishing viscosity limits.  Imposing additional entropy conditions does not resolve the uniqueness issue.  We thus leave it as an open question to identify a suitable definition of ``admissible solution"
which characterizes the unique vanishing viscosity limits in the present setting.
For a characterization of the vanishing viscosity limits, in the case of a general $n\times n$ strictly hyperbolic system, see~\cite{BDL}.
}\end{remark}

An extensive literature is currently available on scalar conservation laws with discontinuous 
flux function. A partial list of references is given below.
%
\begin{itemize}
\item 
For results on scalar conservation laws with flux function discontinuous w.r.t.~the space 
and/or time variables, 
we refer to 
\cite{ AMV05,  BGS18, CR05, Diehl95, Gimse93, GR92, GS19, KR, KR95, {Mishra05},Panov09,
Panov10, SeguinVovelle03, WS15, WS16, WS18, Towers2000}. 
See also the survey paper \cite{And15} and references therein. 
\item
For the case where the flux function is discontinuous w.r.t.~the conserved quantity $u$, 
there is also a rich literature, see for example~\cite{ACnew, BGMS11, BGS13, BGS17, Gimse93, Panov23, Towers2020}.
\item  Conservation law models
where the flux function depends on the gradient $u_x$
have been considered  for hysteretic traffic flow models
in \cite{CF, Fan24}.  
%

In those models, the acceleration and  the deceleration modes can be connected either by 
a family of ``scanning curves'',  or
by using two families of fluxes.

\end{itemize} 

We remark that  our model \eqref{1}-\eqref{2}  can be  applied to traffic flow, 
but it remains meaningful for any couple of flux functions $f,g$, 
in an entirely general setting.

\v
\section{Front tracking approximations}
\label{sec:2}
\setcounter{equation}{0}
In this section, for general initial data $\bar u\in \L^1(\R)$,  we construct a family of piecewise constant
approximations to the 
problem (\ref{1})-(\ref{2}),  under the 
assumption \textbf{(A1)}.
We will show that the limits of these front tracking approximations yield a  semigroup
which is contractive w.r.t.~the $\L^1$ distance.  By the  finite propagation speed, the same result
holds in the spatially periodic case.
The construction follows Chapter 6 in \cite{Bbook}, with a modification
which takes into account the presence of two fluxes.

We first consider initial data (\ref{idata}),
assuming that $\bar u$ has bounded support and bounded total variation.
\begi
\item[(i)] For any integer $\nu\geq 1$, let $f_\nu$, $g_\nu$  be the piecewise affine functions that coincide with $f$, $g$ at the points $\{u_j\doteq 2^{-\nu}{j}\,,\ j\in\Z\}$. That means: 
\bel{fnudef}f_\nu(u)~=~ 2^\nu\Big[(u-u_j) f(u_{j+1}) + (u_{j+1}-u) f(u_j)\Big]\qquad
\hbox{if}~~ u\in [u_j, u_{j+1}],\eeq
 and similarly for $g_\nu$. Notice that (\ref{eq:f-le-g}) implies 
\textcolor{black}{$g_\nu(u)- f_\nu(u)\ge c_0$}
 for all $u\in\R$.

\item[(ii)]  Let $\{\bar u_\nu\}_{\nu\geq 1}$ be a sequence of piecewise constant 
initial data that approximate $\bar u$. More precisely, we assume
\begin{align}
\label{bunu}\TV \{\bar u_\nu\}\,\le\, \TV \{\bar u\}\,,\quad&\qquad  \|\bar u_\nu\|_{\L^\infty}\,\le\, \|\bar u\|_{\L^\infty}\,,\\
\label{1lim} 
\bar u_\nu(x)\in 2^{-\nu}\Z~~\hbox{for all} ~x\in\R,
~\quad&\qquad \| \bar u_\nu-\bar u\|_{\L^1}\to 0\quad\hbox{as}\quad \nu\to\infty.
\end{align}
\endi

At each point $\bar x$ where $\bar u_\nu$ has a jump, the Riemann problem is solved 
in two different ways, using the flux $f_\nu$ or $g_\nu$ according to the sign of $\Delta {\bar u}_\nu ({\bar x})=\bar u_\nu(\bar x+) - \bar u_\nu(\bar x-)$.

If $\Delta {\bar u}_\nu ({\bar x})>0$, we solve the Riemann problem by the piecewise constant entropy solution corresponding to the flux $f_\nu$, as in \cite[Sect.6.1]{Bbook}. In other words, $u_\nu(x,t)$ is locally given by the solution of
\begin{equation*}
    u_t + f_\nu(u)_x~=~0\,,\qquad \mbox{with}\quad 
    u(x,0) = \begin{cases} 
    u^- = {\bar u}_\nu(\bar x-)  &x<\bar x\,, \\
    u^+ = {\bar u}_\nu(\bar x+) & x>\bar x\,.
    \end{cases}
\end{equation*}
Similarly, if $\Delta {\bar u}_\nu ({\bar x})<0$, then the flux $g_\nu$ is employed.

The above construction would yield a piecewise constant solution 
$t\mapsto u_\nu(t,\cdot)$,
with jumps at points $x_\alpha(t)$, $\alpha=1,2,\ldots, N$, traveling with constant speed, 
and  $u_\nu (t,x)\doteq \Hat u_\alpha(t)$ on intervals $\bigl(x_\alpha(t),  x_{\alpha+1}(t)\bigr)$. 
This construction
must be modified on intervals
where $\Hat u_\alpha(t)$ attains a local maximum or a local minimum. Indeed, along this interval the flux switches
between $f$ and $g$. This determines a positive source at 
local minima and a negative source at local maxima.

CASE 1: a local maximum. This means $\Hat u_\alpha>\max\bigl\{ \Hat u_{\alpha-1},\Hat u_{\alpha+1}\bigr\}$. In this case, the value $\Hat u_\alpha(t)$ of  $u_\nu(t,x)$ over the 
interval $x\in \bigl(x_\alpha(t),  x_{\alpha+1}(t)\bigr)$ will no longer be constant in time, 
but satisfy the ODE
\begin{equation}\label{24}
 \frac{d}{ dt}  \Hat u_\alpha(t)~  =~  \frac{ f_\nu
 \bigl( \Hat u_\alpha(t)\bigr)-g_\nu\bigl( \Hat u_\alpha(t)\bigr) 
 }{x_{\alpha+1}(t) - x_\alpha(t)}~ \le ~0\,.\end{equation}
Notice that in this case the speeds $\dot x_\alpha(t)$ and 
$\dot x_{\alpha+1}(t)$ of the jumps are no longer constant in time, either. 
Indeed, they are determined by the 
Rankine-Hugoniot
equations
\bel{spe1}\dot x_{\alpha} ~=~ \frac{  f_\nu(\Hat u_\alpha)  - f_\nu(\Hat u_{\alpha-1})  }{\Hat u_\alpha-\Hat u_{\alpha-1}  } \,,\qquad
\dot x_{\alpha+1} ~=~ \frac{g_\nu(\Hat u_{\alpha+1})  - g_\nu(\Hat u_\alpha) }{\Hat u_{\alpha+1} - \Hat u_\alpha}\,,\eeq
where the right hand sides change in time, because of (\ref{24}).
Here and throughout the following, an upper dot denotes a derivative w.r.t.~time.

CASE 2: a local minimum. 
 This means $\Hat u_\alpha<\min\bigl\{ \Hat u_{\alpha-1},\Hat u_{\alpha+1}\bigr\}$. In this case, the value $\Hat u_\alpha(t)$ of  $u_\nu(t,x)$ over the 
interval $x\in \bigl(x_\alpha(t),  x_{\alpha+1}(t)\bigr)$ will  satisfy the ODE
\begin{equation}\label{26}
 \frac{d}{ dt}  \Hat u_\alpha(t)~  =~  \frac{ g_\nu\bigl( \Hat u_\alpha(t)\bigr)- f_\nu
 \bigl( \Hat u_\alpha(t)\bigr)
 }{x_{\alpha+1}(t) - x_\alpha(t)}~ \ge ~0\,.\end{equation}
In turn, the speeds of the jumps $\dot x_\alpha(t)$ and 
$\dot x_{\alpha+1}(t)$ are  determined by
\bel{spe2}\dot x_{\alpha} ~=~ \frac{ 
g_\nu(\Hat u_\alpha)- g_\nu(\Hat u_{\alpha-1})  }{\Hat u_\alpha-\Hat u_{\alpha-1}  } \,,\qquad
\dot x_{\alpha+1}~=~ \frac{f_\nu(\Hat u_{\alpha+1})  -  f_\nu(\Hat u_\alpha) }{\Hat u_{\alpha+1} - \Hat u_\alpha}\,.\eeq

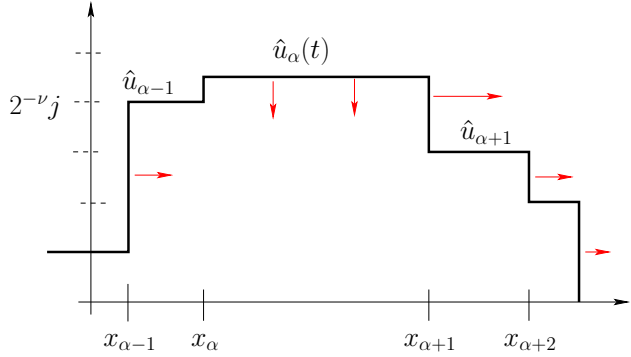
\begin{figure}[htbp]
\begin{center}
\resizebox{.55\textwidth}{!}{
\begin{picture}(0,0)%
\includegraphics{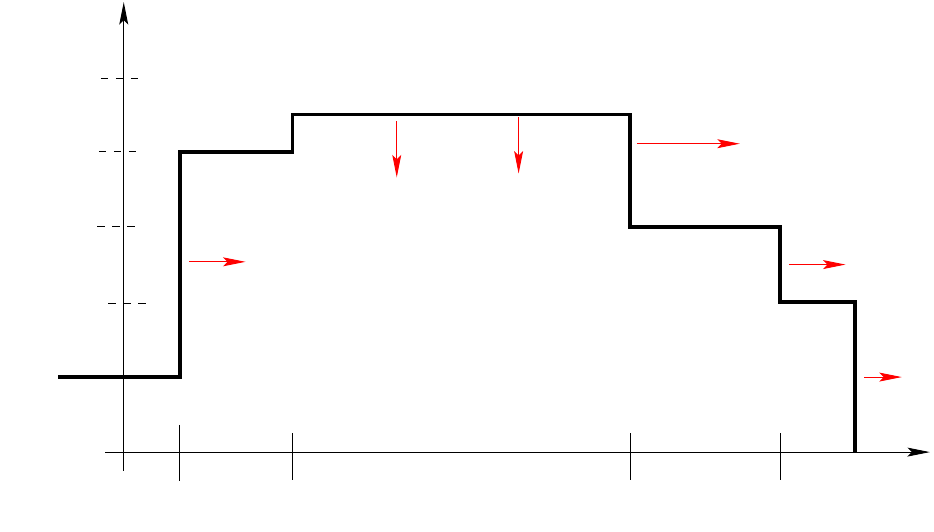}%
\end{picture}%
\setlength{\unitlength}{3947sp}%
\begin{picture}(7452,4246)(3361,-5795)
\put(4726,-2611){\makebox(0,0)[lb]{\smash{\fontsize{20}{24}\usefont{T1}{ptm}{m}{n}{\color[rgb]{0,0,0}$\hat u_{\alpha-1}$}%
}}}
\put(8101,-5686){\makebox(0,0)[lb]{\smash{\fontsize{20}{24}\usefont{T1}{ptm}{m}{n}{\color[rgb]{0,0,0}$x_{\alpha+1}$}%
}}}
\put(4501,-5686){\makebox(0,0)[lb]{\smash{\fontsize{20}{24}\usefont{T1}{ptm}{m}{n}{\color[rgb]{0,0,0}$x_{\alpha-1}$}%
}}}
\put(9301,-5686){\makebox(0,0)[lb]{\smash{\fontsize{20}{24}\usefont{T1}{ptm}{m}{n}{\color[rgb]{0,0,0}$x_{\alpha+2}$}%
}}}
\put(8776,-3211){\makebox(0,0)[lb]{\smash{\fontsize{20}{24}\usefont{T1}{ptm}{m}{n}{\color[rgb]{0,0,0}$\hat u_{\alpha+1}$}%
}}}
\put(6526,-2236){\makebox(0,0)[lb]{\smash{\fontsize{20}{24}\usefont{T1}{ptm}{m}{n}{\color[rgb]{0,0,0}$\hat u_\alpha(t)$}%
}}}
\put(3376,-2911){\makebox(0,0)[lb]{\smash{\fontsize{20}{24}\usefont{T1}{ptm}{m}{n}{\color[rgb]{0,0,0}$2^{-\nu}j$}%
}}}
\put(5551,-5686){\makebox(0,0)[lb]{\smash{\fontsize{20}{24}\usefont{T1}{ptm}{m}{n}{\color[rgb]{0,0,0}$x_\alpha$}%
}}}
\end{picture}%
}
\caption{\small  A front tracking solution to the conservation law (\ref{1})-(\ref{2}) with two fluxes.   Here the  values  $\Hat u_{\alpha-1}$ and $\Hat u_{\alpha+1}$ are integer multiples 
of $ 2^{-\nu}$, and remain constant in time.
On the other hand $\Hat u_\alpha(t)$ is a local maximum, and decreases continuously in time.
At some time $\tau$, we have the equality $\Hat u_\alpha(\tau) = \Hat u_{\alpha-1}$.
The algorithm is then restarted, removing the front at $x_{\alpha}$.}
\label{f:ag60}
\end{center}
\end{figure}

This construction can be prolonged in time up to the first time $\tau$ when 
one of the following cases occurs.
\v
{\bf Case (i)} Two or more adjacent local max/min either merge into a single local max/min or cancel each other so the local max/min is removed.

{\bf Case (ii)}  At an endpoint of an interval where $u_\nu$ attains a local max or a local min,
the jump in $u_\nu$ vanishes,  or the wave-front no longer satisfies Liu admissibility condition for $t>\tau$, due to the change in the value $\Hat u_\alpha(t)$.

{\bf Case (iii)} Two or more fronts interact. 
\v

In all three cases, since $u_\nu(\tau,\cdot)$ is still a piecewise constant  function with a uniformly bounded number of  jumps, the algorithm can be restarted. 
It is clear that  the number $N(t)$ of intervals where the front tracking solution 
$u_\nu(t,\cdot)$ attains
a local maximum or a local minimum does not increase in time.
Moreover, the total variation of $u_\nu(t,\cdot)$ does not increase across any interaction.
Calling ${\cal J}(t)$ the total number of jumps in $u_\nu(t,\cdot)$, 
we thus have the bound
\begin{equation}\label{frontsE}
{\cal J}(t)~\leq~2 N(t) + 2^\nu\cdot \TV\bigl\{ u_\nu(t,\cdot)\bigr\}~\leq~2 N(0)+ 
2^\nu\cdot\TV\{\bar u\}. 
\end{equation}

We claim that  the total number of restarting times of types (i)--(iii)  is bounded. 

As a preliminary we observe that, in a standard front tracking approximation \cite{Bbook,Daf}, 
upward jumps can interact with downward jumps, resulting in a cancellation.   
This never happens in the present setting. Indeed (see Fig.~\ref{f:df52})  by (\ref{spe1}), (\ref{spe2}), 
the rate at which  local maxima $u_\alpha(t) $ decrease (and local minima increase) goes to infinity
as the length of the interval $\bigl[x_{\alpha}(t),\,x_{\alpha+1}(t)\bigr] $ approaches zero. 
As a consequence, one of the jumps at $x_{\alpha}(t)$ or at $ x_{\alpha+1}(t)$ must vanish
before the interval shrinks to a single point.  The next lemma states this fact more precisely.

\begin{figure}[ht]
\begin{center}
\resizebox{.4\textwidth}{!}{
\begin{picture}(0,0)%
\includegraphics{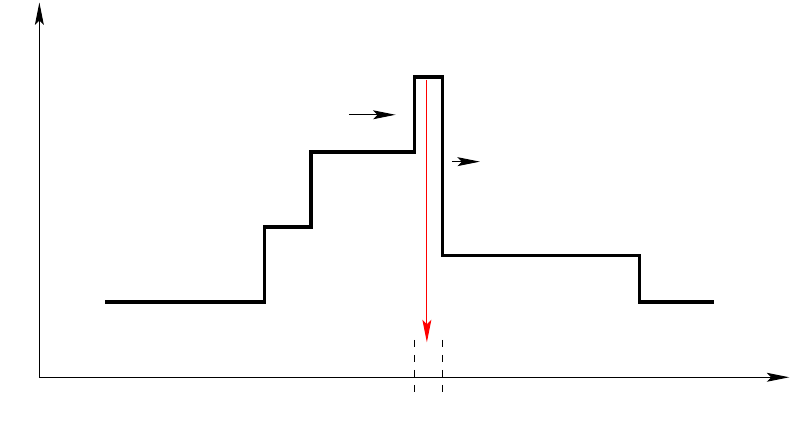}%
\end{picture}%
\setlength{\unitlength}{3947sp}%
\begin{picture}(6327,3511)(-2114,2140)
\put(751,2264){\makebox(0,0)[lb]{\smash{\fontsize{24}{28.8}\usefont{T1}{ptm}{m}{n}{\color[rgb]{0,0,0}$a(t)$}%
}}}
\put(3751,2339){\makebox(0,0)[lb]{\smash{\fontsize{24}{28.8}\usefont{T1}{ptm}{m}{n}{\color[rgb]{0,0,0}$x$}%
}}}
\put(1201,5114){\makebox(0,0)[lb]{\smash{\fontsize{24}{28.8}\usefont{T1}{ptm}{m}{n}{\color[rgb]{0,0,0}$\hat u(t)$}%
}}}
\put(-2099,5189){\makebox(0,0)[lb]{\smash{\fontsize{24}{28.8}\usefont{T1}{ptm}{m}{n}{\color[rgb]{0,0,0}$u$}%
}}}
\put(1426,2264){\makebox(0,0)[lb]{\smash{\fontsize{24}{28.8}\usefont{T1}{ptm}{m}{n}{\color[rgb]{0,0,0}$b(t)$}%
}}}
\end{picture}%
}
\caption{\small An illustration of Lemma~\ref{l:21}. As the length $\ell(t)$ of the interval $\bigl[a(t), \, b(t)\bigr]$ 
shrinks to zero, the value $\Hat u(t)$ of the local maximum decreases very fast, as in (\ref{shrin}).
}
\label{f:df52}
\end{center}
\end{figure}
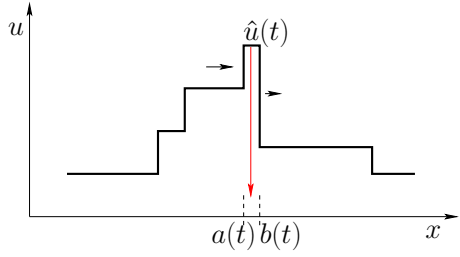

\begin{lemma}\label{l:21}
Let $f,g$ satisfy the assumptions {\bf (A1)}, and consider  a piecewise constant initial data
$\bar u\in \L^1(\R)$. For a given $\nu\geq 1$, let  $u_\nu(t,\cdot) = S_t^\nu\bar u$ be 
the corresponding front tracking solution. 

For $t\ge 0$, let $\ell(t)$ denote the shortest length among all bounded intervals where $u(t,\cdot)$ attains a local maximum or minimum, 
Then, for all $t\geq 0$, the minimum length $\ell(t)$ remains uniformly positive.
%
\end{lemma}

\begin{proof} 
{\bf 1.} \textcolor{black}{By construction, the $L^\infty$ norm of the solution $u_\nu(t,\cdot)$ does not increase in time. Recalling (\ref{bunu}), 
an upper bound on all wave speeds is given by}:
\bel{lmax} \lambda^\dagger~\doteq~\max\Big\{ \bigl|f'(\omega)\bigr|\,, ~\bigl|g'(\omega)\bigr|\,;~~|\omega|\leq\|\bar u\|_{\L^\infty}
\Big\}\,.
\eeq
Moreover, set $\ell_0\doteq \ell(0)>0$.

Let $\bar u$ attain a local max on an interval $[a_0,b_0]$, with $b_0-a_0 \ge  \ell_0$.   At time $t>0$,  call
 $\bigl[a(t), b(t)\bigr]$ the corresponding interval where the front tracking solution $u_\nu(t,\cdot)$ attains a local max.
 For $t\leq \ell_0/4\lambda^\dagger$, since all jumps travel with speed bounded by $\lambda^\dagger$, we have the trivial bound
\bel{eqM1}
\ell(t)~= ~b(t)-a(t)~\geq~\ell_0-2\lambda^\dagger t~\geq~{\ell_0\over 2}\qquad\qquad 0\leq t\leq ~{\ell_0\over 4\lambda^\dagger}\,
 .\eeq
\v
{\bf 2.} Next, consider a time $t > t_0\doteq {\ell_0/ 4\lambda^\dagger}$.
By finite propagation speed, for $t_0<\tau<t$ the length $\ell(\tau)$ satisfies
$$\ell(\tau) ~\leq ~\ell(t) + 2\lambda^\dagger(t-\tau).$$
Calling $\Hat u(\tau)\doteq u_\nu(\tau,x)$ the value of the front tracking solution on the interval 
$ \bigl[a(\tau), \,  b(\tau)\bigr]$ where it achieves the local maximum (see Fig.~\ref{f:df52}),
by (\ref{24}) it follows
\bel{shrin}\frac{d}{d\tau} \Hat u(\tau)~\leq~- \frac{c_0}{ \ell(\tau)}~\leq~-\frac{c_0}{ \ell (t) + 2\lambda^\dagger (t-\tau)}\,.
\eeq
A straightforward integration yields
\begin{align*}
2\|\bar u\|_{\L^\infty} &~\ge~   \Hat u(t_0) -\Hat u(t)
~\ge~  \int_{t_0} ^t \frac{c_0}{ \ell(t) + 2\lambda^\dagger (t-\tau)}\, d\tau 
~=~    \frac{c_0}{2\lambda^\dagger} \ln\left(1+ \frac{2\lambda^\dagger (t-t_0)}{ \ell(t)}\right). 
\end{align*}
This leads to the bound 
\begin{equation}\label{eqM2}
\ell(t) ~\ge~ \frac{2\lambda^\dagger (t-t_0)}{\exp(4\lambda^\dagger \|\bar u\|_{\L^\infty}/c_0) -1}\qquad\qquad
\forall  t\geq ~{\ell_0\over 4\lambda^\dagger}\,.
\end{equation}
%
%
Together, (\ref{eqM1}) and (\ref{eqM2}) yield the lemma in the case of local max.
The size of intervals where $u(t,\cdot)$ attains a local minimum is estimated in an entirely similar way.
\end{proof}

We are now ready to prove that the front-tracking approximations $u_\nu$ constructed by our algorithm 
are well defined.  

\begin{lemma}\label{l:22} For any piecewise constant initial data $\bar u\in \L^1(\R)$, let $u_\nu(t,x)$
be the corresponding front tracking approximation, defined
as in (\ref{fnudef})-(\ref{spe2}).  Then the total number of all restarting times $\tau\in \,]0,+\infty[\,$
of the above types (i)--(iii) is bounded.
\end{lemma}
\v
\begin{proof}
{\bf 1.} Calling $N(t)$ the number of intervals where $u_\nu(t,\cdot)$ attains a local max or min,
it is clear that  Case (i) can occur at most $N(0)$ times.

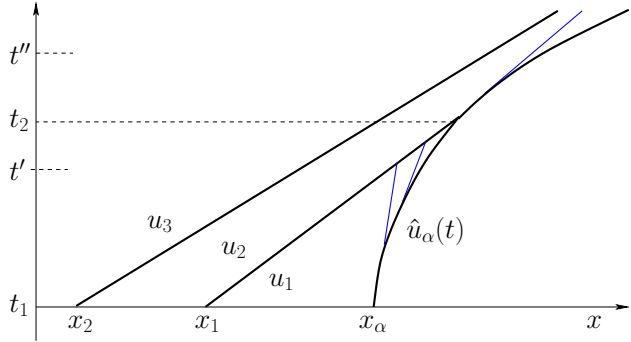
\begin{figure}[ht]
\begin{center}
\resizebox{.55\textwidth}{!}{
\begin{picture}(0,0)%
\includegraphics{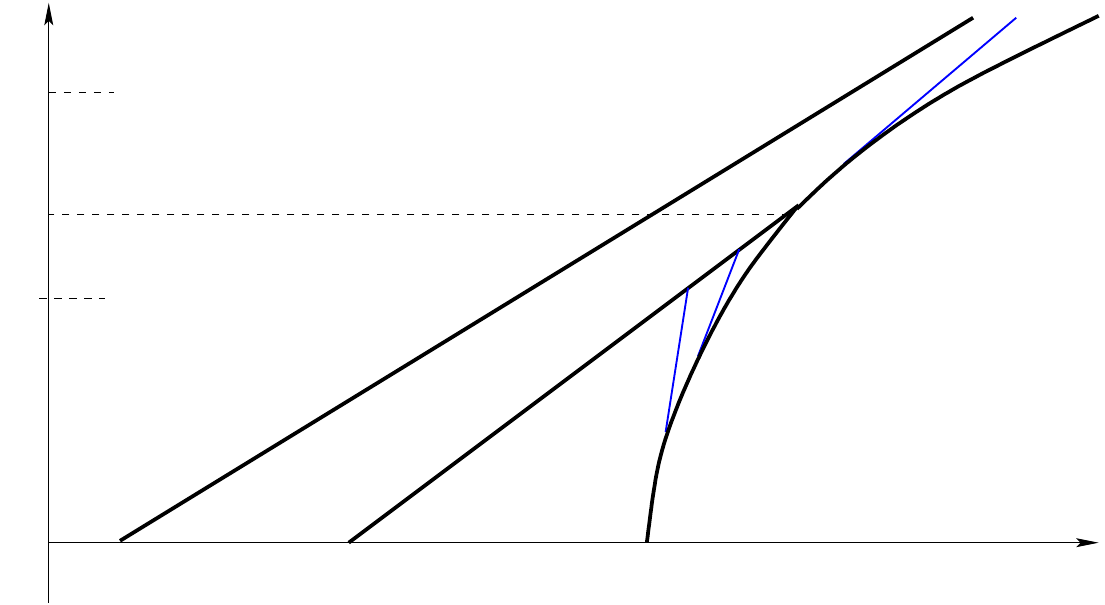}%
\end{picture}%
\setlength{\unitlength}{3947sp}%
\begin{picture}(8823,4824)(-389,347)
\put(-374,764){\makebox(0,0)[lb]{\smash{\fontsize{24}{28.8}\usefont{T1}{ptm}{m}{n}{\color[rgb]{0,0,0}$t_1$}%
}}}
\put(5251,1814){\makebox(0,0)[lb]{\smash{\fontsize{24}{28.8}\usefont{T1}{ptm}{m}{n}{\color[rgb]{0,0,0}$\hat u_\alpha(t)$}%
}}}
\put(4576,539){\makebox(0,0)[lb]{\smash{\fontsize{24}{28.8}\usefont{T1}{ptm}{m}{n}{\color[rgb]{0,0,0}$x_\alpha$}%
}}}
\put(3301,1139){\makebox(0,0)[lb]{\smash{\fontsize{24}{28.8}\usefont{T1}{ptm}{m}{n}{\color[rgb]{0,0,0}$u_1$}%
}}}
\put(2626,1589){\makebox(0,0)[lb]{\smash{\fontsize{24}{28.8}\usefont{T1}{ptm}{m}{n}{\color[rgb]{0,0,0}$u_2$}%
}}}
\put(-374,4289){\makebox(0,0)[lb]{\smash{\fontsize{24}{28.8}\usefont{T1}{ptm}{m}{n}{\color[rgb]{0,0,0}$t''$}%
}}}
\put(-374,3389){\makebox(0,0)[lb]{\smash{\fontsize{24}{28.8}\usefont{T1}{ptm}{m}{n}{\color[rgb]{0,0,0}$t_2$}%
}}}
\put(-374,2639){\makebox(0,0)[lb]{\smash{\fontsize{24}{28.8}\usefont{T1}{ptm}{m}{n}{\color[rgb]{0,0,0}$t'$}%
}}}
\put(451,539){\makebox(0,0)[lb]{\smash{\fontsize{24}{28.8}\usefont{T1}{ptm}{m}{n}{\color[rgb]{0,0,0}$x_2$}%
}}}
\put(2251,539){\makebox(0,0)[lb]{\smash{\fontsize{24}{28.8}\usefont{T1}{ptm}{m}{n}{\color[rgb]{0,0,0}$x_1$}%
}}}
\put(7801,539){\makebox(0,0)[lb]{\smash{\fontsize{24}{28.8}\usefont{T1}{ptm}{m}{n}{\color[rgb]{0,0,0}$x$}%
}}}
\put(1576,1964){\makebox(0,0)[lb]{\smash{\fontsize{24}{28.8}\usefont{T1}{ptm}{m}{n}{\color[rgb]{0,0,0}$u_3$}%
}}}
\end{picture}%
}
\caption{\small A front tracking solution where $f_\nu$ is not convex.
As the value $\Hat u_\alpha(t)$ of the local max decreases,
new fronts are created during the interval $[t_1, t']$.   
In the subsequent  time interval $[t', t_2]$, these fronts merge with the shock at $x_1(t)$.
 At time $t_2 $, a single shock is formed again, with left and right states $u_2, \Hat u_\alpha(t_2)$.
 The same pattern can then be repeated a finite number of times.
}
\label{f:df51}
\end{center}
\end{figure}

{\bf 2.} Next we observe that, if the flux functions $f,g$ are convex
(or concave),
then at each restarting time $\tau$ the number of jumps in $u_\nu$  decreases at least by 1. 
It is thus clear that only finitely many restarting times can occur.

The general case, where  $f,g$ are not convex, requires a more careful analysis.
Indeed, since the  values $\Hat u_\alpha(t)$ 
at  (\ref{24}) or (\ref{26}) vary in time, 
it may occur that a large shock located at $x_\alpha(t)$ 
breaks up into two or more smaller jumps (see Fig.~\ref{f:df50}, right).

It clearly suffices to consider the set of restarting times related to all the fronts 
between two local extrema of $u_\nu$, say between a local min and a local max.
As shown in Fig.~\ref{f:df50}, left, let
$$\Hat u_{\alpha-1}(t_1)\,<\, u_m\,<~\cdots~< \,u_2\,<\,u_1\,<\,\Hat u_\alpha(t_1)$$
be the intermediate states between a local min and a local max of $u_\nu(t_1,\cdot)$.
Consider the increasing sequence of times, inductively defined by setting
$$\bega{rl} t_2&\doteq~\inf\Big\{ t>t_1\,;~~\hbox{the solution $u_\nu(t,\cdot)$ contains a jump} \\[2mm]
&\qquad \qquad \hbox{ with left and right states
~$u_i, \Hat u_\alpha(t)$, for some} ~u_i\doteq u_{i(2)} < u_1\Big\}.\enda$$
$$\bega{rl} t_{k}&\doteq~\inf\Big\{ t>t_{k-1}\,;~~\hbox{the solution $u_\nu(t,\cdot)$ contains a jump} \\[2mm]
&\qquad \qquad \hbox{ with left and right states
~$u_i, \Hat u_\alpha(t)$, for some} ~u_i\doteq u_{i(k)} < u_{i(k-1)}\Big\}.\enda$$

Clearly, there can be at most $m$ such times.
It now suffices to check that the total number of restartings during each time interval $[t_k , t_{k+1}[\,$ is 
finite.   

Indeed, during this time interval new fronts can be produced as the jump 
with left and right states $u_{i(k)}, \Hat u_\alpha(t)$ breaks up into smaller jumps.
Since all intermediate states must be integer multiples of $2^{-\nu}$, the total number of these 
new fronts is $\leq 2^\nu\bigl(\Hat u_\alpha(t_k)-u_{i(k)}\bigr)$.
All other interactions involve two or more fronts merging into each other.  
In every such interaction, the total number of fronts decreases at least by 1.
This yields an upper bound on the total number of restarting times, within each interval $[t_k , t_{k+1}[\,$.

The above analysis yields a bound on the number of new fronts produced as the value of the local maximum
$\Hat u_\alpha(t)$ decreases. An entirely similar argument yields a bound on the number of new fronts
that can be produced as the value of the local minimum
$\Hat u_{\alpha-1}(t)$ increases in time.  Combining these bounds, the lemma is proved.
\end{proof}

\begin{figure}[ht]
\begin{center}
\resizebox{.99\textwidth}{!}{
\begin{picture}(0,0)%
\includegraphics{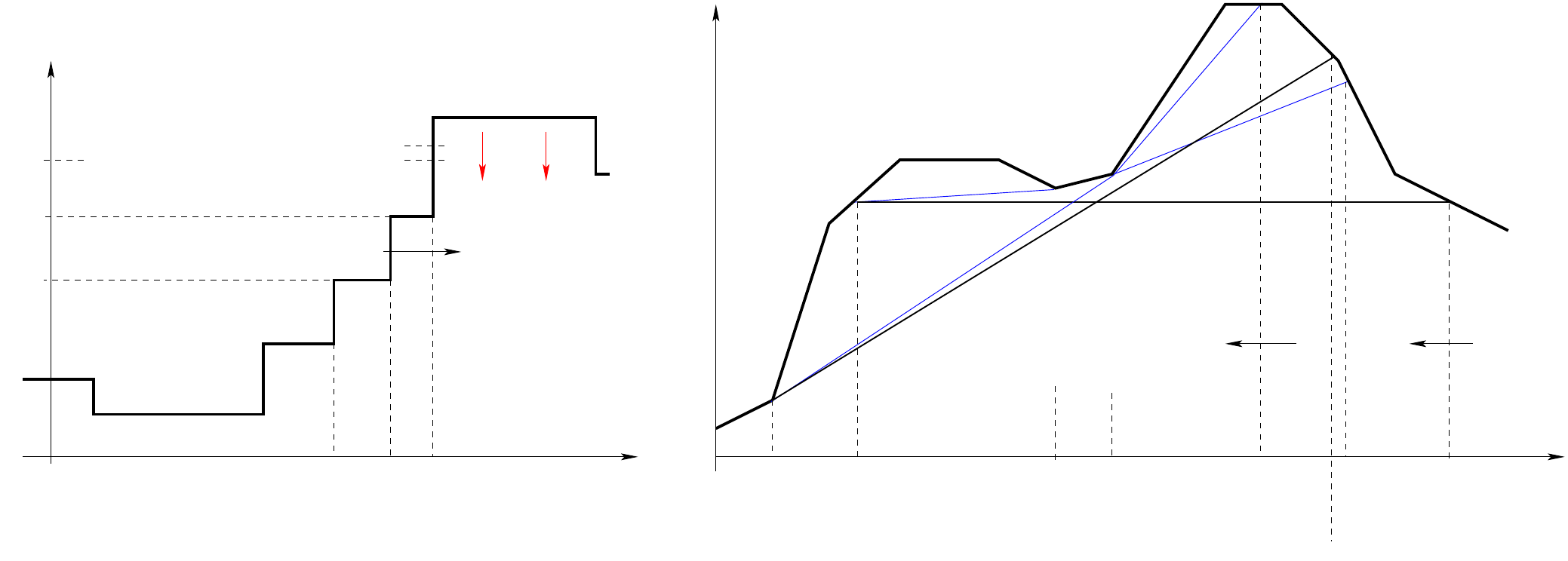}%
\end{picture}%
\setlength{\unitlength}{3947sp}%
\begin{picture}(16602,5936)(-6389,-264)
\put(-1499,4589){\makebox(0,0)[lb]{\smash{\fontsize{24}{28.8}\usefont{T1}{ptm}{m}{n}{\color[rgb]{0,0,0}$\hat u_\alpha(t_1)$}%
}}}
\put(-5099,1439){\makebox(0,0)[lb]{\smash{\fontsize{24}{28.8}\usefont{T1}{ptm}{m}{n}{\color[rgb]{0,0,0}$\hat u_{\alpha-1}$}%
}}}
\put(-6299,4589){\makebox(0,0)[lb]{\smash{\fontsize{24}{28.8}\usefont{T1}{ptm}{m}{n}{\color[rgb]{0,0,0}$u$}%
}}}
\put(-6374,3839){\makebox(0,0)[lb]{\smash{\fontsize{24}{28.8}\usefont{T1}{ptm}{m}{n}{\color[rgb]{0,0,0}$\omega_k$}%
}}}
\put(-6374,3239){\makebox(0,0)[lb]{\smash{\fontsize{24}{28.8}\usefont{T1}{ptm}{m}{n}{\color[rgb]{0,0,0}$u_1$ }%
}}}
\put(-6374,2564){\makebox(0,0)[lb]{\smash{\fontsize{24}{28.8}\usefont{T1}{ptm}{m}{n}{\color[rgb]{0,0,0}$u_2$}%
}}}
\put(-374,464){\makebox(0,0)[lb]{\smash{\fontsize{24}{28.8}\usefont{T1}{ptm}{m}{n}{\color[rgb]{0,0,0}$x$}%
}}}
\put(-1874,464){\makebox(0,0)[lb]{\smash{\fontsize{24}{28.8}\usefont{T1}{ptm}{m}{n}{\color[rgb]{0,0,0}$x_\alpha$}%
}}}
\put(-2474,464){\makebox(0,0)[lb]{\smash{\fontsize{24}{28.8}\usefont{T1}{ptm}{m}{n}{\color[rgb]{0,0,0}$x_1$}%
}}}
\put(-2999,464){\makebox(0,0)[lb]{\smash{\fontsize{24}{28.8}\usefont{T1}{ptm}{m}{n}{\color[rgb]{0,0,0}$x_2$}%
}}}
\put(3976,4139){\makebox(0,0)[lb]{\smash{\fontsize{24}{28.8}\usefont{T1}{ptm}{m}{n}{\color[rgb]{0,0,0}$f_\nu$}%
}}}
\put(1651,464){\makebox(0,0)[lb]{\smash{\fontsize{24}{28.8}\usefont{T1}{ptm}{m}{n}{\color[rgb]{0,0,0}$u_2$}%
}}}
\put(2551,464){\makebox(0,0)[lb]{\smash{\fontsize{24}{28.8}\usefont{T1}{ptm}{m}{n}{\color[rgb]{0,0,0}$u_1$}%
}}}
\put(7726,464){\makebox(0,0)[lb]{\smash{\fontsize{24}{28.8}\usefont{T1}{ptm}{m}{n}{\color[rgb]{0,0,0}$\hat u_\alpha(t')$}%
}}}
\put(5251,464){\makebox(0,0)[lb]{\smash{\fontsize{24}{28.8}\usefont{T1}{ptm}{m}{n}{\color[rgb]{0,0,0}$\omega_{k+1}$}%
}}}
\put(4651,464){\makebox(0,0)[lb]{\smash{\fontsize{24}{28.8}\usefont{T1}{ptm}{m}{n}{\color[rgb]{0,0,0}$\omega_k$}%
}}}
\put(7426,-136){\makebox(0,0)[lb]{\smash{\fontsize{24}{28.8}\usefont{T1}{ptm}{m}{n}{\color[rgb]{0,0,0}$\hat u_\alpha(t_2)$}%
}}}
\put(6601,464){\makebox(0,0)[lb]{\smash{\fontsize{24}{28.8}\usefont{T1}{ptm}{m}{n}{\color[rgb]{0,0,0}$\hat u_\alpha(t'')$}%
}}}
\put(9001,464){\makebox(0,0)[lb]{\smash{\fontsize{24}{28.8}\usefont{T1}{ptm}{m}{n}{\color[rgb]{0,0,0}$\hat u_\alpha(t_1)$}%
}}}
\end{picture}%
}
\caption{\small Left: the profile of the front tracking solution $u_\nu$ considered in Fig.~\ref{f:df51}, at time $t=t_1$.
Right: the left and right states across the jump at $x_\alpha(t)$, at times $t_1<t'<t_2< t''$.
As the left state $\Hat u_\alpha(t)$ decreases, for $t\in [t_1, t']$ the initial shock at $x_\alpha$ 
splits into multiple fronts, with intermediate states $\omega_k,\omega_{k+1}\in 2^{-\nu}\,{\mathbb Z}$. 
During the following time interval $[t', t_2]$ these fronts interact with the shock at $x_1$, 
and at time $t_2$ a single shock with left and right states $u_2, \Hat u_\alpha(t_2)$ is formed.
The pattern can then repeat: by the time $t''$, the value $\Hat u_\alpha(t'')$ has further decreased
and an additional front has emerged.}
\label{f:df50}
\end{center}
\end{figure}
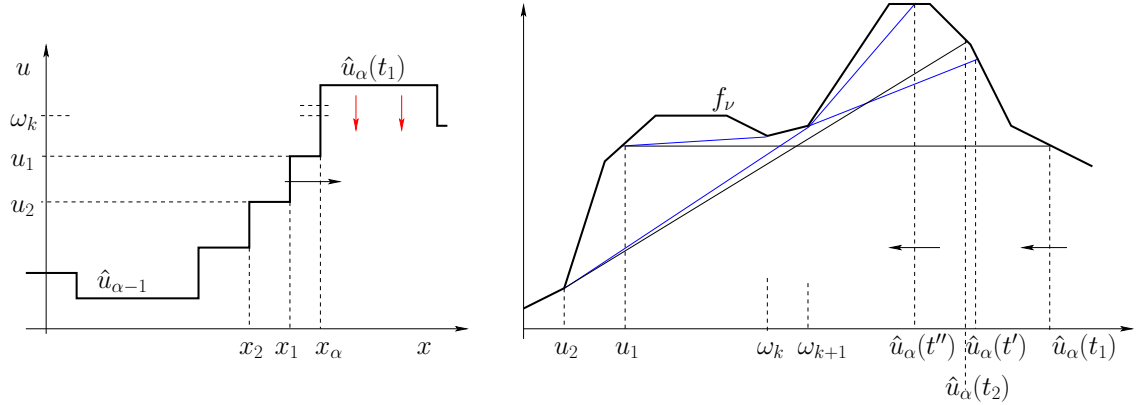

Having proved that front tracking approximations are globally defined, 
in the following proposition we establish some of their properties.

\begin{proposition}\label{p:21} Let $f,g$ satisfy {\bf (A1)}. Let $u_\nu(t,x)$, $v_\nu(t,x)$ be  front tracking solutions 
constructed according to the previous algorithm. Then
\begi
\item[(i)] The norm  $ \bigl\|u_\nu(t,\cdot)\bigr\|_{\L^1}$ is non-increasing in time.
\item[(ii)] The maps $t\mapsto \TV\bigl\{ u_\nu(t,\cdot)\bigr\}$ and 
$t\mapsto \bigl\|u_\nu(t,x)\bigr\|_{\L^\infty}$ are non-increasing. 
\item[(iii)] The comparison principle holds: if $u_\nu(0,x)\leq v_\nu(0,x)$ for all $x\in \R$,  then $u_\nu(t,x)\leq v_\nu(t,x)$
for all $t\geq 0$, $x\in \R$.
\item[(iv)] The $\L^1$ distance between solutions does not increase in time:
\bel{L1d} 
\bigl\| u_\nu(t)-v_\nu(t)\bigr\|_{\L^1}~\leq~\bigl\| u_\nu(0)-v_\nu(0)\bigr\|_{\L^1}\qquad
\hbox{for any} ~~ t>0.
\eeq
\endi
\end{proposition}

\begin{proof}
{\bf 1.} To prove (i) we observe that $u_\nu$ is an entropy weak solution 
of the equation
\bel{Fnu}
u_t + \bigl(F_\nu(t,x,u)\bigr)_x~=~0\,,
\eeq
 for
the interpolated flux
\bel{ipflux}
F_\nu(t,x,u) ~\doteq~\theta_\nu(x,t) f_\nu(u) + \bigl(1-\theta_\nu(x,t)\bigr) g_\nu(u)\,,
\eeq
where the parameter $\theta_\nu\in [0,1]$ is affine-interpolated  as follows.
\begi
\item If $\bigl[ x_\alpha(t), x_{\alpha+1}(t)]$ is an interval 
where $u_\nu(t,\cdot)$ attains a local maximum, then
\bel{th-n-1}
\theta_\nu(t,x)~\doteq~\frac{x_{\alpha+1}(t)-x}{ x_{\alpha+1}(t)-x_\alpha(t)} 
\qquad\forall x\in \bigl[ x_\alpha(t), x_{\alpha+1}(t)].
\end{equation}
\item If $\bigl[ x_\alpha(t), x_{\alpha+1}(t)]$ is an interval 
where $u_\nu(t,\cdot)$ attains a local minimum, then
\bel{th-n-2}
\theta_\nu(t,x)~\doteq~\frac{x-x_{\alpha}(t)}{ x_{\alpha+1}(t)-x_\alpha(t)} 
\qquad\forall x\in \bigl[ x_\alpha(t), x_{\alpha+1}(t)].
\end{equation}
\endi
The definition of $\theta_\nu$  is completed by requiring of $\theta_\nu(t,\cdot)$
to be constant on every
connected component of the remaining set $\R\setminus \bigcup_{\alpha\in \M}
\bigl[ x_\alpha(t), x_{\alpha+1}(t)]$, where $\M$ is the set of indices corresponding to
local maxima or minima. Notice that in this way one obtains a map $x\mapsto \theta_\nu(t,x)$ which is continuous,
piecewise affine, defined for all $x\in\R$.

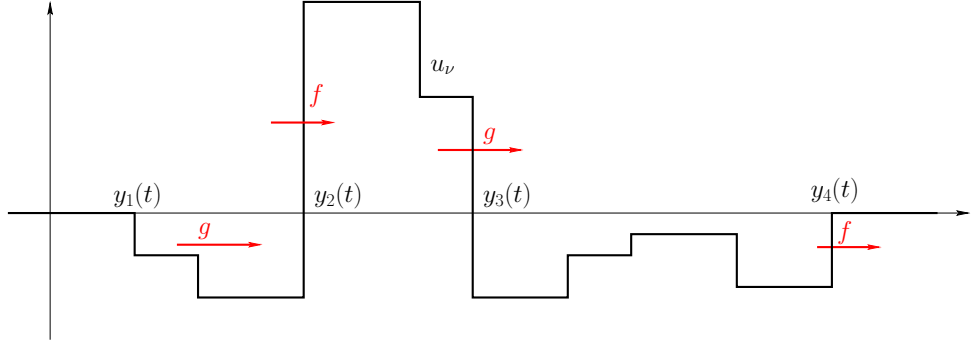
\begin{figure}[htbp]
\begin{center}
\resizebox{.85\textwidth}{!}{
\begin{picture}(0,0)%
\includegraphics{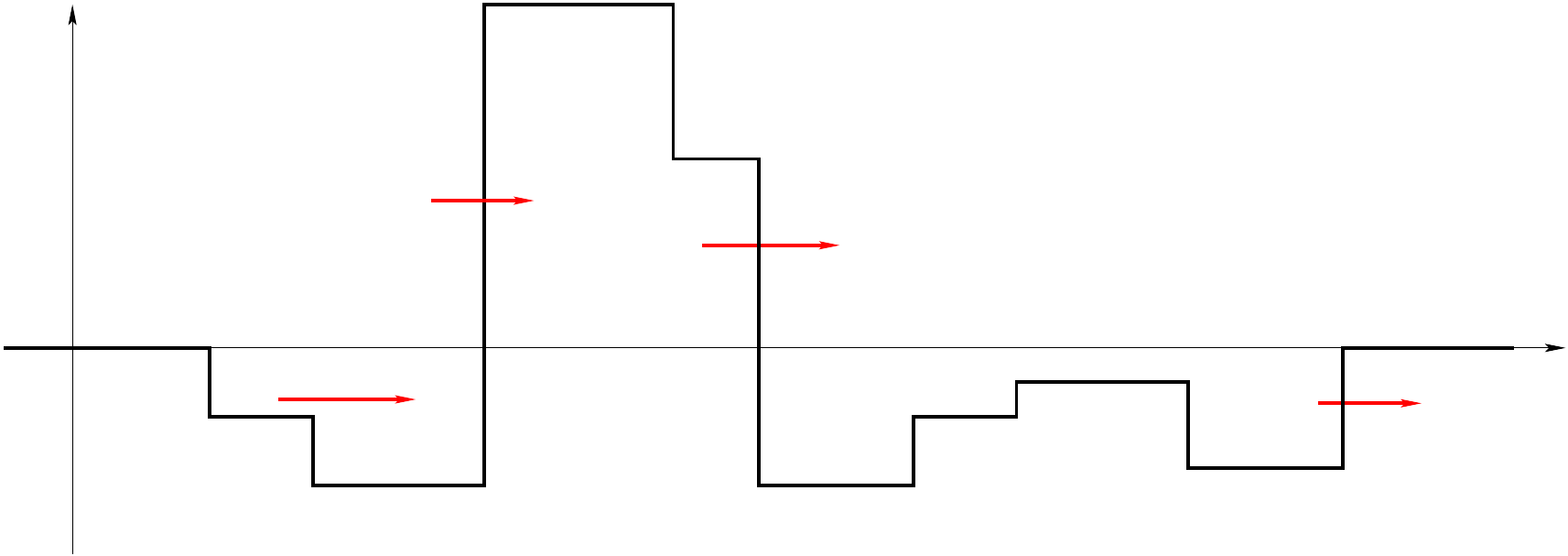}%
\end{picture}%
\setlength{\unitlength}{3947sp}%
\begin{picture}(13695,4845)(-1232,-973)
\put(4801,2864){\makebox(0,0)[lb]{\smash{\fontsize{22}{26.4}\usefont{T1}{ptm}{m}{n}{\color[rgb]{0,0,0}$u_\nu$}%
}}}
\put(5551,989){\makebox(0,0)[lb]{\smash{\fontsize{22}{26.4}\usefont{T1}{ptm}{m}{n}{\color[rgb]{0,0,0}$y_3(t)$}%
}}}
\put(10201,1064){\makebox(0,0)[lb]{\smash{\fontsize{22}{26.4}\usefont{T1}{ptm}{m}{n}{\color[rgb]{0,0,0}$y_4(t)$}%
}}}
\put(3076,2414){\makebox(0,0)[lb]{\smash{\fontsize{22}{26.4}\usefont{T1}{ptm}{m}{n}{\color[rgb]{1,0,0}$f$}%
}}}
\put(5551,1889){\makebox(0,0)[lb]{\smash{\fontsize{22}{26.4}\usefont{T1}{ptm}{m}{n}{\color[rgb]{1,0,0}$g$}%
}}}
\put(1501,539){\makebox(0,0)[lb]{\smash{\fontsize{22}{26.4}\usefont{T1}{ptm}{m}{n}{\color[rgb]{1,0,0}$g$}%
}}}
\put(10576,464){\makebox(0,0)[lb]{\smash{\fontsize{22}{26.4}\usefont{T1}{ptm}{m}{n}{\color[rgb]{1,0,0}$f$}%
}}}
\put(301,989){\makebox(0,0)[lb]{\smash{\fontsize{22}{26.4}\usefont{T1}{ptm}{m}{n}{\color[rgb]{0,0,0}$y_1(t)$}%
}}}
\put(3151,989){\makebox(0,0)[lb]{\smash{\fontsize{22}{26.4}\usefont{T1}{ptm}{m}{n}{\color[rgb]{0,0,0}$y_2(t)$}%
}}}
\end{picture}%
}
\caption{\small  Partitioning the real line into subintervals where $u_\nu$ has a  constant sign.}
\label{f:df27}
\end{center}
\end{figure}

Next, as shown in Fig.~\ref{f:df27}, we insert points $y_k(t)$, $k=0,1,2,\ldots$, such that 
%
\begin{align*}
&u_\nu(t,x)\,\le\,0\qquad \hbox{if}\quad y_{2k-1}(t)\,<x\,<y_{2k}(t)\,,
\\[2mm]
&u_\nu(t,x)\,\ge \,0\qquad \hbox{if}\quad y_{2k}(t)\,<\,x\,<\, y_{2k+1}(t)\, .
\end{align*}
%
Using the fact that $f <g$ and all jumps satisfy the Liu admissibility conditions,
we check that
\bel{dtarea}\frac{d}{ dt} \int_{y_{2k-1}(t)}^{ y_{2k}(t)}u_\nu(t,x)\, dx~\geq~0,\qquad\qquad 
\frac{d}{ dt} \int_{y_{2k}(t)}^{ y_{2k+1}(t)}u_\nu(t,x)\, dx~\leq~0,\eeq
for every $k$.  Property (i) is a direct consequence of (\ref{dtarea}).

To prove (\ref{dtarea}),  consider the first  inequality. The jump at $y_{2k-1}$ 
connects the states
$u_\nu(t,y_{2k-1}-) > u_\nu(t,y_{2k-1}+)$. Being a downward jump, it travels with speed 
determined by the flux $g_\nu$. 
Similarly, the jump at $y_{2k}$ is an upward jump and  travels with speed determined by the flux
 $f_\nu$.
Therefore
\begin{align*}
\frac{d}{ dt} \int_{y_{2k-1}(t)}^{ y_{2k}(t)}u_\nu(t,x)\, dx = & \,
- f_\nu\bigl( u_\nu(t,y_{2k}-)   \bigr)  +  g_\nu\bigl( u_\nu(t,y_{2k-1}+)   \bigr)  \\
 &\,+ \dot y_{2k}(t)\, u_\nu\bigl(t,y_{2k}-\bigr) - \dot y_{2k-1}(t)\, u_\nu\bigl(t,y_{2k-1}-\bigr)\,.
 \end{align*}
At $y_{2k-1}(t)$, since $u=0$ is an intermediate value between the left and right states, the 
Liu stability condition implies
$$
\dot y_{2k-1} (t)~=~ \frac{g_\nu\bigl(u_\nu(t,y_{2k-1}+) \bigr) -  g_\nu\bigl(u_\nu(t,y_{2k-1}-) \bigr) }{ u_\nu(t,y_{2k-1}+) - u_\nu(t,y_{2k-1}-)} ~\ge~ 
\frac{g_\nu\bigl(u_\nu(t,y_{2k-1}+) \bigr) -  g_\nu\bigl(0 \bigr) }{ u_\nu(t,y_{2k-1}+) }\,.
$$
Therefore
\bel{eq:gnu-k}
\dot y_{2k-1} (t) \, u_\nu(t,y_{2k-1}+) ~ \le~  g_\nu\bigl( u_\nu(t,y_{2k-1}+)   \bigr)  -  g_\nu\bigl(0 \bigr) \,.
\eeq
Similarly, at $y_{2k}(t)$, 
the Liu stability condition yields
$$
\dot y_{2k}(t) ~=~ \frac{f_\nu\bigl(u_\nu(t,y_{2k}-) \bigr) -  f_\nu\bigl(u_\nu(t,y_{2k}+) \bigr) }{u_\nu(t,y_{2k}-) - u_\nu(t,y_{2k}+)}~ \le~ 
\frac{f_\nu\bigl(u_\nu(t,y_{2k}-) \bigr) -  f_\nu\bigl(0 \bigr) }{ u_\nu(t,y_{2k}-) }\,,
$$
and hence
\bel{eq:fnu-k}
\dot y_{2k} (t)\, u_\nu(t,y_{2k}-)  ~\ge ~f_\nu\bigl(u_\nu(t,y_{2k}-) \bigr) -  f_\nu\bigl(0 \bigr) \,.
\eeq
We thus conclude
$$\frac{d}{ dt} \int_{y_{2k-1}(t)}^{ y_{2k}(t)}u_\nu(t,x)\, dx ~
\ge  ~ g_\nu\bigl(0 \bigr) -  f_\nu\bigl(0 \bigr)~\ge~ 0\,.$$
The proof of the second inequality in (\ref{dtarea}) is entirely similar. 

\v
{\bf 2.}  Concerning (ii), we observe that, outside interaction times, the total variation does not increase because
the value of $u_\nu$ decreases on intervals where a local max is attained, 
and increases  on intervals where a local min is attained.

On the other hand, when two or more wave-fronts interact, 
the solution to the Riemann problem does not increase the total variation.
Therefore $\TV\bigl\{ u_\nu(t,\cdot)\bigr\}$ is non-increasing in time.

The 
uniform upper and lower bounds for $u_\nu(t,x)$ follow by the above properties of 
front interactions, together with (\ref{24}) and (\ref{26}).
\v

\begin{figure}[htbp]
\begin{center}
\resizebox{0.8\textwidth}{!}{
\begin{picture}(0,0)%
\includegraphics{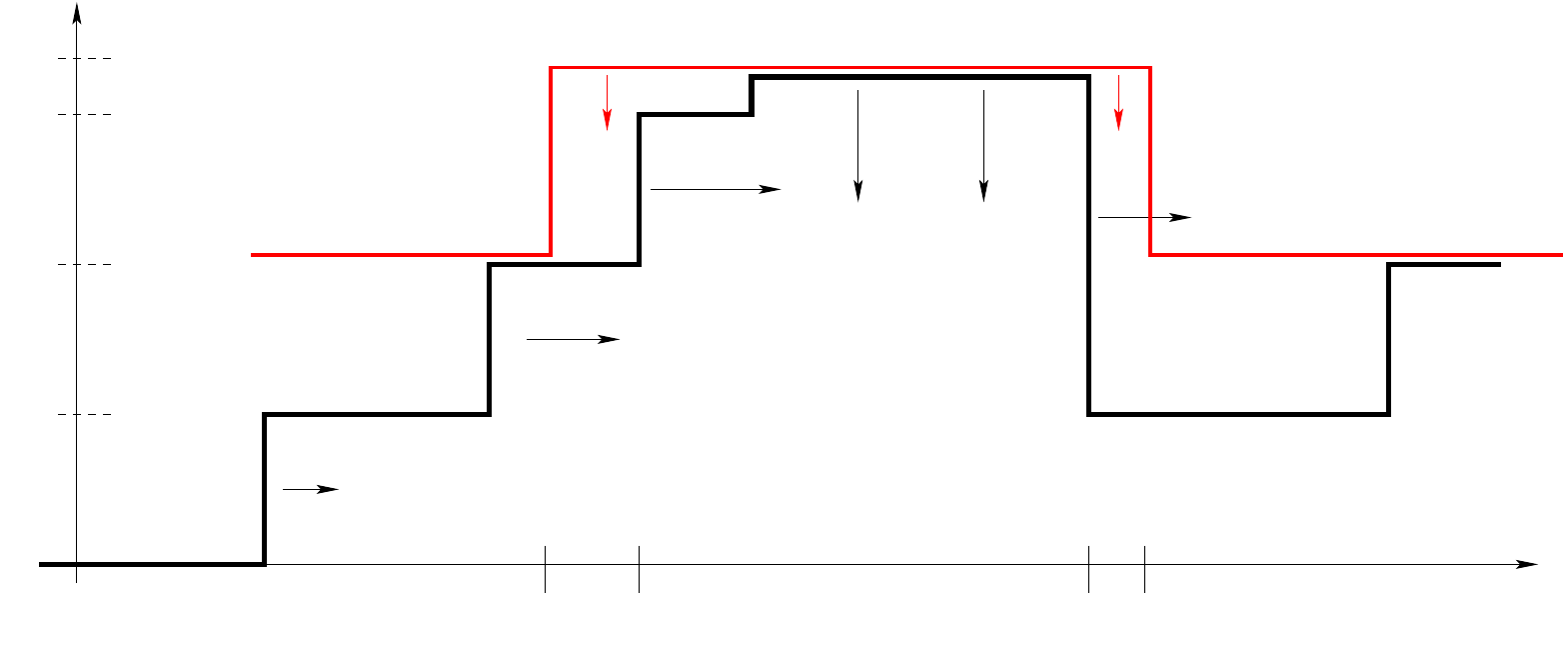}%
\end{picture}%
\setlength{\unitlength}{3947sp}%
\begin{picture}(12543,5231)(286,120)
\put(2926,3464){\makebox(0,0)[lb]{\smash{\fontsize{22}{26.4}\usefont{T1}{ptm}{m}{n}{\color[rgb]{1,0,0}$v$}%
}}}
\put(301,4664){\makebox(0,0)[lb]{\smash{\fontsize{22}{26.4}\usefont{T1}{ptm}{m}{n}{\color[rgb]{0,0,0}$u^\sharp$}%
}}}
\put(5301,299){\makebox(0,0)[lb]{\smash{\fontsize{22}{26.4}\usefont{T1}{ptm}{m}{n}{\color[rgb]{0,0,0}$x_\alpha$}%
}}}
\put(4551,299){\makebox(0,0)[lb]{\smash{\fontsize{22}{26.4}\usefont{T1}{ptm}{m}{n}{\color[rgb]{0,0,0}$y_\beta$}%
}}}
\put(8551,299){\makebox(0,0)[lb]{\smash{\fontsize{22}{26.4}\usefont{T1}{ptm}{m}{n}{\color[rgb]{0,0,0}$x_{\alpha+1}$}%
}}}
\put(9401,299){\makebox(0,0)[lb]{\smash{\fontsize{22}{26.4}\usefont{T1}{ptm}{m}{n}{\color[rgb]{0,0,0}$y_{\beta+1}$}%
}}}
\put(3001,2189){\makebox(0,0)[lb]{\smash{\fontsize{22}{26.4}\usefont{T1}{ptm}{m}{n}{\color[rgb]{0,0,0}$u$}%
}}}
\end{picture}%
}
\caption{\small  Two front tracking approximate solutions to the conservation law
with discontinuous flux (\ref{1})-(\ref{2}). If $u(t,x)\leq v(t,x)$ for all $x\in\R$, then 
$u(t',x)\leq v(t',x)$ for all $x\in\R$ and $t'>t$.   
}
\label{f:df10}
\end{center}
\end{figure}
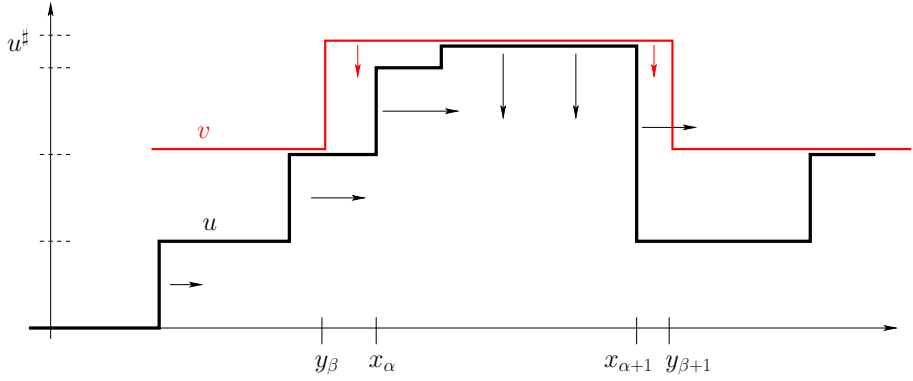

{\bf 3.} To prove the comparison property  (iii), assume that at some time $\tau$ 
one has $u_\nu(\tau,x)\leq v_\nu(\tau,x)$ for all $x\in \R$.
Consider a point $\bar x$  where both $u_\nu$ and $v_\nu$ have a jump, with 
right and left states $u^\pm, v^\pm$ respectively.
 
In a neighborhood of $\bar x$, for $t>\tau$ 
 the comparison property trivially holds if $\max\{u^-,u^+\}\le \min\{v^-,v^+\}$. 
 On the other hand, in the case of  upward jumps
$$
u^-\le v^- < u^+\le v^+\,, $$
both $u_\nu $ and $v_\nu$ are entropy solutions to a conservation law with the same flux:
\bel{clawfnu}u_t+f_\nu(u)_x~=~0.\eeq
In this case the comparison property follows by the standard argument in \cite[Sect.~6.4]{Bbook}.

In the case of downward jumps
$$ u^+\le v^+ < u^-\le v^-\,,
$$
both $u_\nu$ and $v_\nu$ are entropy solutions to a conservation law with the same flux:
$$u_t+g_\nu(u)_x~=~0.$$
Hence the comparison argument in \cite[Sect.~6.4]{Bbook} still holds.

It remains to check that the horizontal portions of the graphs of $u_\nu$, $v_\nu$
cannot cross each other at intervals where local maxima or minima are attained.
This is a consequence of the formulas (\ref{24}) and (\ref{26}). Indeed, assume that 
$u_\nu(\tau,x)\leq v_\nu(\tau,x)$ for all $x\in \R$. As shown in Fig.~\ref{f:df10},
let  $u_\nu$ attain a local maximum $\Hat u(\tau)$ on the interval  $[x_\alpha(\tau),x_{\alpha +1}(\tau)]$ and let
$v_\nu$ attain a local maximum
$\Hat v(\tau)$ on $[y_\beta(\tau),y_{\beta +1}(\tau)] \,\supseteq\,
[x_\alpha(\tau),x_{\alpha +1}(\tau)]$.
If $\Hat u(\tau)=\Hat v(\tau)\doteq u^\sharp$, then at time $t=\tau$ by (\ref{24}) and (\ref{26}) one has 
$$\frac{d}{ dt} \Hat u(t)~~=~\frac{f_\nu(u^\sharp) - g_\nu(u^\sharp)}{ 
x_{\alpha+1}(\tau)- x_{\alpha}(\tau)}~\leq~\frac{f_\nu(u^\sharp) - g_\nu(u^\sharp)}{ 
y_{\beta+1}(\tau)- y_{\beta}(\tau)}~=~\frac{d}{ dt} \Hat v(t)~\leq~0. $$
The case of a local minimum is entirely similar.
\v
{\bf 4.}
The property (iv) is a consequence of (i)  and (iii). Indeed, given any two initial data
$\bar u, \bar v$, consider  the two functions $\Tilde {w}^\pm$ defined by
$$\Tilde  w ^+(x)~=~\max\bigl\{ \bar u(x), \bar v(x)\bigr\},\qquad\qquad
\Tilde  w^-(x)~=~\min\bigl\{ \bar u(x), \bar v(x)\bigr\}.$$
Let $w_\nu^+(t,x)$, $w_\nu^-(t,x)$ be the approximate solutions corresponding to the initial data $\Tilde {w}^+$,  $\Tilde  {w}^-$ respectively.

By the comparison property (iii), for every $t\geq 0$ and $x\in \R$  we have
$$
w^-_\nu(t,x) \le u_\nu(t,x) \le w_\nu^+(t,x)\,,\qquad w_\nu^-(t,x) \le   v_\nu(t,x) \le w_\nu^+(t,x).
$$
Therefore, at any time $t\geq 0$ there holds
\bel{comp1}
\bigl\| u_\nu(t)-v_\nu(t)\bigr\|_{\L^1(\R)} ~\leq~\bigl\| w_\nu^+(t)- w_\nu^-(t)\bigr\|_{\L^1}
~=~
 \int_\R \bigl( w_\nu^+(t,x)- w_\nu^-(t,x)\bigr)\,dx.
\eeq

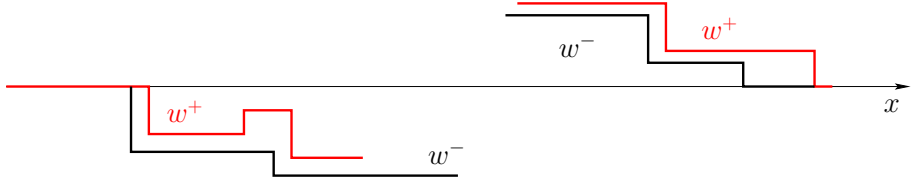
\begin{figure}[htbp]
\begin{center}
\resizebox{.8\textwidth}{!}{
\begin{picture}(0,0)%
\includegraphics{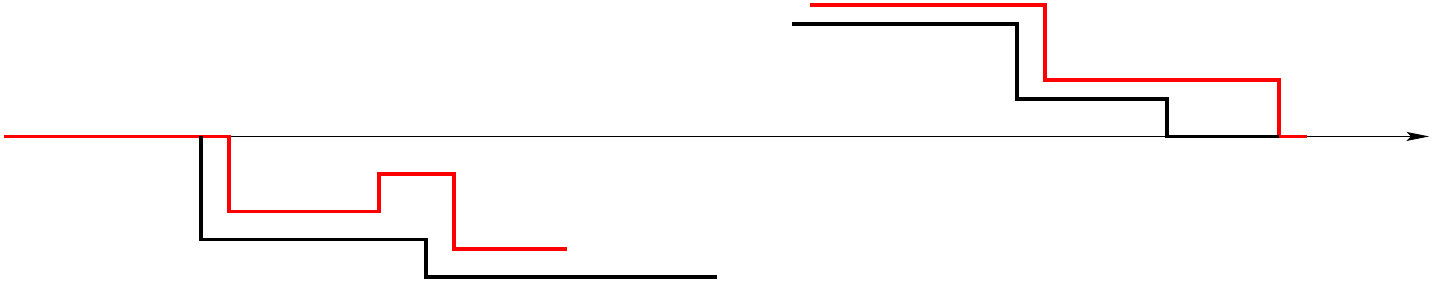}%
\end{picture}%
\setlength{\unitlength}{3947sp}%
\begin{picture}(11445,2241)(-632,-319)
\put(1426,389){\makebox(0,0)[lb]{\smash{\fontsize{22}{26.4}\usefont{T1}{ptm}{m}{n}{\color[rgb]{1,0,0}$w^+$}%
}}}
\put(10476,514){\makebox(0,0)[lb]{\smash{\fontsize{22}{26.4}\usefont{T1}{ptm}{m}{n}{\color[rgb]{0,0,0}$x$}%
}}}
\put(4726,-136){\makebox(0,0)[lb]{\smash{\fontsize{22}{26.4}\usefont{T1}{ptm}{m}{n}{\color[rgb]{0,0,0}$w^-$}%
}}}
\put(6376,1214){\makebox(0,0)[lb]{\smash{\fontsize{22}{26.4}\usefont{T1}{ptm}{m}{n}{\color[rgb]{0,0,0}$w^-$}%
}}}
\put(8176,1439){\makebox(0,0)[lb]{\smash{\fontsize{22}{26.4}\usefont{T1}{ptm}{m}{n}{\color[rgb]{1,0,0}$w^+$}%
}}}
\end{picture}%
}
\caption{\small  The two implications at (\ref{lim-})-(\ref{lim+})
are an immediate consequence of the inequality $w^-\leq w^+$.
}
\label{f:df28}
\end{center}
\end{figure}

We claim that the right hand side of (\ref{comp1}) is a non-increasing 
function of time. Indeed, since both $w_\nu^-, w_\nu^+$ satisfy a conservation 
law of the form (\ref{Fnu}), 
one has
\begin{align} 
\frac{d}{ dt} \int_\R \bigl( w_\nu^+(t,x)- w_\nu^-(t,x)\bigr)\,dx 
~=~&\lim_{x\to -\infty} \Big[ F_\nu(t,x, w_\nu^+) - F_\nu(t,x, w_\nu^-)\Big] 
\nonumber \\
&~ - ~\lim_{x\to +\infty} \Big[ F_\nu(t,x, w_\nu^+) - F_\nu(t,x, w_\nu^-)\Big].
\label{comp2} 
\end{align}
We claim that both terms on the right hand side of (\ref{comp2}) give a non-positive
contribution.
Indeed,  each of the limits $\lim_{x\to\pm \infty}  F_\nu(t,x, w_\nu^\pm)$ can be either
$f_\nu(0)$ or $g_\nu(0)$.   As shown in Fig.~\ref{f:df28},  by the assumption 
$w_\nu^-(t,x)\leq w_\nu^+(t,x)$ for all $x$
we have the implications
\bel{lim-}\lim_{x\to -\infty} F_\nu(t,x, w_\nu^+) ~=~g_\nu(0)\qquad\implies\qquad
\lim_{x\to -\infty} F_\nu(t,x, w_\nu^-) ~=~g_\nu(0),\eeq
\bel{lim+}\lim_{x\to +\infty} F_\nu(t,x, w_\nu^+) ~=~f_\nu(0)\qquad\implies\qquad
\lim_{x\to +\infty} F_\nu(t,x, w_\nu^-) ~=~f_\nu(0).\eeq
From (\ref{lim-})-(\ref{lim+}) it follows that the right hand side of (\ref{comp2})
cannot be positive.   This completes the proof of (iv).
\end{proof}

\begin{remark}\label{r:21}
{\rm Given an integer $\nu\geq 1$,  the piecewise affine 
approximate  flux functions $f_\nu,g_\nu$ are uniquely defined
according to (\ref{fnudef}).   We remark that the above construction of front tracking solutions 
can be performed starting with any piecewise constant initial $\bar v\in \L^1(\R)$,
not necessarily taking values in the set $2^{-\nu}\Z$.

In particular, given any initial data  $\bar u$ as in (\ref{idata}),  we can keep $\nu$ fixed
and construct a sequence
of piecewise constant initial data $\bar u_k:\R\mapsto 2^{-k}\Z$, converging to $\bar u$ as $k\to\infty$.  
Denote by $t\mapsto u_k(t)\doteq S_t^\nu \bar u_k$ the corresponding solutions.
By the contraction property (\ref{L1d}),
for every $t\geq 0$ we have the $\L^1$ convergence $u_k(t)\to u(t)= S^\nu_t \bar u$.
This defines a  contractive semigroup, which we denote by 
\bel{Snu}S^\nu: \L^1(\R)\times \R_+\mapsto\L^1(\R).\eeq
}\end{remark}

We now let $\nu\to\infty$ and consider the limit semigroup $S_t\bar u=\lim_{\nu\to\infty} S_t^\nu \bar u$.  

\begin{theorem}\label{t:21} 
Let $f,g$ satisfy {\bf (A1)} and consider the piecewise affine 
approximations  $f_\nu, g_\nu$, $\nu\geq 1$,
defined as in (\ref{fnudef}).  Call $S^\nu$ the corresponding semigroup at (\ref{Snu}).
Then, for any initial data $\bar u\in \L^1(\R)$,  we have the 
$\L^1$-convergence to a unique limit
\bel{St}\lim_{\nu\to\infty} 
S^\nu_t \bar u~=~u(t)~\doteq~S_t \bar u\qquad\qquad\forall t\geq 0.\eeq
The limit semigroup satisfies the contraction property:
\bel{contr2}
\bigl\|S_t \bar u - S_t\bar v\bigr\|_{\L^1}~\leq~\bigl\| \bar u - \bar v\bigr\|_{\L^1}\,.\eeq
\end{theorem}

\begin{proof}
 {\bf 1.} Consider first the case where $\bar u$ is piecewise constant.
Then for every $\nu\geq 1$ and $t>0$ the function $u_\nu(t,\cdot)=S_t^\nu\bar u$ is piecewise constant.
To prove (\ref{St}), we show that the sequence $\{u_\nu(t,\cdot)\}_{\nu\ge 1}$ is a Cauchy sequence  for every $t>0$ and hence it converges to a function $u(t,\cdot)$ as $\nu\to\infty$.

Consider any two integers $1\leq \nu\leq \mu$. We estimate
\begin{align} 
&\ds\lim_{h\to 0+} \frac{1}{ h}  \bigl\| S^\mu_h u_\nu(t) - u_\nu(t+h)\bigr\|_{\L^1}
\nonumber \\
&\qquad \ds\leq~
\Big(\|f_\mu-f_\nu\|_{W^{1,\infty}}+ \|g_\mu-g_\nu\|_{W^{1,\infty}}\Big) \cdot \TV \bigl\{ u_\nu(t,\cdot)\bigr\} 
\nonumber\\
&\qquad\qquad +~  \Big(\|f_\mu-f_\nu\|_{\L^\infty}+ \|g_\mu-g_\nu\|_{\L^\infty}\Big)\cdot
N_\nu(t),\label{ees} 
\end{align} 
where $N_\nu(t)$ is the number of intervals where $u_\nu(t,\cdot)$ attains a local maximum or a local minimum.  Observing that $f_\nu\to f$ and $g_\nu\to g$ in $W^{1,\infty}$, and the 
number of intervals of local maxima or minima decreases in time, we conclude that the 
right hand side approaches zero as $\mu,\nu\to\infty$.

The convergence of $u_\nu\to u$ now follows by means of the standard error formula \cite[Theorem 2.9]{Bbook}
\bel{numu}
    \bigl\| S^\mu_T \bar u - S^\nu_T \bar u \bigr\|_{\L^1} \le \int_0^T  
    \left\{ \liminf_{h\to 0+}  \frac{ \bigl\| S^\mu_{t+h}\bar u - S^\nu_h S^\mu_{t}\bar u
    \bigr\|_{\L^1} }{h}\right\}  dt,
\eeq
valid for all $T>0$.   This  yields (\ref{St}), for piecewise constant initial data $\bar u$.

\v
{\bf 2.} Since every $S^\nu$ is a contractive semigroup, taking the limit as $\nu\to\infty$  
we conclude
that $S$ satisfies the semigroup property $S_t S_s \bar u= S_{s+t}\bar u$, together
with the contraction property (\ref{contr2}),  for every
piecewise constant initial data $\bar u,\bar v$.
\v
{\bf 3.} It remains to prove that the convergence holds for all initial data $\bar u\in \L^1(\R)$.
For this purpose, let $\epsilon>0$ be given. Choose a piecewise constant $\bar v$ 
such that $\|\bar v-\bar u\|_{\L^1}\leq \epsilon$. 
Since every $S^\mu$ is a contractive semigroup, for every $t\geq 0$
we have
\begin{align*}
\limsup_{\mu,\nu\to\infty} \|S^\mu_t \bar u - S^\nu_t\bar u\|_{\L^1}~&\leq~
 \limsup_{\mu\to\infty} \|S^\mu_t \bar u - S^\mu_t\bar v\|_{\L^1}
 +  \limsup_{\mu,\nu\to\infty} \|S^\mu_t \bar v - S^\nu_t\bar v\|_{\L^1} \\
& \qquad  +\limsup_{\nu\to\infty} \|S^\nu_t \bar v - S^\nu_t\bar u\|_{\L^1}\\
 \qquad& \leq ~\epsilon + 0 + \epsilon.
\end{align*}
%
%
Since $\epsilon>0$ was arbitrary, this proves that the sequence $S^\nu_t\bar u$ is Cauchy and hence converges to a unique limit. 
\end{proof}

\begin{remark}\label{r:22}
{\rm Entirely similar results hold in the spatially periodic case.  For a fixed $\nu\geq 1$,
consider the piecewise affine 
approximations  $f_\nu, g_\nu$,
defined as in (\ref{fnudef}). These determine a contractive semigroup $S^\nu: \L^1_{per}(\R)\times \R_+\mapsto \L^1_{per}(\R)$.
As $\nu\to\infty$,
for any initial data $\bar u\in \L^1_{per}(\R)$, we have the convergence to a unique limit (\ref{St}).
The limit semigroup satisfies the contraction property:
\bel{contr22}
\bigl\|S_t \bar u - S_t\bar v\bigr\|_{\L^1_{per}}~\leq~\bigl\| \bar u - \bar v\bigr\|_{\L^1_{per}}\,.\eeq
}\end{remark}

We conclude this section with a standard error bound valid for front tracking approximations.


\begin{proposition}
\label{p:22} 
Let $\bar u$ be a piecewise constant initial data, with compact support.
Call $t\mapsto u(t)=S_t\bar u$ the exact semigroup solution, and let $u_\nu(t,\cdot)=S_t^\nu\bar u$.
Then, for any $t\ge 0$ we have
\bel{err1}\bigl\| u_\nu(t) - S_t \bar u\bigr\|_{\L^1_{per}}~=~\O(1)\cdot \TV\{\bar u\} \cdot 2^{-\nu} t.\eeq
\end{proposition}

\begin{proof} 
Since the semigroup $S$ is contractive, the 
standard error formula \cite{Bbook} yields  
$$
    \bigl\| u_\nu(t) - S_t \bar u \bigr\|_{\L^1_{per}} \le \int_0^t
    \left\{ \liminf_{h\to 0+}  \frac{ \bigl\| u_\nu(\tau+h) - S_h u_\nu(\tau)
    \bigr\|_{\L^1_{per}} }{h}\right\}  d\tau\,.$$
For every $\tau\in [0,t]$\,, the instantaneous error rate is bounded by
\begin{align*}
\liminf_{h\to 0+} & \frac{\bigl\|u_\nu (\tau+h) - S_h u_\nu(\tau)\bigr\|_{\L^1_{per}}}{ h} \\
&\leq~\TV\{\bar u\}\cdot \sup_v~\Big(\bigl| f'(v)-f'_\nu(v)\bigr|+\bigl| g'(v)-g'_\nu(v)\bigr|\Big)  \\
&=~\O(1)\cdot \TV\{\bar u\}\cdot 2^{-\nu}\,.
\end{align*}
Combining the previous two estimates, we obtain (\ref{err1}).
\end{proof}

\section{Properties of solutions}
\label{sec:3}
\setcounter{equation}{0}

By further analyzing the behavior of front tracking approximations constructed in the previous section, 
several properties of semigroup solutions can be established.

\begin{proposition}\label{p:40}
Let $f,g$ satisfy {\bf (A1)}.  Then for any initial data $\bar u\in\L^1(\R)$, the 
semigroup solution satisfies
\begin{align} 
\bigl\| S_\tau\bar u\bigr\|_{\L^\infty}~<~\infty\quad\quad&\forall \tau>0, \label{Sp1} \\
S_\tau\bar u =0 \quad\qquad &\forall \tau ~\geq~ \frac 1 {c_0} {\int_{\R} |\bar u(x)|\, dx}\,. \label{Sp2}
\end{align}
\end{proposition}

\begin{proof}
{\bf 1.}
It suffices to prove that the above properties are satisfied by all front tracking approximations $u_\nu$,
uniformly w.r.t.~$\nu$.  
By earlier analysis, the norm $\bigl\|u_\nu(t,\cdot)\bigr\|_{\L^\infty}$ is non-increasing in time.
To prove it for any initial data $\bar u\in\L^1(\R)$, 
let $\tau>0$ be given. Choose $M>0$ large enough so that 
\bel{ic0}\int_{\left|\bar u(x)\right|>M} |\bar u(x)|\, dx~<~c_0 \tau.\eeq

Let $0<t<\tau$. Consider any interval $\bigl[a(t), b(t)\bigr]$ which is a connected component of the set 
$$I(t)~\doteq~\Big\{ x\in \R\,;~u_\nu(t,x)>M\Big\}.$$
At $x=a(t)$, either $u_\nu$ is continuous with $u_\nu\bigl(t,a(t)\bigr)=M$ or an upward jump occurs, 
with $u_\nu\bigl(t,a(t)-\bigr)\le M$.
In both cases the 
flux is $f_\nu$. 
At $x=b(t)$ the behavior of $u_\nu$ is analogous, with 
flux $g_\nu$.
As in the proof of (\ref{dtarea}), recalling that all  jumps are Liu-admissible 
and that $g_\nu(v) -f_\nu (v)\geq c_0$ for all $v\in\R$, we have
%
$$\frac{d}{ dt} \int_{a(t)}^{b(t)} u_\nu(t,x)\, dx
~\leq~{-\bigl(g_\nu(M) -f_\nu (M)\bigr)}
~\leq~- c_0\,.$$
By (\ref{ic0}), if  $\meas \bigl(I(\tau)\bigr)>0$, we get a contradiction. 
{A similar analysis holds for the set $\Big\{ x\in \R\,;~u_\nu(t,x)<-M\Big\}$.} 
Since this is true for
any front tracking approximation, taking a limit we obtain (\ref{Sp1}).

{\bf 2.} To prove (\ref{Sp2}), consider any $\tilde\ve>0$ arbitrarily small. 
For any $\tau>0$ that satisfies (\ref{ic0}),
the above argument can be repeated, leading to 
$$\meas \Big(\bigl\{ x\in\R\,;~|u_\nu(\tau,x)|>\tilde\ve \bigr\}\Big)=0 \qquad\hbox{whenever} \quad \tau\,\geq\,  \tau_{\tilde\ve}\,\doteq
\,\frac 1 {c_0} {\int_{\{|\bar u(x)|>\tilde\ve \}} |\bar u(x)|\, dx}\,.$$ 
Letting $\nu\to\infty$ one obtains 
$\bigl\|S_\tau\bar u \bigr\|_{\L^\infty}\leq \tilde\ve$  for any $\tau\geq\tau_{\tilde\ve}$.
In particular, taking $\tilde\ve=0$ we obtain (\ref{Sp2}). 
\end{proof}

\begin{remark} {\rm  The decay properties stated in Proposition~\ref{p:40} have nothing to do with 
the genuine nonlinearity of the fluxes.   In particular, they remain valid  even if
$f(u)=0$, $g(u)=c_0>0$ for all $u\in\R$.}  \end{remark} 

\begin{proposition}\label{p:42}
Let $f,g$ satisfy {\bf (A1)}.   Then for any initial data $\bar u\in\L^1(\R)$, the 
semigroup solution has locally bounded variation:
\bel{Sp3} \TV\bigl\{S_\tau\bar u\,;~~[a,b]\bigr\}~<~\infty\eeq
for any interval $[a,b]$ and every $\tau>0$.
\end{proposition}

\begin{proof}
Let $\bar u\in \L^1(\R)$ be given. It suffices to prove that the total 
variation of front tracking approximations $u_\nu$ satisfies a uniform bound,
independent of $\nu$.

{\bf 1.} 
We have $S_\tau\bar u =S_{\tau/2} \bar v$, with $\bar v=S_{\tau/2}\bar u\in \L^1(\R)\cap \L^\infty(\R)$, as stated in Proposition~\ref{p:40}.  
Therefore, without loss of generality we can assume that 
$\|\bar u\|_{L^\infty}\leq M$ for some constant $M$. 
Let $\bar\lambda>0$
be an upper bound on all wave speeds, as long as $|u|\leq M$.

{\bf 2.} Let $\tau>0$ and an interval $[a,b]$ be given.
Consider the intervals
$$I(t)~\doteq ~\bigl[ a-\bar\lambda(\tau-t)\,,~b+\bar\lambda(\tau-t)\bigr].$$

Consider any front tracking approximation $u_\nu:\R_+\times\R\mapsto [-M,M]$.
Call $N(t)$ the number of intervals, where $u_\nu(t,\cdot)$ 
attains a local max or a local min, which are entirely contained inside $I(t)$.
Then
\bel{TVbd}\TV\bigl\{ u_\nu(t,\cdot)\,;~I(t)\bigr\}~\leq~2M\cdot \bigl(1+N(t)\bigr).\eeq
We observe that $N(t)$ is a non-increasing function of time. Hence $N(\tau)\leq N(t)$.

Consider the intervals $\bigl[ a_k(t), b_k(t)\bigr]$ where $u_\nu(t)$ attains a local maximum
$M_k(t)$.   
By construction, we have
$$\frac{d}{ dt} M_k(t)\cdot \bigl(b_k(t)- a_k(t)\bigr)~\leq~-c_0\,.$$
Therefore, since at every time $t\in \,[0,\tau]$ there are $N(t)\geq N(\tau)$ of these intervals, 
we have
$$\frac{d}{ dt} \sum_k M_k(t) \bigl(b_k(t)- a_k(t)\bigr)~\leq~- N(t) c_0 ~\leq ~- N(\tau) \, c_0\,.$$
Observing that 
\begin{align*}
\sum_k M_k(0) \bigl(b_k(0)- a_k(0)\bigr)~&\leq~ M(b-a+2\hat\lambda \tau),\\
  \sum_k M_k(\tau) \bigl(b_k(\tau)- a_k(\tau)\bigr)~&\geq~-M(b-a),
\end{align*} 
we conclude
\bel{mmi}\tau\cdot N(\tau) c_0~\leq~2M(b-a+\hat\lambda\tau),
\qquad\qquad N(\tau)~\leq~\frac{2M(b-a+\hat\lambda\tau)}{ \tau c_0}\,.\eeq
By (\ref{TVbd}), this yields a bound on the total variation, independent of $\nu$.
\end{proof}
\medskip

From the uniform bound (\ref{mmi}), valid for any front tracking approximation and hence also for the limit solution,
we obtain 

\begin{proposition}\label{p:33}
Let $f,g$ satisfy {\bf (A1)}.   Then for any initial data $\bar u\in\L^1(\R)$, the 
semigroup trajectory $u(t,\cdot)= S_t\bar u$ provides a weak solution to the Cauchy problem (\ref{1}), (\ref{idata}),
in the sense of Definition~\ref{def:11}.\end{proposition}

\begin{proof}
Let an initial data $\bar u$ be given and $u_\nu$ be a front tracking approximation,
constructed as in Section~\ref{sec:2}.
Given any interval $[a,b]$ and any $t_0>0$, by (\ref{mmi}) 
there exists an integer $\ov N$ independent of $\nu$ such that the following holds.
For every $t\geq t_0$ there exists intervals
$\bigl[ a^\nu_k(t), b^\nu_k(t)\bigr]$, $1\leq k\leq N(t)\leq \ov N$ such that:
\begi
\item[(i)] The function $u_\nu(t,\cdot)$ is constant on each of these intervals.
\item[(ii)] On each connected component of the complement $\R\setminus \bigcup_k \bigl[ a^\nu_k(t), b^\nu_k(t)\bigr]$
the function $u_\nu(t,\cdot)$ is monotone (increasing or decreasing).
\endi
 In other words, $\bigl[ a^\nu_k(t), b^\nu_k(t)\bigr]$ are intervals where $u_\nu$ attains its local maximum/minimum.

From the construction of front tracking solutions $u_\nu(t,\cdot)$, it is clear that for $t\geq t_0>0$ 
the endpoints $a_k,b_k$ are 
Lipschitz function of time, and the number $N(t)$ of intervals is non-increasing. Letting $\nu\to\infty$ and 
taking the limit of a suitable subsequence, we conclude that the same holds for the limit solution $u(t,\cdot)$.
In particular, $a^\nu_k(t)\to a_k(t)$, $b_k^\nu(t)\to b_k(t)$ on the time interval where all these functions are defined.
%

Starting with the identities
\bel{wksol}
 \dint \left\{u_\nu\,\vp_t + \Big[\theta_\nu f(u_\nu) + \bigl(1-\theta_\nu \bigr)g(u_\nu)\Big] \vp_x\right\}\, dx\,dt~=~0\eeq
 where $\theta_\nu$ are the piecewise affine functions in (\ref{th-n-1})-(\ref{th-n-2}) and letting
 $\nu\to\infty$,  we thus achieve (\ref{18})-(\ref{weaksol}).
\end{proof}

\begin{figure}[htbp]
\begin{center}
\resizebox{0.9\textwidth}{!}{
\begin{picture}(0,0)%
\includegraphics{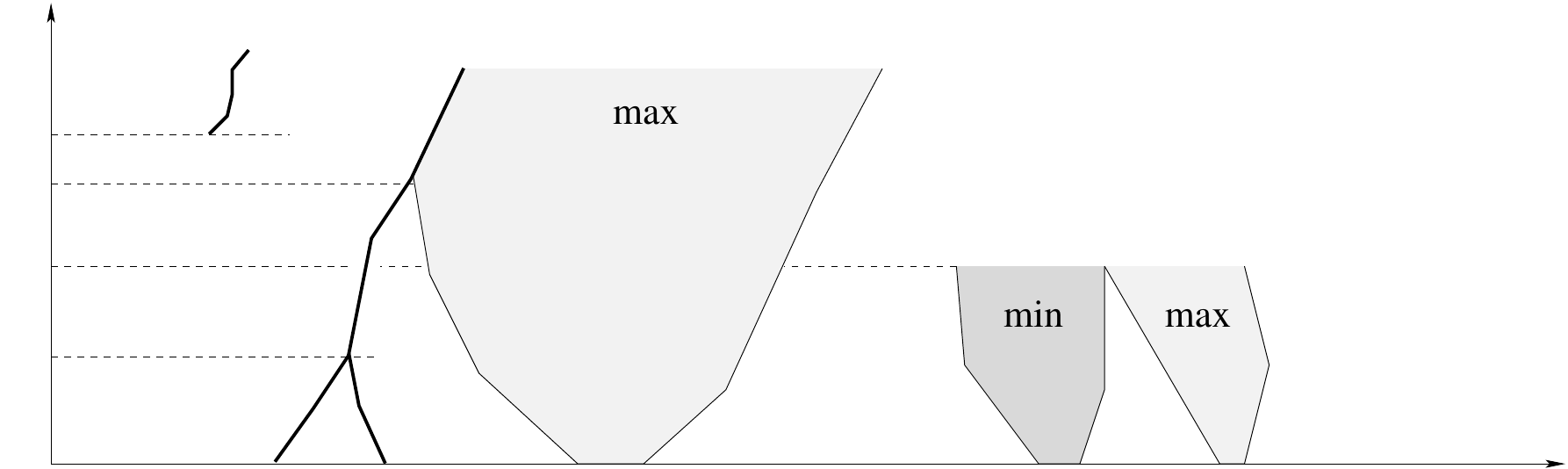}%
\end{picture}%
\setlength{\unitlength}{3947sp}%
\begin{picture}(14277,4254)(-1064,-2203)
\put(-1049,-1336){\makebox(0,0)[lb]{\smash{\fontsize{22}{26.4}\usefont{T1}{ptm}{m}{n}{\color[rgb]{0,0,0}$t_1$}%
}}}
\put(-1049,239){\makebox(0,0)[lb]{\smash{\fontsize{22}{26.4}\usefont{T1}{ptm}{m}{n}{\color[rgb]{0,0,0}$t_3$}%
}}}
\put(-1049,764){\makebox(0,0)[lb]{\smash{\fontsize{22}{26.4}\usefont{T1}{ptm}{m}{n}{\color[rgb]{0,0,0}$t_4$}%
}}}
\put(-1049,-511){\makebox(0,0)[lb]{\smash{\fontsize{22}{26.4}\usefont{T1}{ptm}{m}{n}{\color[rgb]{0,0,0}$t_2$}%
}}}
\put(12601,-2086){\makebox(0,0)[lb]{\smash{\fontsize{22}{26.4}\usefont{T1}{ptm}{m}{n}{\color[rgb]{0,0,0}$x$}%
}}}
\put(826,-1711){\makebox(0,0)[lb]{\smash{\fontsize{22}{26.4}\usefont{T1}{ptm}{m}{n}{\color[rgb]{0,0,0}$f$}%
}}}
\put(6751,-1411){\makebox(0,0)[lb]{\smash{\fontsize{22}{26.4}\usefont{T1}{ptm}{m}{n}{\color[rgb]{0,0,0}$g$}%
}}}
\put(9151,-1861){\makebox(0,0)[lb]{\smash{\fontsize{22}{26.4}\usefont{T1}{ptm}{m}{n}{\color[rgb]{0,0,0}$f$}%
}}}
\put(11176,-811){\makebox(0,0)[lb]{\smash{\fontsize{22}{26.4}\usefont{T1}{ptm}{m}{n}{\color[rgb]{0,0,0}$g$}%
}}}
\put(9451,1139){\makebox(0,0)[lb]{\smash{\fontsize{22}{26.4}\usefont{T1}{ptm}{m}{n}{\color[rgb]{0,0,0}$g$}%
}}}
\put(2626,-1636){\makebox(0,0)[lb]{\smash{\fontsize{22}{26.4}\usefont{T1}{ptm}{m}{n}{\color[rgb]{0,0,0}$f$}%
}}}
\put(2101,1064){\makebox(0,0)[lb]{\smash{\fontsize{22}{26.4}\usefont{T1}{ptm}{m}{n}{\color[rgb]{0,0,0}$f$}%
}}}
\put(301,1064){\makebox(0,0)[lb]{\smash{\fontsize{22}{26.4}\usefont{T1}{ptm}{m}{n}{\color[rgb]{0,0,0}$f$}%
}}}
\put(-524,1664){\makebox(0,0)[lb]{\smash{\fontsize{22}{26.4}\usefont{T1}{ptm}{m}{n}{\color[rgb]{0,0,0}$t$}%
}}}
\end{picture}%
}
\caption{\small A front tracking approximation $u_\nu(t,\cdot)= S^\nu_t \bar u$. 
Here we marked the regions where the flux is constantly equal to $f$ of to $g$.
For each $t\geq 0$ there are finitely many intervals where $u_\nu(t,\cdot)$ 
attains a local maximum or minimum (shaded areas). }
\label{f:df43}
\end{center}
\end{figure}
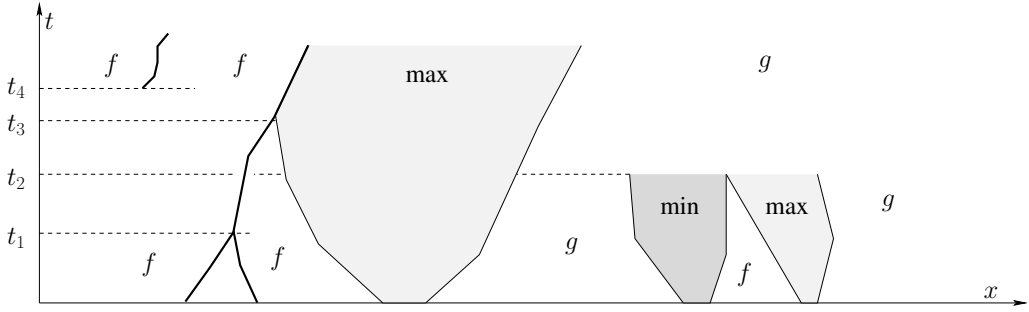

{\color{black}
In the spatially periodic case, the identity (\ref{Sp2}) is replaced by

\begin{proposition}\label{p:34}
Let $f,g$ satisfy the assumptions {\bf (A1)}, and consider  a piecewise constant, spatially periodic initial data
$\bar u\in \L^1_{per}(\R)$. For a given $\nu\geq 1$, let  $u_\nu(t,\cdot) = S_t^\nu\bar u$ be 
the corresponding front tracking solution. 
 Then, for all $t\geq \|\bar u\|_{\L^\infty}/c_0$ one has
\bel{uave} u_\nu(t,x)~=~\bar u^{average} ~\doteq~\int_0^1\bar u(x)\, dx.\eeq
\end{proposition}

\begin{proof}  By assumption, $u_\nu(t,\cdot)$ has period 1.   By \eqref{24} and \eqref{26}  it follows
\begin{align*}
\frac{d}{dt} \Hat u_{k}(t)~\leq~- c_0, \qquad & \mbox{when $\Hat u_{k}(t)$ is a local maximum,}  
\\
\frac{d}{dt} \Hat u_{j}(t)~\geq~ c_0, \qquad & \mbox{when $\Hat u_{j}(t)$ is a local minimum.}  
\end{align*}
This implies that,
for $t> \|\bar u\|_{\L^\infty}/c_0$, 
we have $\Hat u_{k}(t) < \Hat u_{j}(t)$  for any such couple  of indices $k,j$. This yields
a contradiction. We conclude that  $u_\nu(t,\cdot)$ is constant.
Since the equation is in conservation form, this constant value coincides with the average value of the initial data, as in (\ref{uave}).
\end{proof}
}
\v
In the remaining sections we will show that the semigroup trajectories, obtained 
as limits of front tracking 
approximations, are the unique limits of vanishing viscosity solutions (\ref{3}). 
On the other hand, 
the following example shows that the Definition~\ref{def:11}  is not enough to single out a unique weak solution,
even if additional entropy conditions are imposed.

\begin{example}\label{ex:31} {\rm
Consider the two fluxes
$$f(u)~=~\frac{u^2}{ 2}\,,\qquad\qquad g(u)~=~\frac{u^2}{ 2}+1\,.$$
Let $u=u(t,x)$ be the unique entropy weak solution to Burgers' equation
\bel{BuCP}u_t+f(u)_x~=~0,\qquad\qquad u(0,x)=\left\{ \bega{cl} e^x\quad &\hbox{if}\quad x<0,\\[1mm]
-e^{-x}\quad &\hbox{if}\quad x>0.\enda\right.
\eeq
This solution can be easily constructed by characteristics.  We observe that,  for every $t\geq 0$, the function 
$u(t,\cdot)\in \L^1(\R)$ is piecewise smooth, with a single downward jump at the origin.
Moreover, 
$$u(t,-x)~= ~- u(t,x)\qquad\forall t\geq 0, ~x>0,$$
$$u_x(t,x)~>~0\qquad\forall t\geq 0,~ x\not= 0.$$
In this case, setting 
$$\theta(t,x)\,\doteq\,\sign \,u_x(t,x)\, = \, 1\qquad\hbox{for a.e.}~t,x,$$
since $u$ is an entropy weak solution to Burgers' equation, it satisfies 
$$
\dint \Big\{u \phi_t + \bigl[ \theta f(u) + (1-\theta) g(u)\bigr] \phi_x\Big\}\, dxdt~=~0
$$
for every test function $\phi\in \C^1_c$, and 
$$
\dint \Big\{|u-k| \phi_t + \sign(u-k) \bigl[ \theta \bigl(f(u) - f(k)\bigr)+ (1-\theta) \bigl(g(u)-
g(k)\bigr)\bigr] \phi_x\Big\}\, dxdt~\geq~0
$$
for every $k\in \R$, and every non-negative test function  $\phi\in \C^1_c$.

However,  the solution $u$ of (\ref{BuCP}) is NOT a trajectory of our semigroup, obtained  as limit of 
the approximations (\ref{3}).  This can be checked by observing that, for every $\tau>0$, 
the function  $u(\tau,\cdot)$ is
strictly increasing on the two open domains where $-\infty<x<0$ and $0<x<\infty$, respectively.
On the other hand, by (\ref{eqM2}) 
the function $S_\tau \bar u$ is constantly equal to its maximum value on an interval of length $\geq 2C\tau$.   
These two properties are clearly incompatible.
}
\end{example}

\section{Viscous approximations}
\label{sec:4}
\setcounter{equation}{0}

Assuming  $f,g$ satisfy {\bf (A1)}, in this section we focus on the spatially periodic case.
Having constructed the semigroup $S:\L^1_{per}\times \R_+\mapsto \L^1_{per}$ according to Theorem~\ref{t:21} 
and Remark~\ref{r:22}, we  wish to characterize its trajectories as the unique limits 
of viscous approximations.
{ Toward this goal, for any $\ve, \delta>0$ we consider the  uniformly parabolic equation (\ref{3}).}
%
For any initial data $\bar u\in \L^1_{per}$, the Cauchy problem (\ref{3})-(\ref{idata})
has a unique solution 
$$t~\mapsto~ u^{\ve,\delta}(t,\cdot)~\doteq ~S^{\ve,\delta}_t\bar u.$$   
Note that we also have the comparison principle
\bel{comp}
\bar u(x)\leq \bar v(x)\quad\hbox{for all}~~ x\in\R\quad\implies\quad u^{\ve,\delta}(t,x)\leq v^{\ve,\delta}(t,x)\quad\hbox{for all}~~ x\in\R,~t\geq 0.\eeq
Writing (\ref{3}) in the { conservation form}
\bel{30} u_t +  \Big[\theta_\ve(u_x) f(u) + \bigl(1-\theta_\ve(u_x) \bigr)g(u)\Big]_x
~=~\delta u_{xx}\,,
\eeq
we thus obtain the contraction property
\bel{cont}  \bigl\| S^{\ve,\delta}_t\bar u -  S^{\ve,\delta}_t\bar v\bigr\|_{\L^1_{per}}~\leq~\|\bar u-\bar v\|_{\L^1_{per}}\,,\eeq
for all $\bar u,\bar v$, and all $\ve,\delta>0$, $t\geq 0$.

In turn, for any initial data $\bar u\in BV$, this implies that the total variation over one period does not increase in time:
\begin{align*}
\TV\{S^{\ve,\delta}_t \bar u\}~& \dot=~\lim_{h\to 0+} \frac{1}{ h}\int_0^1 \bigl| u^{\ve,\delta}(t, x+h) -
 u^{\ve,\delta}(t,x)
\bigr|\, dx\\
& \leq~\lim_{h\to 0} \frac{1}{ h} \int_0^1 \bigl| \bar u(x+h) -
\bar u(x)
\bigr|\, dx~=~\TV\{\bar u\}.
\end{align*}

The next two sections will be devoted to a proof of next Theorem.

\begin{theorem}
\label{t:41}
 Let $f,g$ satisfy {\bf (A1)}.
For every $\bar u\in \L^1_{per}(\R)$ and $t\geq 0$, 
the  limit 
\bel{HatS}\lim_{\ve,\delta\to 0+} S^{\ve,\delta}_t\bar u~=~S_t\bar u\eeq
exists in $\L^1_{per}(\R)$.   The map $(\bar u, t)\mapsto  S_t \bar u$ is precisely
the semigroup constructed in (\ref{St}) as limit of front tracking approximations.
\end{theorem}

Toward a proof, the heart of the matter is to show that any front tracking solution $u_\nu$ can be
closely approximated by
solutions to  (\ref{3}), for $\ve,\delta>0$ small enough.

\begin{proposition}\label{p:41} Assume that $f,g$ satisfy {\bf (A1)}. Given
any piecewise constant initial data $\bar u\in \L^1_{per}(\R)$, a time interval $[0,T]$ and  $\ve_0>0$,
there exist $\nu\geq 1$ and $\bar \ve, \bar\delta>0$ such that the following holds.
Calling $u_\nu(t,\cdot) = S^\nu_t\bar u$ the corresponding front tracking approximations, and 
$u(t,\cdot) = S_t\bar u$ the exact solution, one has
\bel{SSnu} \bigl\| S^\nu_t \bar u - S_t\bar u\bigr\|_{\L^1_{per}} ~\leq~\ve_0
\qquad\quad \forall~ t\in [0,T]\,.\eeq
Moreover, 
for every $\ve\in \,]0,\bar \ve]$ and $\delta\in \,]0,\bar \delta]$ one can construct an approximate solution $w:[0,T]\mapsto \L^1_{per}(\R)$ of (\ref{3})-(\ref{idata}) with the following properties.
\begi
\item[(i)]
For every $t\in [0,T]$ there holds
\bel{wu1}
\int_0^1 \bigl| w(t,x) - u_\nu(t,x)\bigr|\, dx~\leq~\ve_0.\eeq
\item[(ii)] There exist finitely many points $0=t_0<t_1<\cdots<t_N=T$ such that the restriction of 
$w$ to each open subinterval $\,]t_{j-1}, t_j[\,\times\R$ has $\C^{1,1}$ regularity (i.e., it is continuously differentiable with 
Lipschitz derivatives). For $t\notin \{t_j\,;~j=0,\ldots,N\}$ one has
\bel{wer}\int_0^T \int_0^1 \bigg| 
w_t +  \Big[\theta_\ve(w_x) f(w) + \bigl(1-\theta_\ve(w_x) \bigr)g(w)\Big]_x
-\delta w_{xx}\bigg|\, dx\, dt ~\leq~\ve_0\,.\eeq
\item[(iii)] At the times $t_j$ where the map $t\mapsto w(t,\cdot)$ is discontinuous, the jumps satisfy
\bel{wu2}
\sum_{j=0}^N ~ \int_0^1 \Big| w(t_j+, x) - w(t_j-, x)\Big|\, dx ~\leq~\ve_0\,.\eeq
We used here the convention $w(t_0-,x)\doteq \bar u(x)$.
\endi
\end{proposition}

Sections~\ref{sec:5} and \ref{sec:6} will be devoted to a proof of Proposition~\ref{p:41}.
We conclude this section by showing how Theorems~\ref{t:21} and \ref{t:41} both follow
from the above proposition.
\v

\begin{proof} 
We prove that {Proposition~\ref{p:41} implies Theorem~\ref{t:41}.}

Let a piecewise constant initial data $\bar u\in \L^1_{per}(\R)$ be given, together with $T, \ve_0>0$. 
Choose $\nu\geq 1$ according to Proposition~\ref{p:41}.

We observe that the parabolic equation (\ref{30}) is in conservation form  and
preserves ordering. Namely (see for example  \cite{Lieb}),  for any pair of solutions $u,v$ one has
$$u(0,x)\leq v(0,x)\qquad\implies\qquad u(t,x)\leq v(t,x)\qquad\forall t\geq 0.$$
It follows that (\ref{30})
generates a contractive semigroup $S^{\ve,\delta}$ on $\L^1_{per}(\R)$.
A standard error formula~\cite{Bbook} now yields
\begin{align}
\bigl\|w(\tau,\cdot)& - S^{\ve,\delta}_\tau \bar u\bigr\|_{\L^1_{per}}
\nonumber\\
&\leq~\int_0^\tau\int_0^1 \bigg| 
w_t +  \Big[\theta_\ve(w_x) f(w) + \bigl(1-\theta_\ve(w_x) \bigr)g(w)\Big]_x
-\delta w_{xx}\bigg|\, dx\, dt 
\nonumber \\
&\ds\qquad\ds + \sum_{t_j\in [0,\tau]}   \bigl\| w(t_j+, x) - w(t_j-, x)\bigr\|_{\L^1_{per}}
\nonumber\\
&\leq~\ve_0 + \ve_0\,. \label{errw}
\end{align}
%
%
Combining the estimate (\ref{errw}) with (\ref{wu1}) and (\ref{SSnu}), for every $t\in [0,T]$ we obtain
$$\bega{rl} \|S_t^{\ve,\delta}\bar u - S_t\bar u\|_{\L^1_{per}} 
&\leq~\bigl\|S^{\ve,\delta}_t \bar u-w(t,\cdot)\bigr\|_{\L^1_{per}}+\bigl\| w(\tau,\cdot)
- S^\nu_t \bar u\bigr\|_{\L^1_{per}}+\bigl\| 
S^\nu_t \bar u-S_t \bar u\bigr\|_{\L^1_{per}}\\[2mm]
&\leq~\ve_0 + \ve_0 + \ve_0\,.
\enda
$$
Since $\ve_0>0$ was arbitrary, this implies the convergence (\ref{HatS})
for every piecewise constant initial data $\bar u\in \L^1_{per}$.   Since 
piecewise constant functions are dense  and all semigroups $S^{\ve,\delta}$, $S$
are contractive, we conclude that the limit (\ref{HatS}) holds for all initial data $\bar u\in \L^1_{per}$.
\end{proof}

The statements in Theorem~\ref{t:21} are now an immediate consequence of Theorem~\ref{t:41}
and Proposition~\ref{p:33}.

\section{Special families of viscous solutions}
\label{sec:5}
\setcounter{equation}{0}

The piecewise smooth approximations $w$ satisfying (\ref{wu1})-(\ref{wu2}) 
will be obtained by gluing together different types of smooth solutions,
which we now describe.

\subsection{Viscous shocks}
Given the conservation law
$$u_t+ f(u)_x~=~0,$$
we recall that a shock solution of the form
\bel{RHL}u(t,x) ~=~\left\{\bega{rl} u^-\quad\hbox{if}\quad x<\lambda t,\\[1mm]
u^+\quad\hbox{if}\quad x>\lambda t,\enda\right.
\qquad\qquad  \lambda\,=\, \frac{f(u^+)-f(u^-)}{ u^+-u^-}\,,\eeq
is Liu-admissible if
\bel{Liu}
\frac{f(u^+)-f(u^-)}{ u^+-u^-}~\leq ~\frac{f(u^*)-f(u^-)}{ u^*-u^-}\,.
\eeq
for every intermediate state $u^*$ in the open interval with endpoints $u^-,u^+$.

As shown in \cite{L76}, the Liu admissibility condition (\ref{Liu}) is closely related to 
the existence of 
of  traveling profiles $u(t,x) = U(x-\lambda t)$ for the viscous equation
\bel{viscf}u_t + f(u)_x~=~\delta u_{xx}\eeq
which converge to the shock (\ref{RHL}) as $\delta\to 0$.
These viscous traveling waves are solutions to the ODE
$$\delta U'~=~f(U) -\lambda U + C,$$
where the constant $C$ is determined by
\bel{CC}C~=~\lambda u^+-f(u^+)~=~\lambda u^- - f(u^-).\eeq

When one wishes to approximate a Liu-admissible  jump for a front tracking solution of the conservation 
law (\ref{clawfnu})  with a traveling wave for the viscous equation (\ref{viscf}), a technical difficulty arises.  
Namely, the assumption that a shock with left and right states $u^-$ and $u^+$ is Liu-admissible 
for the polygonal flux $f_\nu$ at (\ref{fnudef}) does not necessarily imply that it is still Liu-admissible for the original flux $f$. 
The issue can be resolved by adding to the flux $f$ a small perturbation. We thus define
\bel{TFG} 
\Tilde f(u)~\doteq~f(u) - \ve_0 u^2,\qquad\qquad \Tilde g(u)~\doteq~g(u)+ \ve_0 u^2,\eeq
to deal with upward and downward jumps, respectively.

\begin{lemma} \label{l:50}  Let $\ve_0>0$ be given, and let $h\in \C^2(\R)$. Then one can choose
$ \bar \nu\geq 1$ depending only on $\ve_0$ and $ \|h\|_{\C^2}$ such that the following holds. 
If $b-a\geq \ve_0$,   $h(a)=h(b)=0$,  $\nu\geq \bar\nu$  and $$h(2^{-\nu} k)\geq 0\qquad
\hbox{for every integer $k$ such that}~~  2^{-\nu}k\in [a,b],$$
then
\bel{Tfu} h(u) + \ve_0(u-a)(b-u)~\geq~\frac{\ve_0(u-a)(b-u)}{ 2}~\geq~0\qquad \forall ~u\in [a,b]\,.\eeq
\end{lemma}

\begin{proof}
Let $\bar u \in  \{a\} \cup \{2^{-\nu}k\in[a,(a+b)/2[\}$. We claim that 
\bel{hu-lb}h(u)~\geq ~-2^{-\nu}(u-\bar u)C_h ,  \qquad 
C_h~\doteq~\max_{u\in [a,b]}~\bigl|h''(u)\bigr|,
\qquad\forall u\in \bigl[\bar u, \bar u+2^{-\nu}\bigr].\eeq
Indeed, by assumption we have $h(\bar u)\ge 0$. If $h(u)\ge 0$ for all $u\in \bigl[\bar u, \bar u+2^{-\nu}\bigr]$, then (\ref{hu-lb}) trivially holds.
Otherwise, assume $\min \{h(u);\ u\in \bigl[\bar u, \bar u+2^{-\nu}\bigr] \}<0$. 
By assumption, in addition to $\bar u$, there exists another point $\bar u_1$ with  $\bar u < \bar u_1\le \bar u+2^{-\nu}$  
where $h$ is non-negative: $h(\bar u_1)\ge 0$.
Therefore, there must exist $v_0\in [\bar u, \bar u_k+2^{-\nu}]$ where $h'(v_0)\geq 0$.
Using the mean value theorem we obtain
$$h'(u)~\geq ~-2^{-\nu} C_h\,,\qquad\qquad\forall u\in [\bar u, \bar u+2^{-\nu}].$$
Since $h(\bar u)\ge 0$, this implies
(\ref{hu-lb}).

Next, observing that $\bar u\ge a$, from (\ref{hu-lb}) we obtain the lower bound 
\bel{fb2}h(u)~\geq ~-2^{-\nu}(u-a)C_h\qquad\qquad\forall u\in \bigl[a, (a+b)/2\bigr].\eeq
Recalling that $b-a\geq \ve_0$, for all $u\in \bigl[a, (a+b)/2\bigr]$ one has
\bel{fb3}
\frac{\ve_0 (u-a)(b-u)}{ 2}~\geq~\frac{\ve_0^2}{ 4}(u-a)  ~\geq~2^{-\nu}(u-a)C_h\,,\eeq
provided that the integer $\nu\geq \bar \nu$ is large enough.

Together, the inequalities  \eqref{fb2} and  \eqref{fb3} yield  \eqref{Tfu} on $\bigl[a,(a+b)/2\bigr]$.
An entirely similar argument yields the same inequality (\ref{fb3}) on the half interval $\bigl[(a+b)/2, b\bigr]$. 
\end{proof}

\v
\begin{corollary}\label{c:41} Let $\ve_0>0$ be given, together with the flux functions $f,g\in \C^2(\R)$. 
For all $\nu\geq \bar\nu$ large enough, defining $\Tilde f, \Tilde g$ as in 
(\ref{TFG}) the following holds.

Let $u^-<u^+$ be the left and right states of a Liu-admissible shock for the flux $f_\nu$.
Assume $u^+-u^-\geq \ve_0$.   Then the same shock is Liu-admissible for the flux
$\Tilde f$. Indeed, for every $u\in [u^-, u^+]$ one has
\bel{fb5}\Tilde f(u) ~\geq~\Tilde f(u^-) + \frac{\Tilde f(u^+) - \Tilde f (u^-)}{ u^+-u^-} (u-u^-) + 
\frac{ \ve_0 (u-u^-) (u^+-u)}{ 2}\,.
\eeq
Similarly, let $u^->u^+$ be the left and right states of a Liu-admissible shock for the flux $g_\nu$.
Assume $u^--u^+\geq \ve_0$.  Then the same shock is Liu-admissible for the flux
$\Tilde g$. Indeed, for every $u\in [u^+, u^-]$ one has
\bel{fb6}\Tilde g(u) ~\leq~\Tilde g(u^+) + \frac{\Tilde g(u^-) - \Tilde g (u^+)}{ u^--u^+} (u-u^+) - 
\frac{\ve_0 (u-u^+) (u^--u)}{ 2}\,.
\eeq
\end{corollary} 

\begin{proof}
The inequality (\ref{fb5}) follows by applying Lemma~\ref{l:50} to the function 
$h(u) \,\doteq\, f(u) - f(u^-) - \lambda (u - u^-)$ on the interval $[u^-,u^+]$, with $\lambda$ 
given by Rankine-Hugoniot condition, and by recalling that $f=f_\nu$ at all points of $2^{-\nu}\Z$.
The proof of (\ref{fb6}) is entirely similar.
\end{proof}

Thanks to Corollary \ref{c:41}, we can approximate an upward shock for  $f_\nu$ or  a downward shock
for $g_\nu$ with traveling profiles for the viscous conservation laws
$$u_t+\Tilde f(u)_x~=~\delta\, u_{xx}\,,\qquad\qquad  u_t+\Tilde g(u)_x~=~\delta\, u_{xx}\,.
$$

For future use, we shall need to truncate these profiles, so that they coincide with $u^-$ or $u^+$ outside a 
bounded interval.
This will be achieved using a non-increasing function $\vp:\R_+\mapsto [0,1]$ such that
\bel{interp}\vp(s)~=~\left\{\bega{cl}1\quad&\hbox{if}~~s\in [0,1],\\[1mm]
0\quad&\hbox{if}~~s\geq 2.\enda\right.\eeq

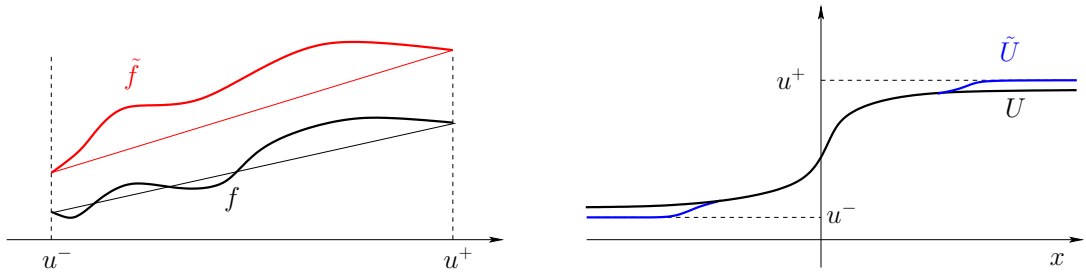
\begin{figure}[htbp]
\begin{center}
\resizebox{0.95\textwidth}{!}{
\begin{picture}(0,0)%
\includegraphics{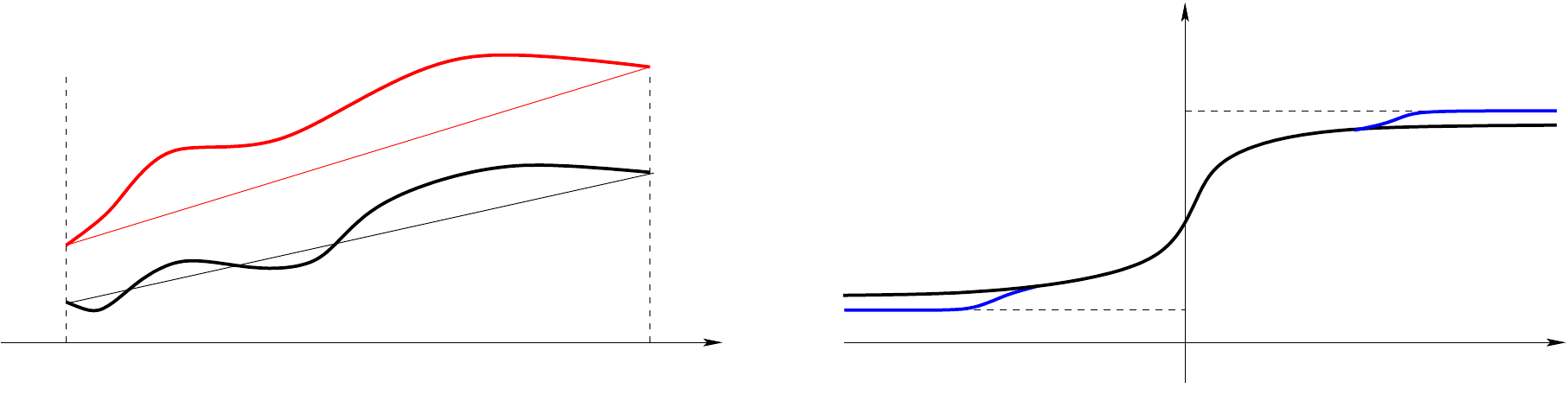}%
\end{picture}%
\setlength{\unitlength}{3947sp}%
\begin{picture}(14499,3642)(589,359)
\put(2176,2939){\makebox(0,0)[lb]{\smash{\fontsize{22}{26.4}\usefont{T1}{ptm}{m}{n}{\color[rgb]{1,0,0}$\tilde f$}%
}}}
\put(14626,494){\makebox(0,0)[lb]{\smash{\fontsize{22}{26.4}\usefont{T1}{ptm}{m}{n}{\color[rgb]{0,0,0}$x$}%
}}}
\put(14026,2489){\makebox(0,0)[lb]{\smash{\fontsize{22}{26.4}\usefont{T1}{ptm}{m}{n}{\color[rgb]{0,0,0}$U$}%
}}}
\put(13951,3239){\makebox(0,0)[lb]{\smash{\fontsize{22}{26.4}\usefont{T1}{ptm}{m}{n}{\color[rgb]{0,0,1}$\tilde U$}%
}}}
\put(10951,2864){\makebox(0,0)[lb]{\smash{\fontsize{22}{26.4}\usefont{T1}{ptm}{m}{n}{\color[rgb]{0,0,0}$u^+$}%
}}}
\put(11626,1064){\makebox(0,0)[lb]{\smash{\fontsize{22}{26.4}\usefont{T1}{ptm}{m}{n}{\color[rgb]{0,0,0}$u^-$}%
}}}
\put(6476,464){\makebox(0,0)[lb]{\smash{\fontsize{22}{26.4}\usefont{T1}{ptm}{m}{n}{\color[rgb]{0,0,0}$u^+$}%
}}}
\put(1076,464){\makebox(0,0)[lb]{\smash{\fontsize{22}{26.4}\usefont{T1}{ptm}{m}{n}{\color[rgb]{0,0,0}$u^-$}%
}}}
\put(3526,1289){\makebox(0,0)[lb]{\smash{\fontsize{22}{26.4}\usefont{T1}{ptm}{m}{n}{\color[rgb]{0,0,0}$f$}%
}}}
\end{picture}%
}
\caption{\small  Left: a flux function $f$ and the modified flux $\Tilde f$, constructed in  Corollary~\ref{c:41}.
Notice that the upward jump with left and right states $u^-<u^+$ is Liu-admissible for $\Tilde f$ but not for $f$.
Right: a viscous traveling profile $U$ constructed at (\ref{Utvd}), and its smooth
truncation $\Tilde  U$ at (\ref{mollU}).
}
\label{f:df42}
\end{center}
\end{figure}

Next,
assume that the shock with left and right states $u^-<u^+$ is Liu admissible for the polygonal flux $f_\nu$
and  define $\Tilde f$ as in (\ref{TFG}). 
Let $u(t,x)=U(x-\lambda t)$ be a viscous traveling profile in connection with the flux $\Tilde f$ 
and with $\delta=1$.  More precisely, let $U$ satisfy
\bel{Utvd} U'~=~\Tilde f(U) -\lambda U + C,\qquad\qquad U(0)\,=\,\frac{u^++u^-}{ 2}\,,\eeq
where the speed $\lambda$ and the constant $C$ are as in (\ref{RHL}) and (\ref{CC}) for the flux $\Tilde f$.
Recalling (\ref{interp}), define the smooth interpolation by setting
 \bel{mollU} \Tilde U(x)~\doteq~\left\{ 
 \bega{cl} \vp(x/\delta^{1/3}) \cdot U(x/\delta) + \bigl( 1-\vp(x/\delta^{1/3})\bigr) u^+
 \qquad &\hbox{if}~~x\geq 0,\\[3mm]
  \vp(-x/\delta^{1/3}) \cdot U(x/\delta) + \bigl( 1-\vp(-x/\delta^{1/3})\bigr) u^-
 \qquad &\hbox{if}~~x\leq 0.\enda\right.\eeq
Notice that the above definition implies
\bel{trU} \Tilde U(x)~\doteq~\left\{\bega{cl} u^-\qquad &\hbox{if}\quad x \leq -2\delta^{1/3},\\[1mm]
U(x/\delta)\qquad &\hbox{if}\quad  |x|<\delta^{1/3},\\[1mm]
 u^+\qquad &\hbox{if}\quad x \geq 2\delta^{1/3}.\enda\right.\eeq
 
 We estimate by how much the profile $\Tilde U$ fails to satisfy the viscous conservation law (\ref{viscf}).
\begin{lemma}\label{l:52}
 Consider 
 the interpolated profile  $\Tilde U$ defined  
 at (\ref{mollU}).  Setting $w(t,x)\doteq~\Tilde U(x-\lambda t)$, 
 for any fixed time $t$ and all $\delta>0$ sufficiently small one has
\bel{erb7}\int_{-\infty}^{+\infty}  \Big|
w_t +  f(w)_x 
-\delta w_{xx}\Big|\, dx\, =\,\O(1)\cdot \ve_0\,  |u^+-u^-|.
\eeq
\end{lemma}

\begin{proof}
The integral in (\ref{erb7}) ranges over three distinct regions: (1) the internal region where
$\Tilde U$ coincides with the viscous shock, (2) an external region where $\Tilde U= u^-$ or 
$\Tilde U=u^+$,  and  (3) an intermediate region where the interpolation occurs.  
We analyze each part below. 

{\bf 1.} For  $|x-\lambda t|\leq  \delta^{1/3}$ one has $w(t,x)=\Tilde U(x-\lambda t)=U(\frac{x-\lambda t}{\delta})$.
On this region there holds
$$w_t + f'(w) w_x -\delta w_{xx}~=~-\frac{\lambda}{\delta} U' +\frac{ f'(U)}{\delta} U' - \frac{1}{\delta} U''
~=~\frac{1}{\delta} \bigl[ f'(U) -\Tilde  f'(U)\bigr] U'.$$
Indeed, the left hand side would vanish if $f$ were replaced by $\Tilde f$.
Performing the change of variable $s=x/\delta$, one obtains
\begin{align}
& \int \Big|w_t + f'(w) w_x -\delta w_{xx}\Big|\, dx~=~
\frac{1}{\delta}
\int \bigg| \Big[f'\bigl(U(x/\delta)\bigr) -\Tilde f'\bigl(U(x/\delta)\bigr) \Big] U'(x/\delta)\bigg|\, dx
\nonumber\\
&\qquad \ds =~
\int\Big| f'\bigl(U(s)\bigr) -\Tilde f'\bigl(U(s)\bigr) \Big| \cdot \bigl|U'(s)\bigr|\, ds
~=~\O(1)\cdot \ve_0\,|u^+-u^-|.
\label{icv}
\end{align}

{\bf 2.} On the domain where
$|x-\lambda t|\geq 2\delta^{1/3}$ the function $w$ is constant.  Hence the integrand 
in (\ref{erb7}) is trivially zero.

{\bf 3.} 
To estimate the error on  the region where ~$x-\lambda t \in I_\delta\doteq \bigl[\delta^{1/3}, ~2\delta^{1/3}
\bigr]$  we observe that, by (\ref{Utvd}) and (\ref{fb5}), for all $s>\!>1$ suitably large  one has
$$U' (s) ~\geq~ \frac{\ve_0}{2}\left(U(s)-u^-\right)\left(u^+-U(s)\right) 
~\geq~\frac14 \ve^2_0\bigl(u^+-U(s)\bigr)\,.$$
Hence
$$ u^+-U(s) ~\le~ e^{-\alpha s}, \qquad \mbox{with}\quad \alpha= \ve^2_0/4.$$
For $x-\lambda t\in I_\delta$ we now have the estimate
$$
\sup_{s\geq \delta^{1/3}}\, \bigl| u^+- U(s/\delta)\bigr| ~=~  \bigl| u^+- U(\delta^{-2/3})\bigr|~
~=~ \O(1) \cdot e^{-\alpha /\delta^{2/3}}\,.
$$
Setting $$\vp_{\delta^{1/3}}(s)~\doteq~\vp(s/\delta^{1/3})$$
and observing that
$$ 
\left\| (\vp_{\delta^{1/3}} )_x\right\|_{\L^\infty}~=~ \delta^{-1/3} ,\qquad
\left\| (\vp_{\delta^{1/3}} )_{xx}\right\|_{\L^\infty}~=~ \O(1) \cdot \delta^{-2/3} ,
$$
we thus have the bound
\begin{align*} 
& \bigl| w_t +  f'(w) w_x -\delta w_{xx} \bigr|
 \\
&  \qquad\leq~ \O(1)\cdot \Big\{1+ \left\| (\phi_{\delta^{1/3}} )_x\right\|_{\L^\infty} +
\left\| (\phi_{\delta^{1/3}} )_{xx}\right\|_{\L^\infty} \Big\}
\\
& \qquad \qquad\qquad 
 \cdot \sup_{s\geq \delta^{1/3}}\,\Big\{  \bigl| u^+- U(s/\delta)\bigr|  + \delta^{-1} \bigl| U'(s/\delta)\bigr|
 + \delta^{-2}\bigl| U''(s/\delta)\bigr|\Big\}
 \\
&  \qquad \leq~\O(1)\cdot \delta^{-3}\cdot e^{-\alpha /\delta^{2/3}}.
\end{align*}
%
%
The right hand side can be rendered arbitrarily small by choosing 
$\delta>0$ small enough. The same estimate holds on the interval 
$- I_\delta\doteq \bigl[-2\delta^{1/3}, \,-\delta^{1/3}
\bigr]$.

In particular, for $\delta>0$ small enough, we achieve
$$ \int_{-I_\delta\cup I_\delta}  \Big|
w_t +  f(w)_x 
-\delta w_{xx}\Big|\, dx\, \leq ~ \ve_0^2~\leq~\ve_0\,  |u^+-u^-|.$$
Together with (\ref{icv}), this yields (\ref{erb7}).
\end{proof}

In a solution constructed by standard front tracking algorithm, 
the states to the left and right of a wave front
remain constant in time.   However, according to the front tracking algorithm described in Section~\ref{sec:2}
a piecewise constant approximation $u_\nu$  changes in time along 
intervals where a local maximum or minimum is attained.   We thus need to 
understand how a viscous traveling profile changes, as the left and right states
$u^-, u^+$ are modified.  

As a first step toward this goal, the next lemma  provides an estimate on the difference between solutions of  two ODEs.

\begin{lemma}\label{l:53}
Let $F,G$ be $\C^1(\R)$ functions such that, for some $U^+\in \R$ and some constants $\Delta_0, \ve_1>0$, there holds
\begin{align}
& F(U^+) \,=\, G(U^+)\,=\,0,\label{FG0} \\
 &F'(U)\,\leq -\ve_1\,,\qquad    G'(U)\,\leq -\ve_1\qquad\qquad \hbox{for}\quad U<U^+,\label{FG1} \\
& \|F-G\|_{\C^1}\,\leq\,\Delta_0\,. \label{FG2} 
\end{align}

Let  $U,V:[s_0, +\infty[\,\mapsto\R$  be functions that satisfy
\begin{align}
&U'(s)~=~F(U(s)), \qquad V'(s)=G(V(s)), \qquad \qquad s\ge s_0\,,\label{UVs1} \\[2mm]
& u_0\doteq U(s_0)<U^+,\quad v_0\doteq V(s_0)< U^+,\qquad 
\bigl|U(s_0)-V(s_0)\bigr| \le \Delta_0 \,. \label{UVs0}
\end{align}
Then 
\bel{UV}\sup_{s\geq s_0} ~\bigl|U(s)-V(s)\bigr|  ~\leq~   
 \left(1 +   \frac{U^+-v_0}{\ve_1} \right)\cdot \Delta_0\,.
\eeq 
\end{lemma}

\begin{proof}
We start by observing that $F$, $G$ are positive for $U<U^+$. Hence 
both functions $U$, $V$ 
are strictly increasing on $[s_0, +\infty[$ and approach $U^+$ as $s\to +\infty$.

Defining $W(s)\doteq U(s)-V(s)$, one has
\begin{align}
W'(s)&
~=~F\bigl(U(s)\bigr)- G\bigl(V(s)\bigr) 
 ~=~ -\vp_1(s) \cdot W(s) + \vp_2(s) \cdot \Delta_0\cdot  \bigl(U^+-V(s)\bigr)\,, \label{Wexp}
\end{align}
where
$$\vp_1(s) ~\doteq~ - \frac{F\bigl(U(s)\bigr) - F\bigl(V(s)\bigr)}{U(s)-V(s)} ,
\qquad
\vp_2(s) ~\doteq~ \frac{F(V(s))- G(V(s))}{\Delta_0\cdot  \bigl(U^+-V(s)\bigr)}.
$$
In the above setting, we claim that
\begin{equation} \label{vp12}
\vp_1(s) \ge \ve_1>0, 
\qquad
|\vp_2(s)|\le1 \qquad \mbox{ for all}\quad  s\ge s_0\,.
\end{equation}
Indeed, the first estimate follows immediately from \eqref{FG1}. 
The second estimate is obtained by writing
$$
\bigl|F(V(s))- G(V(s))\bigr| 
~\le~ \|F-G\|_{\C^1}\cdot \int_s^{+\infty} V'(\sigma)\,d\sigma~ =~ \Delta_0\cdot  \bigl|U^+-V(s)\bigr|\,.
$$
%
{}From (\ref{Wexp}) and  (\ref{vp12}) it follows
\bel{Wexp2}\frac d{ds}|W(s)|~\le ~ - \ve_1 \cdot |W(s)| +  \Delta_0\cdot  \bigl|U^+-V(s)\bigr|\,.
\eeq
We now denote by $v\mapsto V^{-1}(v)$  the inverse of the strictly increasing map $s\mapsto V(s)$, 
and make a coordinate change. 
Defining $\widetilde W(v) \doteq W \bigl(V^{-1}(v)\bigr)$, 
from (\ref{Wexp2}) we obtain
$$
\ds \frac{d }{ dv}\bigl|\widetilde W(v)\bigr|
~\le ~\Big(-\ve_1\bigl|\widetilde W(v)\bigr| +  \Delta_0\cdot  
\left|U^+-v\right|\Big) \cdot \frac{1}{G(v)}~
 \le ~ \Delta_0 \frac{\left|U^+-v\right|}{G(v)}~\le ~ \frac{\Delta_0}{\ve_0}\,.
 $$
 Here we used the estimate $0\le  \frac{U^+-v}{G(v) - G(U^+)} \le \frac{1}{ \ve_1}$ for $v<U^+$, which follows from 
 \eqref{FG0}-\eqref{FG1}.  
Furthermore, the condition (\ref{UVs0}) implies that $|\widetilde W(v_0)|\le \Delta_0$. 
Therefore 
$$
\bigl|\widetilde W(v)\bigr|~\le ~\bigl|\widetilde W(v_0)\bigr| + \frac{\Delta_0}{\ve_1} \left(v-{v_0}\right) 
~ \le~ \Delta_0\cdot \left(1 +   \frac{U^+-v_0}{\ve_1} \right),
 $$
from which we conclude that   
$$\bigl|W(s)\bigr| ~=~ \bigl|\widetilde W(V(s))\bigr|~\le ~ \Delta_0\cdot \left(1 +   \frac{U^+-v_0}{\ve_1} \right)
\qquad \quad \forall s\ge s_0\,.
$$
This yields  (\ref{UV}).
\end{proof}

\begin{lemma}\label{l:54} Let $f,\Tilde f$ be  as in Corollary~\ref{c:41}, so that  $\Tilde f(u) ~\doteq~ f(u) - \ve_0u^2$. 
Let $U$ be a viscous shock profile satisfying (\ref{Utvd}) with left and right states $u^-<u^+$, 
and let $V$ be a similar profile, but with left and right states $v^-<v^+$. 
Let $\ve_0\in\, ]0,1]$ and assume that
\bel{uv-pm}
|u^+-u^-|\,\geq \,\ve_0\,,\qquad  |v^+-v^-|\,\geq\, \ve_0\,.\eeq
Then there exists a constant $C=C(\ve_0)>0$, independent of $u^\pm, v^\pm$, 
such that
\bel{UVdis}
\|U-V\|_{\L^\infty}~\leq~ C(\ve_0) 
\cdot \bigl(|u^+-v^+|+|u^--v^-|\bigr)\,.\eeq
\end{lemma}

\begin{proof}
{\bf 1.} Set
\bel{Delta}
\Delta\,\doteq \,|u^+-v^+|+|u^--v^-| \eeq
and define the maps
\bel{Fuab}F(u,a,b) ~\doteq~ \Tilde f(u) - \left[\Tilde f(a) +  \frac{\Tilde f(b) - \Tilde f (a)}{ b-a} (u-a) \right]\,, 
\qquad u,a,b\in \R\,, \ a\not=b\,, \eeq
\bel{Fuvw}
w~\mapsto~F^u(w)~\doteq~ F(w,u^-,u^+)\,, \qquad w~\mapsto~F^v(w)~\doteq~ F(w,v^-,v^+)\,.
\eeq
With this notation, the viscous shock profiles $U$, $V$ satisfy
$$
U'(s) = F^u\bigl(U(s)\bigr) \,, \qquad  V'(s) = F^v\bigl(V(s)\bigr) , 
\qquad
U(0)=\frac{u^++u^-}2, \qquad V(0)=\frac{v^++v^-}2.
$$
%
Notice that $|U(0) - V(0)|\leq \Delta$. 
We also observe that the map $F$ in (\ref{Fuab}), extended by continuity for $a=b$, is locally Lipschitz continuous in its variables $(u,a,b)$.
\v
{\bf 2. Estimates on $F^u$, $F^v$.}\quad We claim that, for some positive constant $c_1$ depending on $f$ but not on $u^\pm, v^\pm$, 
the derivatives of the functions in (\ref{Fuvw}) satisfy 
\begin{align}
\label{Fp-}(F^u)'(u~)\le~ - {\ve^2_0/ 4} \qquad &\forall u\in [u^+- c_1\ve_0, u^+]\,,\\[2mm]
\label{Fpv-}(F^v)'(v) ~\le ~  - {\ve^2_0/ 4}\qquad &\forall v \in[v^+- c_1\ve_0, v^+]\,.\end{align}
The proofs are based on (\ref{fb5}). Starting with the identity
$$
\Tilde f(u^+) + \Tilde\lambda (u-u^+)~ =~ \Tilde f(u^-) + \Tilde\lambda (u-u^-)\,,\qquad   \qquad 
\Tilde\lambda~=~\frac{\Tilde f(u^+) - \Tilde f (u^-)}{ u^+-u^-}\,,
$$
from (\ref{fb5}) it follows
$$
\frac{\Tilde f(u^+) - \Tilde f (u)}{u^+-u} ~\le~ \Tilde\lambda - \ve_0 \frac{u-u^-} 2\,,\qquad   
\hbox{for}~~u\in \,]u^-,u^+[\,.
$$
Letting $u\to u^+$ we deduce a bound on the derivative $\Tilde f'(u^+)$, namely
$$
\Tilde f'(u^+) ~\le~ \Tilde\lambda - \ve_0 \frac{u-u^-} 2\,.
$$
Using a  bound on the second derivative $\Tilde f''$,
we can find a constant $c_1>0$ small enough so that, for all $u\in  [u^+- c_1\ve_0, u^+]$,
$$
(F^u)'(u) ~=~ \Tilde f'(u) -  \Tilde\lambda  
~\le \,-\ve_0\frac{(u^+-u^-)}2 + c_1 \ve_0\, \|\Tilde f''\|_{\L^\infty} ~ \le~  -\ve_0\frac{(u^+-u^-)}4 \,.
$$
In view of (\ref{uv-pm}), we obtain (\ref{Fp-}). The proof of (\ref{Fpv-}) is analogous.

Next, we claim that,
\bel{Fbig}F^u(u)~\geq~\frac{c_1\ve_0^3}{ 4}~~~\qquad \forall~ u\in \Big[ \frac{u^+ + u^-}{ 2}, ~u^+-c_1\ve_0\Big].\eeq
Indeed, by (\ref{fb5}) and  (\ref{uv-pm}), for the above values of $u$ we have
$$
F^u(u)~\ge~  \frac{\ve_0 (u-u^-) (u^+-u)}{ 2}  ~\ge~\frac{\ve_0 (\ve_0/2) (c_1\ve_0)}{ 2}\,.$$
\v

In the remainder of the proof,
calling $s_0>0$ the point where $U(s_0)= u^+- c_1\ve_0$, we shall 
estimate the difference $U(s)-V(s)$ separately for $s\in [0, s_0]$ and for $s\in [s_0, +\infty[\,$.

\v
{\bf 3.}  To achieve an estimate for $s\in [0, s_0]$, we start with the inequality
\bel{dFuv}
\bega{rl}\ds
\frac{d}{ ds} \Big|U(s) -V(s)\Big|
&\le ~
\Big|F^u\bigl(U(s)\bigr) -F^u\bigl(V(s)\bigr)\Big|+ \Big|F^u\bigl(V(s)\bigr) -F^v\bigl(V(s)\bigr)\Big|
\\[4mm]
&\leq~C_1 \bigl|U(s)-V(s)\bigr| + C_2 \Delta
\enda
\eeq
for some positive constant $C_1$, $C_2$. 
These bounds follow  from the  Lipschitz continuity of $F$, $F^u$, $F^v$, defined at 
(\ref{Fuab})-(\ref{Fuvw}).

Since $s\mapsto U(s)$ is strictly increasing, we can consider the inverse function $U\mapsto s(U)$.
Defining
$\Tilde V(u) ~\doteq~ V\bigl(s(u)\bigr)$, by  (\ref{dFuv}) and (\ref{Fbig}) we obtain
$$
\frac{d}{ du} \Big| u -\Tilde V(u)\Big| \ds=~\frac{\frac{d}{ ds} \Big|U(s) -V(s)\Big|}{F^u \bigl(U(s)\bigr)}
\ds\leq~\Big( C_1 \bigl| u-\Tilde V(u)\bigr| + C_2 \Delta
\Big) \frac{4}{ c_1\ve_0^3}\,.
$$
Setting $u_0\doteq U(0)$ and $\Hat C_1\doteq  4C_1/ (c_1\ve_0^3)$, we obtain
$$
\left| u -\Tilde V(u)\right|~\le~ e^{\Hat C_1 (u-u_0)} \left| u_0 - \Tilde V(u_0)\right|
+ \frac{C_2}{C_1} \,\left(e^ {\Hat C_1 (u-u_0)}  -1\right)\cdot\Delta\,,
$$
for all $ u \in [u_0, u^+- c_1\ve_0]$. 
Since $\left|u_0 - \Tilde V(u_0) \right|= \left|U(0) - V(0)\right|\leq \Delta$,
we conclude that
\bel{UV3}
\bigl|U(s)-V(s)\bigr|~\leq~C_3\, \Delta\qquad\qquad \forall s\in [0, s_0]\eeq
for some constant $C_3$ independent of $u^\pm, v^\pm$.

%

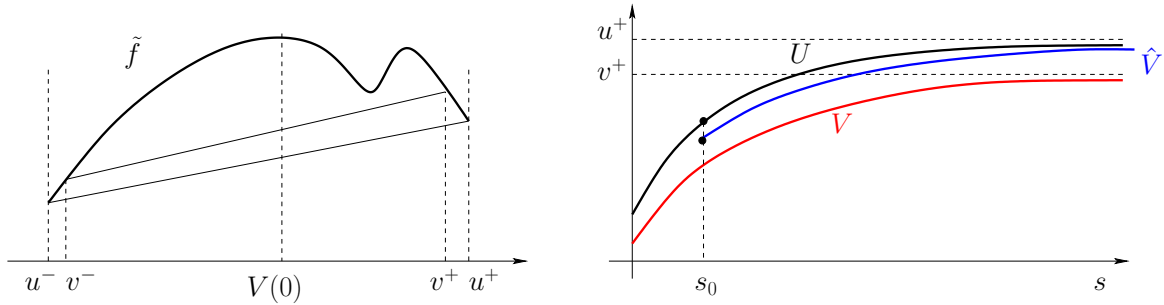
\begin{figure}[htbp]
\begin{center}
\resizebox{1\textwidth}{!}{
\begin{picture}(0,0)%
\includegraphics{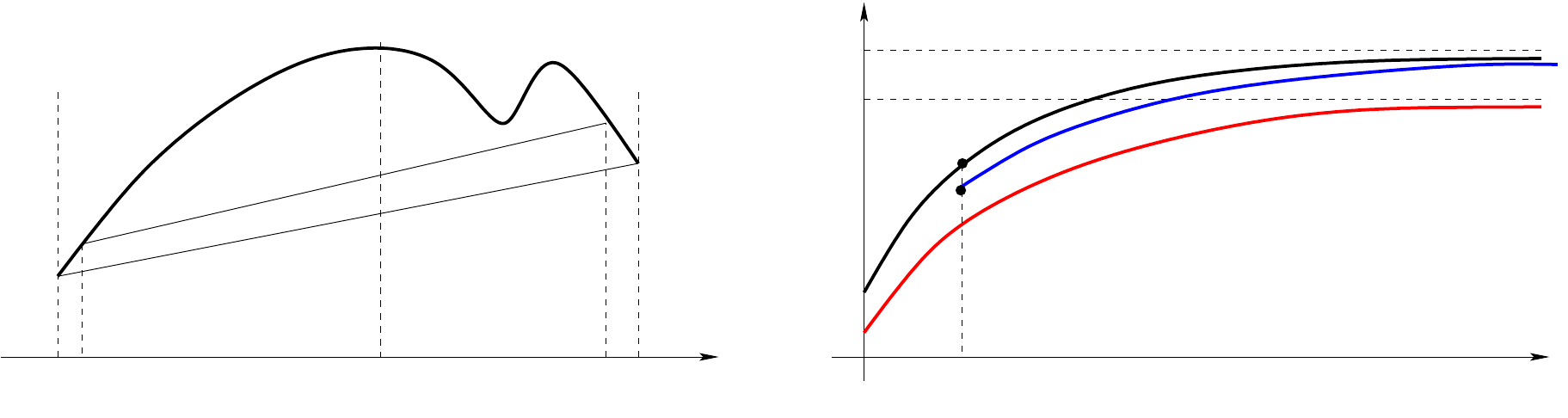}%
\end{picture}%
\setlength{\unitlength}{3947sp}%
\begin{picture}(14577,3876)(664,275)
\put(8251,3164){\makebox(0,0)[lb]{\smash{\fontsize{22}{26.4}\usefont{T1}{ptm}{m}{n}{\color[rgb]{0,0,0}$v^+$}%
}}}
\put(10726,3389){\makebox(0,0)[lb]{\smash{\fontsize{22}{26.4}\usefont{T1}{ptm}{m}{n}{\color[rgb]{0,0,0}$U$}%
}}}
\put(15226,3239){\makebox(0,0)[lb]{\smash{\fontsize{22}{26.4}\usefont{T1}{ptm}{m}{n}{\color[rgb]{0,0,1}$\hat V$}%
}}}
\put(8251,3689){\makebox(0,0)[lb]{\smash{\fontsize{22}{26.4}\usefont{T1}{ptm}{m}{n}{\color[rgb]{0,0,0}$u^+$}%
}}}
\put(14626,464){\makebox(0,0)[lb]{\smash{\fontsize{22}{26.4}\usefont{T1}{ptm}{m}{n}{\color[rgb]{0,0,0}$s$}%
}}}
\put(9501,464){\makebox(0,0)[lb]{\smash{\fontsize{22}{26.4}\usefont{T1}{ptm}{m}{n}{\color[rgb]{0,0,0}$s_0$}%
}}}
\put(3751,389){\makebox(0,0)[lb]{\smash{\fontsize{22}{26.4}\usefont{T1}{ptm}{m}{n}{\color[rgb]{0,0,0}$V(0)$}%
}}}
\put(901,464){\makebox(0,0)[lb]{\smash{\fontsize{22}{26.4}\usefont{T1}{ptm}{m}{n}{\color[rgb]{0,0,0}$u^-$}%
}}}
\put(6101,464){\makebox(0,0)[lb]{\smash{\fontsize{22}{26.4}\usefont{T1}{ptm}{m}{n}{\color[rgb]{0,0,0}$v^+$}%
}}}
\put(6601,464){\makebox(0,0)[lb]{\smash{\fontsize{22}{26.4}\usefont{T1}{ptm}{m}{n}{\color[rgb]{0,0,0}$u^+$}%
}}}
\put(1426,464){\makebox(0,0)[lb]{\smash{\fontsize{22}{26.4}\usefont{T1}{ptm}{m}{n}{\color[rgb]{0,0,0}$v^-$}%
}}}
\put(11251,2489){\makebox(0,0)[lb]{\smash{\fontsize{22}{26.4}\usefont{T1}{ptm}{m}{n}{\color[rgb]{1,0,0}$V$}%
}}}
\put(2176,3389){\makebox(0,0)[lb]{\smash{\fontsize{22}{26.4}\usefont{T1}{ptm}{m}{n}{\color[rgb]{0,0,0}$\tilde f$}%
}}}
\end{picture}%
}
\caption{\small Left: the flux $\Tilde f$ and the Liu-admissible jumps in Lemma~\ref{l:54}, 
with left, right states $u^- < u^+$ and  $v^- < v^+$. 
Right: the viscous shock profiles $U$, $V$ and the translated profile $\Hat V$ defined at (\ref{HVdef}).
}
\label{f:df49}
\end{center}
\end{figure}

{\bf 4.} It remains to estimate the difference $U-V$  for $s\in [s_0, +\infty[\,$.
Setting
\bel{HVdef}\Hat V(s) ~\doteq~V(s) + u^+-v^+\,, \eeq
(see Fig.~\ref{f:df49}), we trivially have
$$\bigl|U(s)-V(s)\bigr|~\leq~\bigl|U(s)-\Hat V(s)\bigr|+|u^+-v^+|~\leq~\bigl|U(s)-\Hat V(s)\bigr|+\Delta.$$
The difference $U-\Hat V$ will be estimated using  Lemma~\ref{l:53}, with $F(w)= F^u(w)$, $G(w)= F^v(w+v^+-u^+)$.
It is readily checked that
$$U'\,=\,F(U)\,,\qquad \quad \Hat V' \,=\, V' = F^v(V) \,=\, F^v(\Hat V  - u^+ +v^+) \,=\, G(\Hat V)\,,$$
and that (\ref{FG0}) holds at $w=u^+$ since $F(u^+)=G(u^+)=0$.
Moreover,
$$
\bigl|F(w) - G(w)\bigr|~\le ~\bigl| F^u(w) - F^v(w) \bigr| +  \bigl| F^v(w) -   F^v(w+v^+-u^+)\bigr|~\le~\left( C_2 +C_1 \right) \Delta\,,$$
where $C_1$, $C_2$ are the constants introduced  at (\ref{dFuv}).
Therefore (\ref{FG2}) holds, with $\Delta_0= \left( C_2 +C_1 \right) \Delta$.

Since $G'\bigl(w\bigr) = (F^v)'(w+v^+-u^+)$ with $w\in [u^+- c_1\ve_0, u^+]$, 
in view of (\ref{Fp-})-(\ref{Fpv-}),
on the interval $[u^+- c_1\ve_0, u^+]$ the conditions (\ref{FG1}) are satisfied with $\ve_1=\ve_0^2/4$.

By (\ref{UV3}) and (\ref{HVdef}) it follows
$$\bigl| U(s_0) - \Hat V(s_0)\bigr|~\leq~(C_3 + 1)\Delta.$$
An application of Lemma~\ref{l:53} thus yields the bound 
$$\bigl| U(s) - \Hat V(s)\bigr|~\leq~C_4\,\Delta$$
for some constant $C_4$ and all $s\in [s_0, \, +\infty[\,$.

Entirely similar arguments can be used to achieve a bound on $\bigl| U(s) - V(s)\bigr|$ for
$s\leq 0$, completing the proof of the lemma.
\end{proof}

\v
\begin{remark} {\rm
In view of Lemma~\ref{l:54}, if we replace $u^-, u^+$ by $v^-, v^+$, after the rescaling, truncation and mollification at (\ref{trU})-(\ref{mollU}), instead of $\Tilde U$ we obtain a profile $\Tilde V$ such that
\bel{HUC}\bigl\|\Tilde  V-\Tilde U\bigr\|_{\L^1\bigl([-2\delta^{1/3}, 2\delta^{1/3}]\bigr)} ~=~C_5 \cdot \delta^{1/3}\cdot \bigl( |u^+-v^+| +|u^--v^-|\bigr).
\eeq
for some constant $C_5$ depending on $f$ and $\ve_0$, but not on $u^\pm$, $v^\pm$ or $\delta$.

The same estimates can be proved for downward jumps of size $\geq\ve_0$, 
replacing $f$ by the flux 
function $g$. 
}
\end{remark}
\v
\subsection{Near-constant solutions at local maxima or minima}
Next, consider a front tracking approximation $u_\nu$, constructed as in Section~\ref{sec:2}.
For a fixed time $t$,  let $\bigl[ x_\alpha, \, x_{\alpha+1}\bigr]= [y_\alpha-h, y_\alpha+h]$ be an interval
where $u_\nu(t,x)= \Hat u_\alpha(t)$ is constant and attains a local maximum.
According to (\ref{24}), on this interval the front tracking solution is decreasing in time.

To construct a 
 solution of the viscous equation (\ref{3})  with a similar behavior, we seek a function  
$V:[-h,h]\mapsto \R$ such that 
$$\theta_\ve'(V')\bigl[ g(\Hat u_\alpha(t)) -  f(\Hat u_\alpha(t)) \bigr] V''
~=~\frac{d}{ dt} \Hat u_\alpha(t)~=~\frac{ f(\Hat u_\alpha(t)) -  g(\Hat u_\alpha(t)) }{
 x_{\alpha+1}(t)-x_{\alpha}(t)}\,.$$
 Here $\theta_\ve$ is defined at~\eqref{tdef}, see Fig.~\ref{f:df39}.
 After a cancellation, we have an equation for $V$
 \bel{V2} \theta'_\ve(V') \, V''~=~-\frac{1}{ 2h}\,.\eeq
To solve for $V$, we consider the auxiliary function $Z:[-1,1]\mapsto \R$, implicitly defined by
\bel{Zeq}  \theta\bigl(Z(s)\bigr)~=~\frac{1-s}{ 2}\,,\qquad\qquad
Z(0)=0.\eeq
Notice that, by (\ref{tprop1})-(\ref{tprop2}), it follows 
\bel{Zprop}\left\{\bega{ll} \theta\bigl(Z(-1)\bigr)=0,\qquad &Z(-1)=1,
\\[1mm]
\theta\bigl(Z(0)\bigr)={1/2} ,\qquad &Z(0)=0,\\[1mm]
\theta\bigl(Z(1)\bigr)=1,\qquad &Z(1)=-1,\enda\right.\qquad\qquad \theta'\bigl(Z(s)\bigr)  Z'(s)~=\,-\frac12\,. \eeq
Choosing the function $V$ so that  $V'(x)= \ve\,Z(x/h)$, for $x\in [-h,h]$ we obtain
$$\theta'_\ve\bigl(V'(x)\bigr) V''(x) 
~=~\theta'_\ve\bigl(\ve Z(x/h)\bigr)\frac{\ve Z'(x/h)}{ h}~
=~ \theta'\bigl( Z(x/h)\bigr)\frac{ Z'(x/h)}{ h}~=\,-\frac{1}{ 2h}\,,$$
as required by (\ref{V2}).

\begin{figure}[htbp]
\begin{center}
\resizebox{0.95\textwidth}{!}{
\begin{picture}(0,0)%
\includegraphics{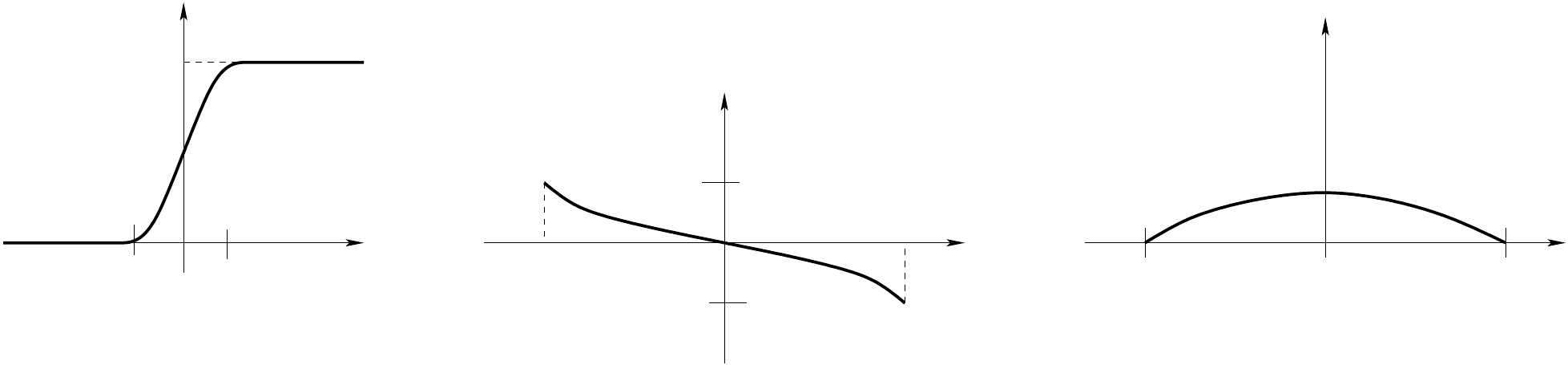}%
\end{picture}%
\setlength{\unitlength}{3947sp}%
\begin{picture}(15645,3624)(-1232,-373)
\put(7876,464){\makebox(0,0)[lb]{\smash{\fontsize{22}{26.4}\usefont{T1}{ptm}{m}{n}{\color[rgb]{0,0,0}$1$}%
}}}
\put(3976,539){\makebox(0,0)[lb]{\smash{\fontsize{22}{26.4}\usefont{T1}{ptm}{m}{n}{\color[rgb]{0,0,0}$-1$}%
}}}
\put(9901,464){\makebox(0,0)[lb]{\smash{\fontsize{22}{26.4}\usefont{T1}{ptm}{m}{n}{\color[rgb]{0,0,0}$y-h$}%
}}}
\put(11926,464){\makebox(0,0)[lb]{\smash{\fontsize{22}{26.4}\usefont{T1}{ptm}{m}{n}{\color[rgb]{0,0,0}$y$}%
}}}
\put(13576,464){\makebox(0,0)[lb]{\smash{\fontsize{22}{26.4}\usefont{T1}{ptm}{m}{n}{\color[rgb]{0,0,0}$y+h$}%
}}}
\put(5251,164){\makebox(0,0)[lb]{\smash{\fontsize{22}{26.4}\usefont{T1}{ptm}{m}{n}{\color[rgb]{0,0,0}$-\varepsilon$}%
}}}
\put(1426,2789){\makebox(0,0)[lb]{\smash{\fontsize{22}{26.4}\usefont{T1}{ptm}{m}{n}{\color[rgb]{0,0,0}$\theta_\varepsilon$}%
}}}
\put(976,389){\makebox(0,0)[lb]{\smash{\fontsize{22}{26.4}\usefont{T1}{ptm}{m}{n}{\color[rgb]{0,0,0}$\varepsilon$}%
}}}
\put(-224,389){\makebox(0,0)[lb]{\smash{\fontsize{22}{26.4}\usefont{T1}{ptm}{m}{n}{\color[rgb]{0,0,0}$-\varepsilon$}%
}}}
\put(4576,1289){\makebox(0,0)[lb]{\smash{\fontsize{22}{26.4}\usefont{T1}{ptm}{m}{n}{\color[rgb]{0,0,0}$\varepsilon Z$}%
}}}
\put(6251,1364){\makebox(0,0)[lb]{\smash{\fontsize{22}{26.4}\usefont{T1}{ptm}{m}{n}{\color[rgb]{0,0,0}$\varepsilon$}%
}}}
\put(12301,1439){\makebox(0,0)[lb]{\smash{\fontsize{22}{26.4}\usefont{T1}{ptm}{m}{n}{\color[rgb]{0,0,0}$V^{h,y}$}%
}}}
\put(226,2489){\makebox(0,0)[lb]{\smash{\fontsize{22}{26.4}\usefont{T1}{ptm}{m}{n}{\color[rgb]{0,0,0}$1$}%
}}}
\end{picture}%
}
\caption{\small  Left: the function $\theta_\ve$ in (\ref{tdef}).   Center and right: the
corresponding functions $\ve Z$ and $V^{h,y}$, considered at (\ref{Zeq})--(\ref{Vhy}). 
}
\label{f:df39}
\end{center}
\end{figure}
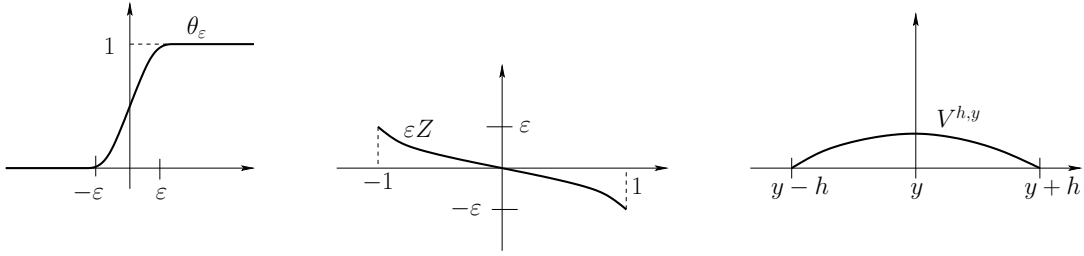

In the following, for a fixed $\ve>0$,  to keep track of the dependence on $h,y$, we write
\bel{Vhy} V^{h,y}(x)~\doteq~\ve  h \cdot \int_{-1}^{\frac{x-y}{h}}  Z(s) \, ds.
\eeq
The next lemmas collect some basic properties of these functions.

\begin{lemma} \label{l:41} The function $V^{h,y}:[y-h, y+h]\mapsto\R$  
defined at (\ref{Vhy}) provides a solution to
\bel{V4} \theta'_\ve(V'(x)) V''(x)~=~-\frac{1}{ h}\,,\qquad\qquad V(y-h)~=~V(y+h)~=~0.\eeq
Moreover, for all $x\in [y-h,\,y+h]$ one has 
\bel{Vprop}
\theta_\ve(V'(x))=\frac12 - \frac{x-y}{ 2h}, \quad
V'(y-h)=\ve, \quad V'(y+h)=-\ve, 
\quad 
V(x)\in[0,\ve h]. 
\eeq
\end{lemma}
Indeed, all the above properties are straightforward consequences of the definitions.

\begin{lemma} \label{l:56} For any $h>0$ and $y\in \R$, the following estimates hold.
\bel{Vxx} \int_{y-h}^{y+h} \left| \frac{\partial^2}{\partial x^2}  V^{h,y}(x)\right| dx~=~2\ve,\eeq
\bel{Vy} \int_{y-h}^{y+h} \left|\frac{\partial}{\partial y} V^{h,y}(x)\right|dx
~\leq ~2\ve h,
\eeq
\bel{Vh}  \int_{y-h}^{y+h} \left|\frac{\partial}{\partial h} V^{h,y}(x)\right| dx~\leq ~ 6\ve h.
\eeq
\end{lemma} 

\begin{proof}
{\bf 1.} 
The identity (\ref{Vxx})  follows from the middle two equalities in (\ref{Vprop}) and the fact that 
the derivative $V'=\ve Z$ is monotone decreasing. 

{\bf 2.}
Computing  the partial derivative of~\eqref{Vhy} and observing that $|Z|\le 1$, we obtain
$$ 
\left| \frac{\partial }{\partial y} V^{h,y}(x)\right| ~ = ~\left| \frac{\ve}{h} Z\Big(\frac{x-y}{ h}\Big)h \right| ~ \le~ \ve. 
$$
This yields~\eqref{Vy}.

{\bf 3.} 
Toward a proof of  (\ref{Vh}), again we perform a straightforward differentiation of (\ref{Vhy}):
$$
\left|\frac{\partial }{\partial h} V^{h,y}(x)\right| ~\le~ \left|\ve h \cdot \frac{x-y}{h^2} \cdot  Z((x-y)/h) 
\right|
+ \left| \int_{-1}^{\frac{x-y}{h}}\ve  Z(s)\, ds\right|  ~\le~ \ve + 2 \ve ~= ~3\ve.
$$ 
Here we used the fact $\bigl|(x-y)/h\bigr|\le 1$. 
Then~\eqref{Vh} follows. 
\end{proof}

\section{Construction of approximate viscous solutions}
\label{sec:6}
\setcounter{equation}{0}

Consider a piecewise constant, spatially periodic initial data  $\bar u\in \L^1_{per}(\R)$, and let $\ve_0>0$ be given. 
The goal of this section is to construct an approximate solution $w$ which satisfies the conclusion of Proposition~\ref{p:41}. 
The construction will be achieved in several steps.

{\bf 1.}
 For $\nu\geq 1$, denote by   $S^\nu:\L^1_{per}\times \R_+\mapsto \L^1_{per}$ the contractive semigroup
constructed in Section~\ref{sec:2}, replacing $f,g$ with the piecewise affine functions 
$f_\nu, g_\nu$.
Let
$t\mapsto u_\nu(t,\cdot)\doteq S^\nu_t\bar u$ be the corresponding front tracking solution.
As shown in Fig.~\ref{f:df43}, there will be a finite number (decreasing in time) of intervals
where $u_\nu(t,\cdot)$ attains a local maximum or minimum.   On each connected component of the complement of these
intervals, the flux is either equal to $f_\nu$ or to $g_\nu$. 

According to Proposition~\ref{p:22}, by choosing $\nu\geq 1$ large enough, 
one obtains the estimate (\ref{SSnu}).  In addition, we can also achieve the bound
\bel{nuc}
\Big(\|f_\nu-f\|_{W^{1,\infty}}+ \|g_\nu-g\|_{W^{1,\infty}}\Big) \cdot \bigl(\ov N+\TV \{\bar u\}
\bigr)~\leq~\ve_0\,,\eeq
where $\ov N$ is the number of intervals where $\bar u$ attains a local maximum or a local minimum (over one period).  

{\bf 2.} Consider the set of times 
\bel{restt}0\,=\,t_0\,<\,t_1\,<\,\cdots \,< \,t_N\,=\,T \eeq 
which includes:
\begin{itemize} 
\item all times $\tau_i$ where two fronts in $u_\nu$ interact,
\item all times $\tau_j$ where the strength $\bigl|\sigma_i(t)\bigr|$ of a jump in $u_\nu(t,\cdot)$
becomes smaller than $\ve_0$.
\end{itemize}
We remark that, in view of (\ref{26}), jumps located at an endpoint of an interval 
$\bigl[ x_{\alpha}(t), x_{\alpha+1}(t)\bigr]$, where $u_\nu(t,\cdot)$ attains a local max or min,
have size that decreases in time.   It is quite possible that one of these jumps has strength
$$\bega{l}\bigl|\sigma_i(t)\bigr|~\geq \ve_0\quad\hbox{for}~~t\leq \tau_j\,,\\[1mm]
\bigl|\sigma_i(t)\bigr|~<\ve_0\quad\hbox{for}~~t>\tau_j\,.\enda $$
%

{\bf 3.} Given the constants $\ve,\delta>0$, we now describe how to construct  an approximate solution
to the viscous equation (\ref{30})
on a time interval $\,]t_{\ell-1}, t_\ell]$
where the front tracking solution retains a constant structure (same number of fronts,
same number of shocks of size $\geq\ve_0$).
We start by constructing three types of approximate solutions, on overlapping domains.
Recalling that all our solutions are spatially periodic, 
for $t\in \,]t_{\ell-1}, t_\ell[\,$ let 
$$x_1(t)~<~x_2(t)~<~\cdots~<~x_m(t)$$
be the set of points (within one period in space) which includes
all endpoints of intervals where $u_\nu(t,\cdot) $ has a local maximum or minimum, and all points  where $u_\nu(t,\cdot)$ has a jump of size $\geq \ve_0$.
\v

TYPE 1: Let $\bigl[ x_{k-1}(t), x_k(t)\bigr]$ be an interval where $u_\nu(t,\cdot)$ attains a local maximum, with constant value $u_\nu(t,x) = \Hat u_k(t)$.
Define 
\bel{T1a}
y_k\doteq \frac{x_{k-1}(t)+ x_k(t)}{ 2}\,,
\qquad  h_k\doteq \frac{x_{k}(t)- x_{k-1}(t)}{ 2}\,,
\qquad I_k \doteq \bigl[x_{k-1}(t), x_k(t) \bigr]\,.
\eeq
Here and below $y_k$ is the mid point of $I_k$. 
Recalling (\ref{Vhy}), we then define
\bel{wsha1}w^\sharp_k(t,x)~\doteq~ \Hat u_k(t)+V^{h_k,y_k}(x) ,\qquad\qquad
 x\in I_k\,. \eeq
In the case of a local minimum,  the definition (\ref{wsha1}) is replaced by
\bel{wsha2}w^\sharp_k(t,x)~\doteq~\Hat u_k(t)-V^{h_k,y_k}(x)  ,\qquad\qquad
 x\in I_k\,.\eeq

TYPE 2: This type includes neighborhoods of a big jump.     Let $y_k(t)$ be a point where $u_\nu(t,\cdot)$ has a large Liu-admissible
shock of size $\geq  \ve_0$,
say with left and right states $u_k^-(t), u_k^+(t)$. 
We then define 
\bel{wdag}
w^\dag(t,x)~\doteq~\Tilde U\bigl(t,x- y_k(t)\bigr),\qquad\qquad 
{ x\in I_k \doteq  \bigl[y_k(t)-2\delta^{1/3}, y_k(t)+2\delta^{1/3}\bigr]},\eeq
where $\Tilde U$ is the approximate viscous  traveling profile constructed at (\ref{mollU}), with left and right 
values
$u_k^-(t), u_k^+(t)$.  Note that these values may change in time, if $y_k(t)$ is the endpoint of an interval where $u_\nu(t,\cdot)$ attains a local max or min.

TYPE 3. We now consider the open set 
$$\Gamma~\doteq~\R\setminus\left( \bigcup_{\alpha\in \M} I_\alpha \cup\bigcup_{\ell\in \J} \{y_\ell\}\right),$$
where the first sum includes all closed intervals $I_\alpha$ where $u_\nu(t,\cdot)$ attains a local max or min,
and the second sum is over all points $y_\ell$ where $u_\nu(t,\cdot)$ has a jump of size $\geq \ve_0$.

Therefore, the set $\Gamma$  is the union of a set of disjoint open intervals which are connected 
 components of $\Gamma$. 
Consider such an open  interval $\,\bigl] x_{k-1}(t), x_k(t)\bigr[\,$, and define 
\bel{T3a}
y_k \doteq \frac12(x_{k-1}(t) + x_k(t) ), \qquad 
I_k \doteq \bigl[x_{k-1}(t)-\delta^{1/3} ,\,  x_k(t)+\delta^{1/3}\bigr]\,.
\eeq
On this interval $I_k$ we define a function $w_k^*(t,x)$ as follows.
We first define the boundary values
$$u^-_k(t)~\doteq~\left\{ \bega{rl} u_\nu\bigl(t, x_k(t)-\bigr)\qquad &\hbox{if $x_k(t)$ is the location of a jump of size $\geq\ve_0$,}\\[2mm]
u_\nu\bigl(t, x_k(t)+\bigr)\qquad &\hbox{if $x_k(t)$ is the location of a jump of size $<\ve_0$,}
\enda\right.$$
$$u^+_{k-1}(t)~\doteq~\left\{ \bega{rl} u_\nu\bigl(t, x_{k-1}(t)+\bigr)\qquad &\hbox{if $x_{k-1}(t)$ is the location of a jump of size $\geq\ve_0$,}\\[2mm]
u_\nu\bigl(t, x_{k-1}(t)-\bigr)\qquad &\hbox{if $x_{k-1}(t)$ is the location of a jump of size $<\ve_0$.}
\enda\right.$$
We then introduce the  truncated function
$$u^k_\nu(t,x)~\doteq~\left\{\bega{cl} u_\nu(t,x)\quad&\hbox{if} \quad x_{k-1}(t)<x< x_k(t),\\[2mm]
u_{k-1}^+\quad &\hbox{if} \quad x\leq x_{k-1}(t),\\[2mm]
u_k^-\quad&\hbox{if} \quad x\geq x_{k}(t).\enda\right.$$
Finally, we take a mollification.  As usual, this is done by 
introducing the standard mollifier
\bel{moll}\phi(x)~\doteq~\left\{\bega{cl} C_0\, \exp\left\{ \frac{1}{ x^2-1}\right\}\quad&\hbox{if} \quad |x|<1,\\[2mm]
0\quad&\hbox{if} \quad |x|<1,\enda\right.\qquad\qquad \phi_\delta(x)\doteq  \frac{1}{\delta}\, \phi\left(\frac{x}{\delta}\right),\eeq
where the constant $C_0$ is chosen so that $\int\phi(x)dx=1$.
We then set
\bel{wmoll}
w_k^*(t,x)~\doteq
~\int \phi_{\delta^{1/3}}(x-y) u^k_\nu(t,y)\, dy,\qquad x\in I_k.\eeq

\begin{figure}[htbp]
\begin{center}
\resizebox{0.9\textwidth}{!}{
\begin{picture}(0,0)%
\includegraphics{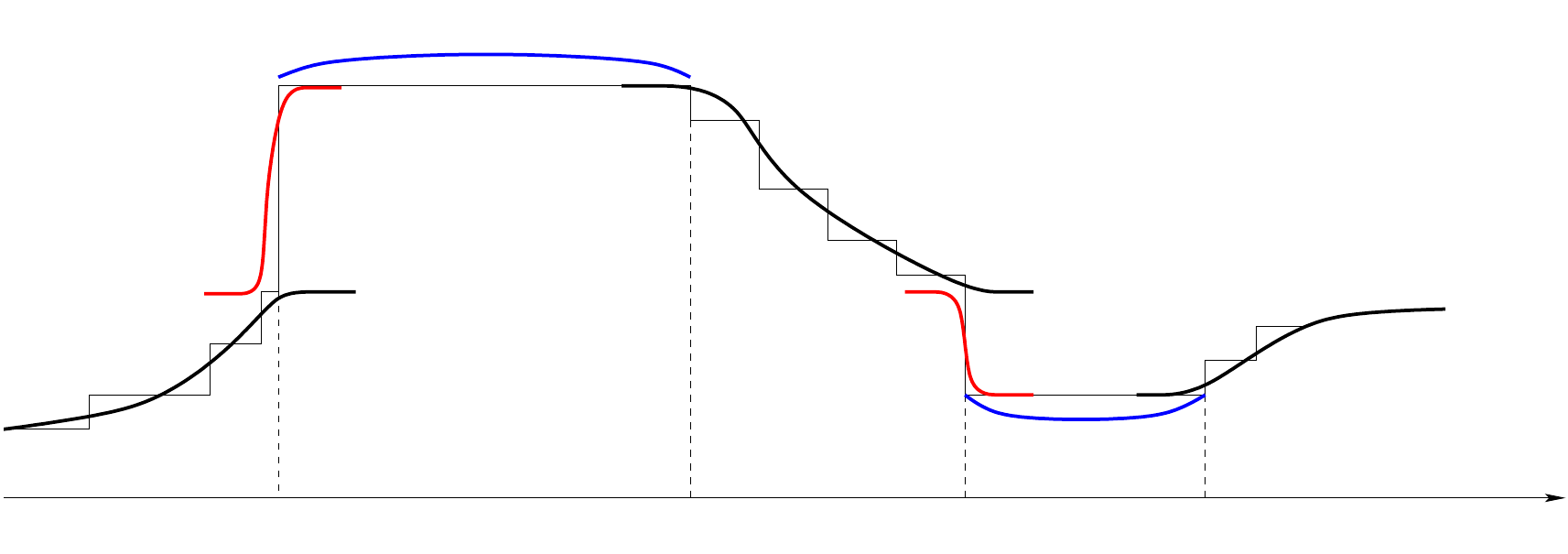}%
\end{picture}%
\setlength{\unitlength}{3947sp}%
\begin{picture}(13695,4825)(-2432,-244)
\put(1201,4289){\makebox(0,0)[lb]{\smash{\fontsize{22}{26.4}\usefont{T1}{ptm}{m}{n}{\color[rgb]{0,0,1}$w^\sharp_2$}%
}}}
\put(5826,-136){\makebox(0,0)[lb]{\smash{\fontsize{22}{26.4}\usefont{T1}{ptm}{m}{n}{\color[rgb]{0,0,0}$x_3$}%
}}}
\put(7926,-136){\makebox(0,0)[lb]{\smash{\fontsize{22}{26.4}\usefont{T1}{ptm}{m}{n}{\color[rgb]{0,0,0}$x_4$}%
}}}
\put(-174,-136){\makebox(0,0)[lb]{\smash{\fontsize{22}{26.4}\usefont{T1}{ptm}{m}{n}{\color[rgb]{0,0,0}$x_1$}%
}}}
\put(-1949,1439){\makebox(0,0)[lb]{\smash{\fontsize{22}{26.4}\usefont{T1}{ptm}{m}{n}{\color[rgb]{0,0,0}$w^*_1$}%
}}}
\put(6601,539){\makebox(0,0)[lb]{\smash{\fontsize{22}{26.4}\usefont{T1}{ptm}{m}{n}{\color[rgb]{0,0,1}$w^\sharp_4$}%
}}}
\put(4576,3164){\makebox(0,0)[lb]{\smash{\fontsize{22}{26.4}\usefont{T1}{ptm}{m}{n}{\color[rgb]{0,0,0}$w^*_3$}%
}}}
\put(9376,2114){\makebox(0,0)[lb]{\smash{\fontsize{22}{26.4}\usefont{T1}{ptm}{m}{n}{\color[rgb]{0,0,0}$w^*_5$}%
}}}
\put( 76,2864){\makebox(0,0)[lb]{\smash{\fontsize{22}{26.4}\usefont{T1}{ptm}{m}{n}{\color[rgb]{1,0,0}$w^\dagger_1$}%
}}}
\put(5401,1439){\makebox(0,0)[lb]{\smash{\fontsize{22}{26.4}\usefont{T1}{ptm}{m}{n}{\color[rgb]{1,0,0}$w^\dagger_3$}%
}}}
\put(3526,-136){\makebox(0,0)[lb]{\smash{\fontsize{22}{26.4}\usefont{T1}{ptm}{m}{n}{\color[rgb]{0,0,0}$x_2$}%
}}}
\end{picture}%
}
\caption{\small  The different types of approximations $w_k^*, w^\dagger_k, w_k^\sharp$
introduced at (\ref{wsha1})--(\ref{wmoll}),  on various subdomains.}
\label{f:df46}
\end{center}
\end{figure}
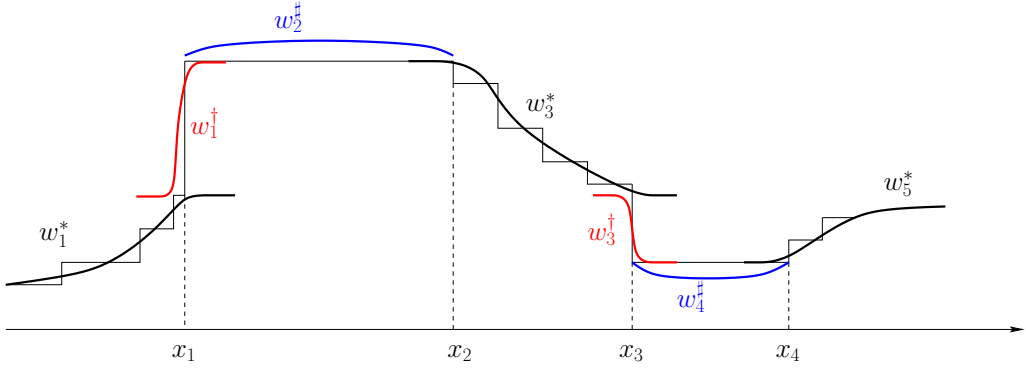

{\bf 4.} As shown in Fig.~\ref{f:df46}, the approximate viscous solutions $w^\sharp_k, w^\dagger_k, w_k^*$
are defined on overlapping domains.  To construct the $\C^{1,1}$ function $w(t,\cdot)$, we need to
glue together the above functions, so that the derivative $w_x(t,\cdot)$ is continuous at junctions.

Toward this goal, we observe that, if $y_k$ is a point where $u_\nu(t,\cdot)$
has a jump of size $\geq \ve_0$, then   we have the following matching conditions:
\bel{match}
\left\{    \bega{rl} u_\nu (y_k-) &=~w_k^\dagger(y_k-2\delta^{1/3}) 
~=~w^*_{k}(y_k+\delta^{1/3})
~=~w_{k}^\sharp (y_k) ,
\\[2mm]
u_\nu (y_k+) &=~w_k^\dagger(y_k+2\delta^{1/3}) 
 ~=~w^*_{k+1}(y_k-\delta^{1/3}) 
~=~w_{k+1}^\sharp (y_{k+1}),
\enda\right.
\eeq
whenever the above functions are defined.  
Similar identities hold at the common endpoints of intervals of Type 1 and 3.
For example, if $\bigl[ x_{k-1}(t), x_k(t)\bigr]$ is an interval of Type 1 (where
$u_\nu$ attains a local max or min) and  
$\,\bigl] x_{k}(t), x_{k+1}(t)\bigr[\,$ is an interval of Type 3 (where $u_\nu$ is monotone
but without jumps of size $\geq \ve_0$), then 
$$w^\sharp(x_k)~=~w^*\bigl(x_{k+1} - \delta^{1/3}\bigr).$$

Moreover, at all matching points, the functions $w_k^*$ and $w^\dagger_k$
have derivative zero, while $w^\sharp_k$ has derivative $\pm\ve$.

\begin{figure}[htbp]
\begin{center}
\resizebox{1\textwidth}{!}{
\begin{picture}(0,0)%
\includegraphics{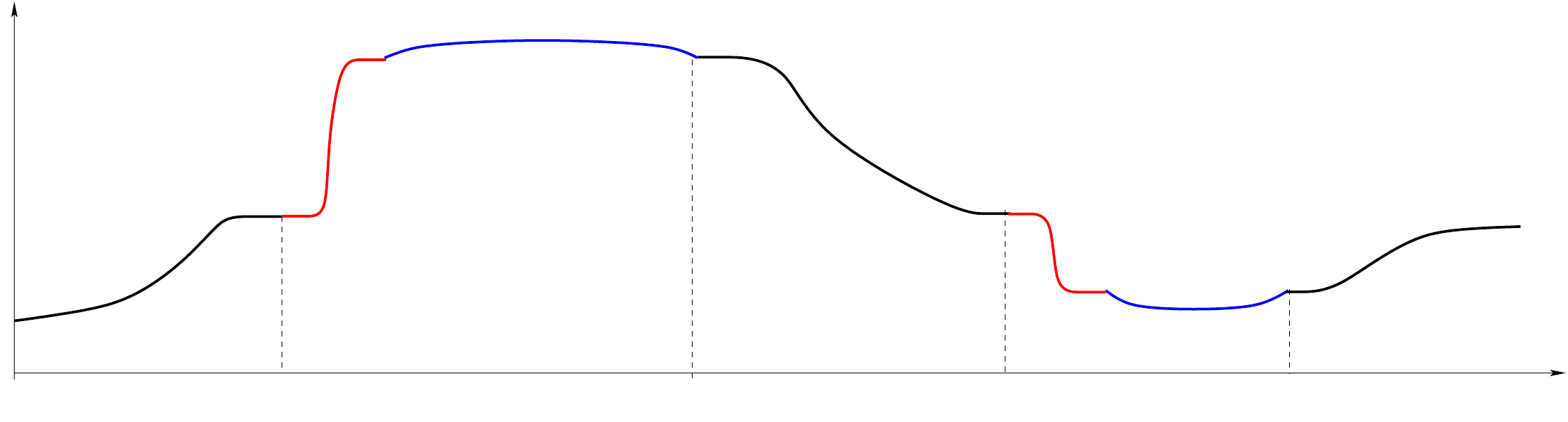}%
\end{picture}%
\setlength{\unitlength}{3947sp}%
\begin{picture}(18027,4845)(-2564,-319)
\put(9601,1514){\makebox(0,0)[lb]{\smash{\fontsize{22}{26.4}\usefont{T1}{ptm}{m}{n}{\color[rgb]{1,0,0}$w_3^\dagger$}%
}}}
\put(12076,-211){\makebox(0,0)[lb]{\smash{\fontsize{22}{26.4}\usefont{T1}{ptm}{m}{n}{\color[rgb]{0,0,0}$x_4'$}%
}}}
\put(5226,-211){\makebox(0,0)[lb]{\smash{\fontsize{22}{26.4}\usefont{T1}{ptm}{m}{n}{\color[rgb]{0,0,0}$x'_2$}%
}}}
\put(8801,-211){\makebox(0,0)[lb]{\smash{\fontsize{22}{26.4}\usefont{T1}{ptm}{m}{n}{\color[rgb]{0,0,0}$x_3'$}%
}}}
\put(476,-211){\makebox(0,0)[lb]{\smash{\fontsize{22}{26.4}\usefont{T1}{ptm}{m}{n}{\color[rgb]{0,0,0}$x_1'$}%
}}}
\put(14926,-116){\makebox(0,0)[lb]{\smash{\fontsize{22}{26.4}\usefont{T1}{ptm}{m}{n}{\color[rgb]{0,0,0}$x$}%
}}}
\put(10576,614){\makebox(0,0)[lb]{\smash{\fontsize{22}{26.4}\usefont{T1}{ptm}{m}{n}{\color[rgb]{0,0,1}$w_4^\sharp$}%
}}}
\put(13576,2039){\makebox(0,0)[lb]{\smash{\fontsize{22}{26.4}\usefont{T1}{ptm}{m}{n}{\color[rgb]{0,0,0}$w_5^*$}%
}}}
\put(7276,2864){\makebox(0,0)[lb]{\smash{\fontsize{22}{26.4}\usefont{T1}{ptm}{m}{n}{\color[rgb]{0,0,0}$w_3^*$}%
}}}
\put(1276,2639){\makebox(0,0)[lb]{\smash{\fontsize{22}{26.4}\usefont{T1}{ptm}{m}{n}{\color[rgb]{1,0,0}$w_1^\dagger$}%
}}}
\put(3076,4214){\makebox(0,0)[lb]{\smash{\fontsize{22}{26.4}\usefont{T1}{ptm}{m}{n}{\color[rgb]{0,0,1}$w_2^\sharp$}%
}}}
\put(-1724,1289){\makebox(0,0)[lb]{\smash{\fontsize{22}{26.4}\usefont{T1}{ptm}{m}{n}{\color[rgb]{0,0,0}$w_1^*$}%
}}}
\put(-2549,-211){\makebox(0,0)[lb]{\smash{\fontsize{22}{26.4}\usefont{T1}{ptm}{m}{n}{\color[rgb]{0,0,0}$0$}%
}}}
\end{picture}%
}
\caption{\small  By shifting the approximations $w_k, w^\dagger_k, w_k^\sharp$ shown in Fig.~\ref{f:df46},
thanks to the matching conditions (\ref{match}) we obtain a continuous function $W$.
}
\label{f:df47}
\end{center}
\end{figure}
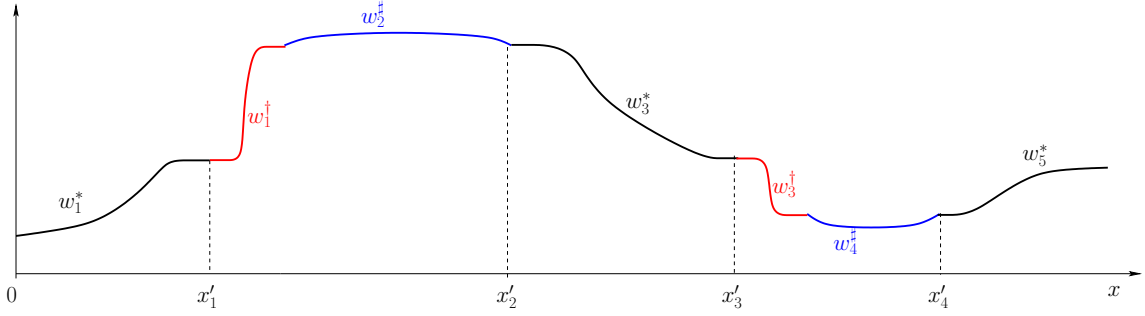

To combine these functions, we first construct an auxiliary function $W$ 
defined an a somewhat larger interval to accommodate for the overlapping regions.  
As shown in Fig.~\ref{f:df47}, we shift the graphs of all functions $w_k^\sharp, w_k^\dagger, w_k^*$
by integer multiples of $\delta^{1/3}$, so that the intervals where they are defined no longer overlap, and the matching conditions (\ref{match}) remain valid at the new endpoints 
$x_k'(t)$.

We define a precise way to make these shifts.
 Consider the overlapping intervals $I_k$ defined for TYPE 1, 2 and 3. 
We place them in the order of their mid points $y_k$, i.e.
$$I_0~\prec~I_1~\prec~I_2~\prec~\cdots~\prec~I_N\qquad\hbox{iff}\qquad y_0<y_1<y_2<\cdots<y_N\,.$$
Call $\xi_0$ the  left endpoint of the interval $I_0$, and set
\bel{elli}\ell_i\,\doteq\,\meas(I_i),\qquad\ell\doteq \sum_{i=0}^N \ell_i = 1 + m\delta^{1/3} \,,
\eeq
for some integer $m\ge 1$. 
For each $k=0,1,\ldots, N$, let  $W_k: I_k\mapsto \R$ be a function defined on $I_k$.
A function $W:[\xi_0, ~\xi_0+\ell]\mapsto \R$ is then defined as follows.
\bel{shiftW} 
\qquad W(x)~=~W_k\left( x-\xi_0-{ \sum_{j=1}^k \ell_j}\right)
\qquad \hbox{whenever} \quad
 \sum_{j=1}^k \ell_j ~\leq~x-\xi_0~<~{ \sum_{j=1}^{k+1} \ell_j} \,.
 \eeq
For each time $t$, applying the above construction in connection with the functions $w_k^\sharp, w_k^\dagger, w_k^*$,
defined on their respective intervals as in (\ref{wsha1}--(\ref{wmoll}),
we obtain a function $W(t,\cdot)$ defined on the interval $\bigl[\xi_0(t)\,,~ \xi_0(t)+ \ell \bigr]$.

{\bf 5.}
As shown in Fig.~\ref{f:df47}, three issues remain to be addressed:
\begi
\item[(i)]  The functions $w_k^\dagger, w_k^*$ have zero derivative at the endpoints
of the intervals where they are defined, while the functions $ w_k^\sharp$ have derivative 
$\pm\ve$.
For this reason the derivative $W_x(t,\cdot)$ can be discontinuous at some of the junction points $x_j'$. This must be avoided.
\item[(ii)] On the regions where $W= w_k^\dagger$ or $W=w_k^*$, we would like to have
a derivative $W_x(t,x)$  either $\geq \ve$ or $\leq -\ve$.   This will imply
$\theta_\ve(W_x)\in \{0,1\}$, $\theta_\ve'(W_x)=0$, greatly simplifying the estimate 
(\ref{wer}). 
\item[(iii)] The function $W(t,\cdot) $ can be periodically extended to a spatially periodic function 
with period $\ell >1$, while we would like it to have period 1.
\endi

To cope the above difficulties,  we partition the domain of $W(t,\cdot) $ as
$$\bigl[\xi_0\,,~\xi_0+ \ell 
\bigr]~=~\Omega_+(t)\cup \Omega_-(t)\cup \Omega_0(t),$$
where $\Omega_+$ corresponds to the intervals where $W$ is increasing, $\Omega_-$
to the intervals where $W$ is decreasing, and $\Omega_0$ to the intervals where $W$
attains a local max or min (so that $W = w_k^\sharp$ for some $k$).
We then define
\bel{TWdef}\Tilde W(t,x)~\doteq~W(t,x)  + \ve \cdot \meas\Big(\Omega_+(t)\cap [\xi_0,x]\Big) - \
\ve \cdot \meas\Big(\Omega_-(t)\cap [\xi_0,x]\Big).\eeq
Note that this modification solves both issues (i) and (ii) at the same time.   Indeed, (\ref{TWdef}) implies that 
$\Tilde W_x\leq -\ve$ on the intervals where
$\Tilde W$ is decreasing, while 
$\Tilde W_x\geq \ve$ on intervals where $\Tilde W$ is increasing.
Moreover, at the endpoints of the intervals where $\Tilde W$ attains a local max or min,
its derivative exists and is continuous, since $\Tilde W_x =\ve$  or $\Tilde W_x =-\ve$.

Unfortunately, the definition (\ref{TWdef}) 
creates another problem.  Namely, 
if $\meas\bigl(\Omega_+(t)\bigr)\not= \meas\bigl(\Omega_-(t)\bigr)$, the identity
\bel{TWper}\Tilde W(t, \,\xi_0 + \ell 
)~=~\Tilde W(\xi_0),\eeq
will fail, 
hence $\Tilde W(t,\cdot)$ cannot be extended to a continuous periodic function.

To fix ideas, assume that the excess length
$$E(t)~\doteq~\meas\bigl(\Omega_+(t)\bigr)- \meas\bigl(\Omega_-(t)\bigr)~>~0,$$
the case $E(t)<0$ being entirely similar.  

To resolve the problem, we consider a slightly modified 
(no longer periodic) piecewise constant  function 
$v_\nu(t,\cdot)$, obtained from $u_\nu(t,\cdot)$ by slightly increasing each of its 
downward jumps.   Repeating all steps of the above construction in connection 
with $v_\nu(t,\cdot)$, we shall obtain a function $\Tilde W$ which  satisfies the periodicity requirement.

More precisely, let 
$$\xi_0(t)~<~ x_1(t)~<~x_2(t)~<~~\cdots~~<~ x_N(t) ~<~ \xi_0(t)+1$$
be the points where the jumps in $u_\nu(t)$ are located, and 
call 
$$\sigma_j(t)~\doteq~u_\nu\bigl( t, x_j(t)+\bigr) - u_\nu\bigl( t, x_j(t)-\bigr) $$
the jump at $x_j(t)$. 
By construction, the total variation of $u_\nu$ over one period is
$$\TV\Big\{ u_\nu(t,\cdot)\,;~\bigl[\xi_0(t), \xi_0(t)+1\bigr[\,\Big\} =2\mu(t),
\qquad
\mu(t)\doteq \sum_{\sigma_j>0} \sigma_j(t) =-\sum_{\sigma_j<0} \sigma_j(t).
$$ 
For $x\in \bigl[ \xi_0(t), \xi_0(t)+1\bigr]$, we now define
the modified function $v_\nu(t,\cdot)$ by setting
\bel{vnu}
v_\nu(t, x)~\doteq~u_\nu\bigl(t, \xi_0(t)+\bigr) + \sum_{\xi_0(t)<x_j(t)\leq x,~
\sigma_j>0}\sigma_j(t) +  \sum_{\xi_0(t)<x_j(t)\leq x,~
\sigma_j<0}\frac{\ve\, E(t)}{ \mu(t)} \sigma_j(t) \,.\eeq
In other words, $v_\nu$ is obtained from $u_\nu$ by multiplying the size of each downward jump
by the same factor $\ve E(t)/\mu(t)$.   
Thus, the  periodicity relation (\ref{TWper}) is now satisfied. 


Finally, to achieve a spatially periodic function with period 1, it suffices to define
\bel{wdefinal}
w(t,x)~\doteq~\Tilde W\left(t,  x/\ell
\right).\eeq
Since $\ell>1$, this rescaling increases the absolute value of all derivatives.
In particular, we have the implications
\bel{derbig}\bega{c}\ds\Tilde W_x\left(t, x/\ell
\right)\leq-\ve\qquad\implies\quad w_x(t,x)\leq -\ve,\\[2mm]\ds
\Tilde W_x\left(t, x/\ell
\right)\geq\ve\qquad\implies\quad w_x(t,x)\geq \ve.
\enda\eeq

\begin{remark} {\rm Approximate viscous solutions to a hyperbolic system of
conservation laws, close to a front tracking approximation $u_\nu$,
were first  constructed in \cite{BY} using a similar technique.
The main difference is that in \cite{BY} the viscous shock profiles $w_k^\dagger$
(tracing large shocks) were patched together with the mollified profiles $w_k^*$
simply by a smooth interpolation.

In the present setting, we need to  keep track of the intervals where 
the gradient satisfies $|w_x|\geq \ve$, hence $\theta_\ve '(w_x)=0$,
and  the remaining intervals around local maxima or minima, where
 $\theta_\ve'(w_x)>0$.   A more careful construction is thus needed.
}
\end{remark}

\section{Error estimates}
\label{sec:7}
\setcounter{equation}{0}

In this section we prove that the approximations constructed in Section~\ref{sec:6}
satisfy the conclusion of  Proposition~\ref{p:41}.

\begin{proof} (of Proposition~\ref{p:41})
Given $\ve_0>0$ and a piecewise constant initial data $\bar u\in \L^1_{per}(\R)$, 
by choosing $\nu\ge 1$ large enough, it is clear that \eqref{SSnu} is  satisfied.

Replacing $u_\nu(t,\cdot) = S^\nu_t \bar u$ with the piecewise constant function 
$v_\nu(t,\cdot)$ defined at (\ref{vnu}), for every $t$ the difference is estimated as
\bel{uvdif}
\int_0^1 \bigl| v_\nu(t,x) - u_\nu(t,x)\bigr|\, dx~=~\O(1)\cdot\ve.\eeq
This can be rendered arbitrarily small by choosing $\ve>0$ small enough. 
Starting with $v_\nu(t,\cdot)$ and following the various steps in the construction of $w(t,\cdot)$, described in Section~\ref{sec:6}, one readily checks that 
\bel{vwdif}
\int_0^1 \bigl| w(t,x) - v_\nu(t,x)\bigr|\, dx~=~\O(1)\cdot(\ve+\delta^{1/3}).\eeq
Again, this can be rendered arbitrarily small by choosing $\ve,\delta>0$ small enough. 
As a consequence, for all $t\in [0,T]$ we obtain the bound
\bel{wunudif}
\int_0^1 \bigl| w(t,x) - u_\nu(t,x)\bigr|\, dx~=~\O(1)\cdot(\ve+\delta^{1/3})~\leq~\ve_0\,.\eeq
This already  yields (\ref{wu1}).  

Next, let $0=t_0<t_1<\cdots<t_p\leq T$ be the times where some front interaction occurs 
in the front tracking solution $u_\nu$, or the size of one of the jumps becomes smaller than $\ve_0$.   These are the only times when the map $t\mapsto
w(t,\cdot)$ can be discontinuous, as a map with values in $\L^1_{per}(\R)$. 
  At each time $t_j$ where the map $t\mapsto w(t,\cdot)$ has a jump, by (\ref{wunudif}) it follows
\bel{jumpk}
\int_0^1 \bigl| w(t_j+,x) - w(t_j-,x)\bigr|\, dx~=~\O(1)\cdot(\ve+\delta^{1/3})\,.\eeq
We observe that the total  number of wave front interactions  in $u_\nu$  
depends on the initial condition $\bar u$, on $\ve_0$ and on the integer $\nu\geq 1$, but is independent of $\ve,\delta$.
By choosing $\ve,\delta>0$ small enough in the construction of $w$,  we  thus obtain (\ref{wu2}).
\v
It now remains to establish the bound (\ref{wer}), estimating by how much the function
$w$ fails to satisfy the parabolic equation  (\ref{3}) during time intervals where 
the map $t\mapsto w(t,\cdot)$ is continuous.

To begin, we observe that the presence of shifts as in (\ref{shiftW}), the 
$\ve$-perturbation  at (\ref{TWdef}) 
and the rescaling in (\ref{wdefinal}), all  contribute to the instantaneous error 
\bel{IE}\E(t)~\doteq~\int
\bigg| 
w_t +  \Big[\theta_\ve(w_x) f(w) + \bigl(1-\theta_\ve(w_x) \bigr)g(w)\Big]_x
-\delta w_{xx}\bigg|\, dx
\eeq
by an amount which approaches zero as $\ve, \delta\to 0$.  As usual, since we are dealing with periodic functions, it is understood that the integral ranges over one spatial period.

As an example, we give here a bound on the above integral relative to an interval  of Type~1, where
$u_\nu$ attains a local maximum.
To fix ideas, recalling (\ref{Vhy}) let 
$$w(x) ~=~V^{0,h}(x)~\doteq~\ve h\int_{-1}^{x/h} Z(s)\, ds\qquad\qquad x\in [-h,h].$$
Here $Z$ is the function introduced at (\ref{Zeq})-(\ref{Zprop}).   
By construction, $w$ satisfies the ODEs
\bel{thw}\theta'_\ve(w'(x)) w''(x) ~=~{-1\over 2h},
\qquad\qquad x\in [-h,h].\eeq
Next, for a given $\ell = 1+ m\delta^{1/3}$ as in (\ref{elli}),
consider the rescaled function
$$\Tilde w(x)~\doteq~w(\ell x),\qquad\qquad x\in \left[ -{h\over \ell}, ~{h\over \ell}\right].$$ 
We need to estimate by how much $\Tilde w$ fails to satisfy the equation (\ref{thw}).
In other words, we seek a bound on
\bel{errZ}\bega{rl} E&\ds\doteq~\int_{-h/\ell}^{h/\ell} \left| \theta'_\ve\bigl(\Tilde w'(x)\bigr) \Tilde w''(x) + {\ell\over 2h}
\right|\, dx\\[4mm]
&\ds =~{1\over \ell}\cdot  \int_{-h}^{h} \left| \theta'_\ve\bigl(\ell w'(\ell x)\bigr) \,\ell^2 w''(\ell x)
-  \theta'_\ve\bigl( w'(x)\bigr) w''(x) + {\ell-1\over 2h}
\right|\,dx \\[4mm]
&\leq  \ds ~{\ell-1\over \ell} + {1\over \ell}\cdot  \int_{-h}^{h} \Big| \theta'_\ve \bigl(\ell \ve  Z(\ell x)\bigr) \,\ell^2\ve Z'(\ell x)
-  \theta'_\ve\bigl( \ve Z(x)\bigr) \ve Z'(x)
\Big|\,dx \\[4mm]
&\leq  \ds ~{\ell-1\over \ell} + {1\over \ell}\cdot  \int_{-h}^{h} \Big| \theta' \bigl(\ell   Z(\ell x)\bigr) \,\ell^2 Z'(\ell x)
-  \theta'\bigl(  Z(x)\bigr)  Z'(x)
\Big|\,dx \,.
\enda
\eeq
As $\delta\to 0$, one has $\ell = 1+m \delta^{1/3}\to 1$. 
It is thus clear that the right hand side of (\ref{errZ}), which is independent of $\ve$, approaches zero.

The estimates relative to the intervals of Type 2 and 3,
are straightforward.  In all cases, the additional errors coming from the 
shifts, $\ve$-perturbations and rescalings at (\ref{shiftW})--(\ref{wdefinal}), approach zero as $\ve,\delta\to 0$.
\v
In the following, we focus on the error on the left hand side of (\ref{wer}) in the cases where
$w$ is replaced by $w_k^\sharp, w^\dagger_k$ or $w^*_k$, defined at (\ref{wsha1})-(\ref{wsha2}), 
at (\ref{wdag}), or at (\ref{wmoll}), respectively.
\v

{\bf 1 - Intervals where a local maximum or a local minimum is attained.}

The error over an interval  
$\bigl[x_k(t)\,,~ x_{k+1}(t)]\bigr]$ where $u_\nu$ attains a local max
or min can be estimated using Lemma~\ref{l:56}.
Namely,
\bel{ER1}\int_{x_k(t)}^{x_{k+1}(t)}
\bigg| 
w^\sharp_t +  \Big[\theta_\ve(w^\sharp_x) f(w^\sharp) + \bigl(1-\theta_\ve(w^\sharp_x) \bigr)g(w^\sharp)\Big]_x
-\delta w^\sharp_{xx}\bigg|\, dx~=~  \O(1)\cdot \ve .
\eeq
Indeed, by construction we have
\bel{ER4}\frac{d}{ dt}\Hat u_k(t) +  \Big[\theta_\ve(w^\sharp) f\bigl(\Hat u_k(t)\bigr) + \bigl(1-\theta_\ve(w^\sharp_x) \bigr)g\bigl(\Hat u_k(t)\bigr)\Big]_x~=~0.\eeq
The difference between (\ref{ER4}) and the integrand in (\ref{ER1}) is thus estimated by the sum of 
the following three terms:
 \bel{ER5}\int_{x_k(t)}^{x_{k+1}(t)}\delta \,|w^\sharp_{xx}|\,dx~=~\delta\,\Big|w^\sharp_x\bigl(t,x_{k+1}(t)\bigr) - w^\sharp_x \bigl(t, x_k(t)\bigr)\Big| ~=~2 \delta\ve\,,\eeq
\bel{ER6}\bega{l} \ds \int_{x_k(t)}^{x_{k+1}(t)}
\bigg| 
\Big[\theta_\ve(w^\sharp_x) f(w^\sharp) + \bigl(1-\theta_\ve(w^\sharp_x) \bigr)g(w^\sharp)\Big]_x
\\[4mm]
\qquad\qquad \ds-
 \Big[\theta_\ve(w^\sharp) f\bigl(\Hat u_k(t)\bigr) + \bigl(1-\theta_\ve(w^\sharp_x) \bigr)g\bigl(\Hat u_k(t)\bigr)\Big]_x \bigg| dx~=~\O(1)\cdot \ve. \enda\eeq
\bel{ER7} \int_{x_k(t)}^{x_{k+1}(t)} \left| \frac{\partial}{\partial t} V^{y(t), h(t)}(t,x)\right|\, dx
~=~\O(1)\cdot \ve\,.\eeq
This leads to \eqref{ER1}.
\v
 {\bf 2 - Intervals around big jumps of size $\geq  \ve_0$.}
 
 To fix ideas, consider an upward, Liu-admissible shock in $u_\nu$, of size $\geq \ve_0$, located at 
 the point $x_k(t)$.
 Consider $w^\dagger = \Tilde U$.
 If the left and right states $u^\pm(t) = u\bigl(t, x_k(t) \pm\bigr)$ remain constant in time, from (\ref{erb7}) it follows
 \bel{ER2}\int_{x_k(t)-\delta^{1/3}}^{x_{k}(t)+\delta^{1/3}}
\bigg| 
w^\dagger_t +  f(w^\dagger)_x
-\delta w^\dagger_{xx}\bigg|\, dx~=~\O(1)\cdot \ve_0\, |u_k^+-u_k^-|.\eeq
 If one or both states $u^-,u^+$ vary in time
 (this happens if the shock coincides with the endpoint of an interval where $u_\nu$
 attains a local max or min), this produces an additional error that must be estimated.
 
 We recall that, by (\ref{HUC}), the viscous profile $\Tilde U$ depends
 Lipschitz continuously on the left and right states $u^-, u^+$, in the $\L^1$ distance.
At time $t$,
this additional error can thus be bounded by
\bel{aer}\sum_{k\in {\cal J}(t)} C_5\cdot \delta^{1/3} \left( \Big| \frac{d}{ dt} u_k^-(t)\Big|+ \Big| \frac{d}{ dt} u_k^+(t)\Big|
\right),\eeq
where the summation ranges over the set of all upward jumps in $u_\nu(t,\cdot)$ with size $\geq\ve_0$.  Since we are always  dealing with spatially periodic functions, it is understood
that the summation is over one period.
As soon as $u_\nu$ is given, by choosing $\delta>0$ sufficiently small
we can render the error in (\ref{aer}) as small as we like.  
\v
{\bf 3. Remaining domain, containing only small jumps. }
Consider an open interval $\bigl]x_k(t), x_{k+1}(t)\bigr[$ where
$u_\nu(t,\cdot)$ does not attain a local max or min, and does not have any jumps of size $\geq\ve_0$.

To derive an error estimate  on the function $w^*= w^*_k(t,x)$, assume that, on this interval,
$u_\nu$ solves the conservation law with flux $f_\nu$.
Moreover, assume: 
\begi
\item[{$\bf (A_\delta)$}]
{\it At time $t$, any two jumps in $u_\nu(t,\cdot)$ have distance $\geq 2\delta^{1/3}$ from each other. }
\endi
Performing the mollification  (\ref{wmoll}),  this implies
\bel{wunc}\bigl| w_k^*(t,x)- u_\nu(t,x)\bigr|~\leq~\ve_0\qquad\quad\forall  ~x \in \bigl]x_k(t), x_{k+1}(t)\bigr[\,.\eeq

Call $\xi_i(t)\in \bigl[x_k(t), x_{k+1}(t)\bigr]$ the locations of the jumps in 
$u_\nu(t,\cdot)$ with size
$$\sigma_i\,\doteq\,u_\nu\bigl(t,\xi_i(t)+\bigr)-u_\nu\bigl(t,\xi_i(t)+\bigr)~<~\ve_0\,.$$ 
If $u_\nu$ does not have  any jump located inside the interval $\bigl[x-\delta^{1/3}, \, x+\delta^{1/3}\bigr]$, then trivially
$$w(t,x)\,=\,u_\nu(t,x),\qquad w_t(t,x)\,=\, w_x(t,x)\,=\,0.$$
Next, assume that  $u_\nu$ has a jump located at 
$\xi_i(t)\in \bigl[x-\delta^{1/3}, \, x+\delta^{1/3}\bigr]$. Notice that at most one jump
can be contained in this interval, because of $\bf (A_\delta)$.
Since this jump has size  $\sigma_i < \ve_0$, 
we have
\bel{f'w}
f'\bigl(w^*(t,x)\bigr) - \dot \xi_i(t)~=~\O(1)\cdot\ve_0\,.\eeq
Next, differentiating (\ref{wmoll}) we compute
$$w^*_x(t,x)~=~ \phi_{\delta^{1/3}}(x-x_i)\cdot \sigma_i\,,\qquad\qquad 
w^*_t(t,x)~=~- \phi_{\delta^{1/3}}(x-x_i)\cdot \sigma_i\dot\xi_i$$
By (\ref{f'w}), this yields 
\begin{align}
&\int \Big| w^*_t (t,x)+  f'\bigl(w^*(t,x)\bigr) w^*_x(t,x)\Big| \,dx 
\nonumber \\
& =~\int \bigg| -\sum_{|x_i-x|<\delta^{1/3}} \phi_{\delta^{1/3}}(x-x_i)\cdot \sigma_i\dot\xi_i
+\sum_{|x_i-x|<\delta^{1/3}} \phi_{\delta^{1/3}}(x-x_i)\cdot \sigma_i f'\bigl(w^*(t,x)\bigr)
\bigg| \,dx
\nonumber \\
& =~\O(1)\cdot\ve_0\cdot \TV\Big\{ u_\nu(t,\cdot)\,;~ \bigl]x_k(t), x_{k+1}(t)\bigr[\,\Big\}.
\label{errb33}
\end{align}

In addition, one has
\bel{errb6}\delta \, w_{xx}~=~\delta \cdot\bigl((\phi_{\delta^{1/3}})_{xx} * u_\nu\bigr)~=~\delta\cdot \O(1)\cdot \delta^{-2/3}
~=~\O(1)\cdot \delta^{1/3}.\eeq
Combining the previous bounds, for all $\delta>0$ small enough, we obtain
\bel{errb3}\sum_k
\int_{x_k(t)}^{x_{k+1}(t)} \Big|
w^*_t +  f(w^*)_x 
-\delta w^*_{xx}\Big|\, dx~ =~\O(1)\cdot\ve_0\cdot \TV\bigl\{ u_\nu(t,\cdot)\bigr\}\,.
\eeq
We recall that, since we are considering spatially periodic functions, all the above estimates are computed over one period.

It remains to examine the contribution to the error case where {$\bf (A_\delta)$} does not hold,
and the last bound in (\ref{errb33}) fails. 
For this purpose, consider the set of ``bad" times
$$
{\cal B}^\delta~\doteq~\Big\{ t\in [0,T]\,;~~u_\nu(t,\cdot)~\hbox{contains two fronts at distance $< 2\delta^{1/3}$ from each other}\Big\}.
$$
We trivially have
\bel{mbad}\lim_{\delta\to 0+}~
\meas({\cal B}^\delta)~=~0.\eeq
Therefore, by choosing $\delta>0$ small enough, the set ${\cal B}^\delta$ 
of times  where {$\bf (A_\delta)$} 
does not hold can be rendered as small as we like.
In particular, we can assume
\bel{Bbad} \int_{{\cal B}^\delta} \sum_{k} \left(\int_{x_k(t)}^{x_{k+1}(t)} \Big|
(w^*_k)_t +  f(w^*_k)_x 
-\delta (w^*_k)_{xx}\Big|\, dx\right)dt ~<~\ve_0\,.\eeq
Here the summation ranges over all intervals (within one period) where
$u_\nu(t,\cdot)$ does not attain a local max or min, and does not have any jump of size $\geq\ve_0$.
\v
{\bf 4.} Summarizing all previous arguments:
\begi
\item At the beginning, a piecewise constant initial data $\bar u\in \L^1_{per}(\R)$ is given, together with $\ve_0>0$.
\item We then choose the integer $\nu\ge 1$ and construct the front tracking approximation
$u_\nu(t,\cdot)=S^\nu_t\bar u$ so that  (\ref{SSnu}) holds.
\item Finally, we choose $\ve,\delta>0$ and construct an approximate viscous solution $w$
such that 
$$\int_0^T \int 
\bigg| 
w_t +  \Big[\theta_\ve(w_x) f(w) + \bigl(1-\theta_\ve(w_x) \bigr)g(w)\Big]_x
-\delta w_{xx}\bigg|\, dx\, dt~\leq ~C \ve_0\, ,$$
where $C$ is a constant depending on $f,g$ and the initial data $\bar u$, but not on $\ve_0$.
\endi
Since $\ve_0>0$ can be taken arbitrarily small, this yields a proof of Proposition~\ref{p:41}.
\end{proof}

\section{The non-periodic case}
\label{sec:8}
\setcounter{equation}{0}

If $\bar u\in \L^1(\R)$, without any periodicity assumption, by finite propagation speed
we can approximate $\bar u$ with a piecewise constant initial data $\bar v$, say supported
inside $[-R,R]$. Given a time interval $[0,T]$, we can extend $\bar v$ to a periodic function:
$$\bar v(x+p) ~=~\bar v(x),$$
choosing the period $p$ so large that 
\bel{pbig}T\cdot \max~\Big\{ \bigl|f'(u)\bigr|, ~\bigl|g'(\omega)\bigr|\,;~~|\omega|\leq \|\bar v\|_{\L^\infty}\Big\}
~\leq~p-2R.\eeq
This guarantees that solutions with spatially periodic data
coincide with the original one, for $t\in [0,T]$ and $|x|\leq R+(p/2)$.
The previous result shows that these solutions can be obtained as the limits of  solutions
to the parabolic equation (\ref{3})
in $\L^1_{loc}(\R)$, letting $\ve,\delta\to 0+$.

\begin{remark}\label{r:31}{\rm 
It is important to notice that in general  the limit in (\ref{HatS}) cannot hold in $\L^1$ but only in $\L^1_{loc}$.  
For example, consider initial data $\bar u(x)\geq 0$, not identically zero.
Since (\ref{30}) is in conservation form, it follows
$$\| S^{\ve,\delta}_t \bar u\|_{\L^1} ~=~\|\bar u\|_{\L^1}\qquad \forall t\geq 0.$$
On the other hand, by the strict inequality $f(u)<g(u)$ for all $u\in\R$,  the limit
of front tracking approximations satisfies
$$\| S_t \bar u\|_{\L^1} ~<~\|\bar u\|_{\L^1}\qquad \forall t> 0.$$
This apparent paradox can be explained observing that in the viscous
semigroups $S^{\ve,\delta}$ the total mass is conserved, but there is a huge diffusion when $u_x\approx 0$.
As a result,
part of the mass is pushed toward $\pm\infty$.   
In the limit as $\ve\to 0$, some mass leaks to 
$\pm\infty$ in finite time and disappears.
}
\end{remark}

\section{Concluding Remarks}
\label{sec:9}
\setcounter{equation}{0}

In this paper we study the stable case 
where $f(u) < g(u)$ for all $u\in\R$,  for general flux functions $f$ and $g$. 
The unstable case where $f(u) > g(u)$ is treated in the companion paper \cite{ABSu}, 
for strictly convex flux functions. 
Naturally, the next topic is the case where the graphs of $f$ and $g$ intersect at one or more points,
dividing the solution domain into stable and unstable regions for different values of $u$. 
Such cases are particularly relevant 
in connection to traffic flow models with hysteresis, where  the flux functions are concave. 
Recorded traffic data from highways \cite{TM_data,CF} indicates that $f<g$ 
for low density traffic, and $f>g$ for high density traffic. 
Our results indicates that the model \eqref{1} is stable for low density traffic and unstable for high density traffic.
As shown in Section~\ref{sec:2} in \cite{ABSu},  in the unstable region 
an arbitrarily small random perturbation will trigger new spikes in the solution.
Once formed, such spikes would persist in time, possibly leading to the  formation of stop-and-go waves.
Such  phenomena would be interesting future research topics. 

\v
{\small
{\bf Acknowledgment.} The research by A.~Bressan was partially supported by NSF with
grant  DMS-2306926, ``Regularity and approximation of solutions to conservation laws". 
The research by D.~Amadori was partially supported by the Ministry of University and Research (MUR), Italy under the grant PRIN 2020 - Project N. 20204NT8W4, ``Nonlinear evolution PDEs, fluid dynamics and transport equations: theoretical foundations and applications" and by the INdAM-GNAMPA Project 2023, CUP E53C22001930001, ``Equazioni iperboliche e applicazioni".
D.~Amadori acknowledges the kind hospitality of Penn State University, where this research started.
}


\end{document}